\documentclass[reqno,11pt]{amsart}

\usepackage{amscd}
\usepackage{amsmath}
\usepackage[new]{old-arrows}
\usepackage{amssymb}
\usepackage[american]{babel}
\usepackage{bbm}
\usepackage{bookmark}
\usepackage{cmap} 
\usepackage{dsfont}
\usepackage{enumerate}
\usepackage{epigraph}
\usepackage[mathscr]{euscript}
\usepackage[myheadings]{fullpage}
\usepackage{graphicx}
\usepackage[geometry]{ifsym}
\usepackage{accents}
\usepackage{mathrsfs}
\usepackage{mathtools}
\usepackage{enumitem}
\usepackage{refcount}
\usepackage{cleveref}
\usepackage{setspace}
\usepackage{stmaryrd}
\usepackage{verbatim}
\usepackage[all,2cell]{xy} \UseTwocells
\usepackage{xr}
\usepackage[multiple]{footmisc}
\usepackage{hyperref}

\usepackage{tikz}
\usetikzlibrary{decorations.pathmorphing}
\usetikzlibrary{cd}

\numberwithin{equation}{subsection}

\newtheorem{thm}{Theorem}[subsubsection]
\newtheorem*{thm*}{Theorem}

\newtheorem{cor}[thm]{Corollary}
\newtheorem*{cor*}{Corollary}
\newtheorem{lem}[thm]{Lemma}

\newtheorem{prop}[thm]{Proposition}
\newtheorem{prop-const}[thm]{Proposition-Construction}

\newtheorem{conjecture}{Conjecture}[subsection]
\newtheorem*{conjecture*}{Conjecture}

\newtheorem*{princ*}{Principle}

\theoremstyle{remark}
\newtheorem{rem}[thm]{Remark}

\newtheorem{defin}[thm]{Definition}

\newtheorem{construction}[thm]{Construction}

\newcounter{steps}[thm]
\newtheorem{step}[steps]{Step}

\newcommand{\into}{\hookrightarrow}

\newcommand{\bB}{{\mathbb B}}
\newcommand{\bC}{{\mathbb C}}

\newcommand{\bG}{{\mathbb G}}

\newcommand{\bQ}{{\mathbb Q}}
\newcommand{\bR}{{\mathbb R}}

\newcommand{\bZ}{{\mathbb Z}}

\newcommand{\cC}{{\mathcal C}}

\newcommand{\cF}{{\mathcal F}}
\newcommand{\cG}{{\mathcal G}}

\newcommand{\cL}{{\mathcal L}}

\newcommand{\cO}{{\mathcal O}}
\newcommand{\cP}{{\mathcal P}}

\newcommand{\cW}{{\mathcal W}}

\newcommand{\cZ}{{\mathcal Z}}

\newcommand{\sC}{{\EuScript C}}

\newcommand{\sH}{{\EuScript H}}

\newcommand{\sL}{{\EuScript L}}

\newcommand{\sP}{{\EuScript P}}

\newcommand{\sS}{{\EuScript S}}

\newcommand{\sU}{{\EuScript U}}

\newcommand{\sY}{{\EuScript Y}}

\newcommand{\fM}{{\mathfrak M}}

\newcommand{\fX}{{\mathfrak X}}

\newcommand{\fg}{{\mathfrak g}}

\newcommand{\fm}{{\mathfrak m}}
\newcommand{\fn}{{\mathfrak n}}

\newcommand{\fp}{{\mathfrak p}}

\newcommand{\on}{\operatorname}

\newcommand{\Spec}{\on{Spec}}

\newcommand{\Bun}{\on{Bun}}

\newcommand{\colim}{\on{colim}}

\renewcommand{\lim}{\on{lim}}

\renewcommand{\subset}{\subseteq}

\makeatletter
\newcommand{\biggg}{\bBigg@{4}}
\newcommand{\Biggg}{\bBigg@{5}}
\makeatother

\begin{document}

\frenchspacing

\setlength{\epigraphwidth}{0.4\textwidth}
\renewcommand{\epigraphsize}{\footnotesize}

\title{Local Constancy of the Betti Automorphic Category Over The Moduli of Curves}

\author{Joakim F\ae rgeman and Marius Kj\ae rsgaard}

\address{Institute for Advanced Study, 
School of Mathematics, 
1 Einstein Dr, 
Princeton, NJ 08540}

\email{jfaergeman@ias.edu}

\address{IHES, 35 route de Chartres, Bures-sur-Yvette 91440, France}

\email{kjaersgaard@ihes.fr}

\begin{abstract}
Nadler and Yun conjectured that the automorphic side of the Betti geometric Langlands correspondence forms a local system of categories over the moduli of curves. We prove this local constancy conjecture using purely microlocal methods, independent of the recently established Betti geometric Langlands conjecture. In particular, it follows that the Betti automorphic category depends only on the genus of the curve, verifying a conjecture of Ben-Zvi and Nadler.
\end{abstract}

\maketitle

\tableofcontents

\section{Introduction}

Let $X$ be a connected smooth projective complex curve and $G$ a connected reductive complex algebraic group.

\subsection{Betti geometric Langlands Conjecture} The Betti geometric Langlands conjecture of \cite{ben2018betti} asserts an equivalence of categories:
\begin{equation}\label{eq:BettiGL}
\on{Sh}_{\on{Nilp}}(\on{Bun}_G(X))\simeq \on{IndCoh}_{\on{Nilp}^{\on{Spec}}}(\on{LocSys}^{\on{Betti}}_{\check{G}}(X)).
\end{equation}

\noindent Here, $\on{Bun}_G(X)$ denotes the moduli stack of $G$-bundles on $X$, and $\on{LocSys}^{\on{Betti}}_{\check{G}}(X)$ denotes the moduli stack of topological $\check{G}$-local systems on the analytification of $X$. We refer to \cite{ben2018betti} for more details but note that the left-hand side denotes the category of all sheaves of $\bC$-vector spaces on the analytic stack underlying $\on{Bun}_G(X)$ whose singular support is contained in the global nilpotent cone.

\subsubsection{} Tautologically, the moduli stack $\on{LocSys}^{\on{Betti}}_{\check{G}}(X)$ only depends on the underlying topological space of $X^{\on{an}}$. On the other hand, the algebraic stack $\on{Bun}_G(X)$ is very sensitive to the algebraic structure on $X^{\on{an}}$. As such, the right-hand side of the equivalence is ``algebra on top of topology'', whereas the left-hand side is ``topology on top of algebra''.

One immediate consequence of the equivalence (\ref{eq:BettiGL}) is the following conjecture of Ben-Zvi and Nadler:
\begin{conjecture}\label{c:independence} (\cite[Conjecture 4.14]{ben2018betti})

The category
\[
\on{Sh}_{\on{Nilp}}(\on{Bun}_G(X))
\]

\noindent only depends on the underlying topological surface of $X$.
\end{conjecture}

\noindent As is emphasized immediately after the statement of \cite[Conjecture 4.14]{ben2018betti}, the above conjecture is far from obvious, since the geometry of the global nilpotent cone of $X$ is highly sensitive to the algebraic structure of $X$.

\subsubsection{} The Betti geometric Langlands conjecture, along with its de Rham version, was recently proven in the series of preprints \cite{GLCII}, \cite{GLCII},\cite{GLCIII}, \cite{GLCIV}, \cite{GLCV}. In particular, this proves Conjecture \ref{c:independence}. However, we highlight two unsatisfactory parts of this proof:
\begin{enumerate}
    \item The proof of the Betti geometric Langlands conjecture is proved in conjunction with its de Rham version. In particular, the proof relies heavily on methods that are only available in the de Rham setting, most notably the representation theory of Kac-Moody Lie algebras at critical level, its relation to opers etc.

    \item The statement of Conjecture \ref{c:independence} only concerns the automorphic of automorphic sheaves on $\on{Bun}_G(X)$ and makes no reference to Langlands duality. Yet its proof is given by applying the Betti geometric Langlands conjecture and thus has Langlands duality at its very core.
\end{enumerate}

\noindent As such, it is not particularly clear from the current proof of Conjecture \ref{c:independence} \emph{why} it should be true from a topological point of view.

The goal of the present paper is to give a purely micro-local proof of Conjecture \ref{c:independence}. In particular, it is independent of the geometric Langlands conjecture.

\subsection{Universal automorphic category}

In the paper \cite{nadler2021automorphic}, Nadler and Yun formulated a conjecture for which the independence of $\on{Sh}_{\on{Nilp}}(\on{Bun}_G(X))$ is a natural consequence. Let us review this setting.

\subsubsection{} Let $S$ be a finite type smooth complex scheme, and let $X\to S$ be a family of smooth projective connected curves over $S$. For $s\in S$, we write $X_s$ for the fiber over $S$, which is a smooth projective connected curve.

We may form the stack
\[
\on{Bun}_G(X/S)\to S
\]

\noindent whose fiber at $s$ is $\on{Bun}_G(X_s)$. As a moduli spaces this stack sends a test object $T\to S$ to the space of $G$-bundles on $X\times_S T$. 

\subsubsection{} One of the main constructions of \cite{nadler2021automorphic}, see also \cite{nadler2026global}, is the existence of a closed conic Lagrangian
\[
\on{Nilp}_u\subset T^*(\on{Bun}_G(X/S))
\]

\noindent whose fiber at $s$ is the usual global nilpotent cone $\on{Nilp}(X_s)\subset T^*(\on{Bun}_G(X_s))$. We refer to $\on{Nilp}_u$ as the ``universal global nilpotent cone''.

$\on{Nilp}_u$ has the property that it is non-characteristic over $S$. That is, it contains no non-zero cotangent vector pulled back from the projection $\on{Bun}_G(X/S)\to S$.

One may then consider the category
\[
\on{Sh}_{\on{Nilp}_u}(\on{Bun}_G(X/S))
\]

\noindent of all Betti sheaves on $\on{Bun}_G(X/S)$ whose singular support is contained in $\on{Nilp}_u$. 

Since $\on{Nilp}_u\to S$ is non-characteristic, for any $s\in S$, restricting along $\on{Bun}_G(X_s)\to \on{Bun}_G(X/S)$ induces a functor
\begin{equation}\label{eq:res_intro}
\on{Sh}_{\on{Nilp}_u}(\on{Bun}_G(X/S))\to \on{Sh}_{\on{Nilp}}(\on{Bun}_G(X_s)).
\end{equation}

\subsubsection{} The above also makes sense if $S$ is replaced by a real analytic space, and $\on{Bun}_G(X/S)$ is replaced by its corresponding real analytic stack (see Appendix \ref{S:analyticbases} or \cite[Appendix A]{nadler2021automorphic} for foundational theory of sheaves on real analytic stacks). Nadler and Yun then pose the following conjecture:
\begin{conjecture}\label{c:univ}(\cite[Conjecture 4.4.3]{nadler2021automorphic})

Suppose $S$ is a contractible real analytic manifold. Then the restriction functor
\[
\on{Sh}_{\on{Nilp}_u}(\on{Bun}_G(X/S))\to \on{Sh}_{\on{Nilp}}(\on{Bun}_G(X_s))
\]

\noindent is an equivalence.
    
\end{conjecture}

\subsubsection{} Denote by $\fM_g$ the stack of smooth proper connected holomorphic curves of genus $g$ with its universal family $\fX_g\to \fM_g$. It follows from Conjecture \ref{c:univ} that the assignment
\[
U\mapsto \on{Sh}_{\on{Nilp}_u}(\on{Bun}_G(X_U/U))
\]

\noindent is a locally constant sheaf of categories on $\fM_g$, where $X_U=\fX_g\underset{\fM_g}{\times} U$. In particular, the fiber category 
\[
\on{Sh}_{\on{Nilp}}(\on{Bun}_G(X_s))
\]

\noindent carries an action of the mapping class group $\pi_1(\fM_g^{\on{an}},X_s)$.

\subsubsection{} The main result of this paper is a resolution of Nadler and Yun's conjecture:
\begin{thm}\label{t:univ_intro}

Conjecture \ref{c:univ} is true.
    
\end{thm}

We emphasize that while Ben-Zvi and Nadler's conjecture that $\on{Sh}_{\on{Nilp}}(\on{Bun}_G(X))$ is independent of the algebraic structure of $X$ follows from the Betti geometric Langlands conjecture, Nadler-Yun's Conjecture \ref{c:univ} does not.

\subsection{Strategy of the proof}\label{S:1.3}

We now briefly outline the idea for the proof of Conjecture \ref{c:univ}.

\subsubsection{} Suppose for all $\cF,\cG\in\on{Sh}_{\on{Nilp}_u}(\on{Bun}_G(X/S))$, one knows that the inner Hom space
\begin{equation}\label{eq:innerhom}
\underline{\on{Hom}}(\cF,\cG)\in\on{Sh}(S)
\end{equation}

\noindent is locally constant, with the stalk at $s\in S$ given by $\on{Hom}(\cF_s,\cG_s)$. When $S$ is contractible (and hence any locally constant sheaf is constant), it implies that the restriction functor
\begin{equation}\label{eq:resintro2}
\on{Sh}_{\on{Nilp}_u}(\on{Bun}_G(X/S))\to \on{Sh}_{\on{Nilp}}(\on{Bun}_G(X_s))
\end{equation}

\noindent is fully faithful. The main idea is to find generators of $\on{Sh}_{\on{Nilp}_u}(\on{Bun}_G(X/S))$ where we know that the inner Hom space is locally constant and, furthermore, that the images of these generators under (\ref{eq:resintro2}) generate the target under colimits.

\subsubsection{} These generators come from two sources:
\begin{itemize}

    \item Geometric Eisenstein series.
    
    \item Beilinson's spectral projector applied to (Hecke translates) of the Poincaré sheaf.

\end{itemize}

We discuss each of these points above in the next two subsections.

\subsection{Geometric Eisenstein series}

A key result is to show that geometric Eisenstein series and constant term functors preserve (universal) nilpotent singular support.

Assume that $S$ is smooth and connected.

\subsubsection{} For a connected component $\mu$ of $\on{Bun}_M(X/S)$, consider the correspondence:
\[\begin{tikzcd}
	& {\mathrm{Bun}_P^{\mu}(X/S)} \\
	{\mathrm{Bun}_M^{\mu}(X/S)} && {\mathrm{Bun}_G(X/S).}
	\arrow["q^{\mu}"', from=1-2, to=2-1]
	\arrow["p^{\mu}", from=1-2, to=2-3]
\end{tikzcd}\]

\noindent We have the constant term and Eisenstein functors:
\[
\on{CT}_P^{\mu}: \on{Sh}(\on{Bun}_G(X/S))\to \on{Sh}(\on{Bun}_M^{\mu}(X/S)),\;\; \cF\mapsto q^{\mu}_!\circ p^{\mu,*}(\cF),
\]
\[
\on{Eis}_P^{\mu}: \on{Sh}(\on{Bun}_M(X/S))\to \on{Sh}(\on{Bun}_G^{\mu}(X/S)),\;\; \cF\mapsto p^{\mu}_!\circ q^{\mu,*}(\cF).
\]

\noindent Define:
\[
\on{CT}_P:=\underset{\mu}{\bigoplus}\on{CT}_{P}^{\mu},
\]
\[
\on{Eis}_P:=\underset{\mu}{\bigoplus}\on{Eis}_{P}^{\mu}.
\]

We prove:
\begin{thm}
The functors $\on{CT}_P, \on{Eis}_P$ preserve universal nilpotent singular support.
\end{thm}

\subsubsection{} When $S=\on{pt}$, the assertion that $\on{Eis}_P$ preserves nilpotent singular support has been folklore for some time. Using the characterization of nilpotent singular support in terms of the Hecke action being locally constant \cite[Thm 14.4.3]{arinkin2020stack}, it follows from the compatibility of the constant term functor with Hecke modifications \cite{braverman2006deformations}, \cite{faergeman2025parabolic} that $\on{CT}_P$ preserves nilpotent singular support (see \cite[\S C.1]{reeves2025geometric} for this line of reasoning).

However, when $S$ is not discrete, this argument does not work (for example, it is not even a priori clear that $\on{CT}_P$ preserves sheaves that are non-characteristic over $S$). In this paper, we give a direct geometric argument that $\on{CT}_P$ preserves (universal) nilpotent singular support, not relying on the characterization of nilpotent sheaves as being Hecke-lisse \cite[Thm 14.4.3]{arinkin2020stack}.

\subsubsection{} Our argument goes roughly as follows:
\begin{itemize}
    \item Prove that $\on{CT}_P^{\mu}$ preserves universal nilpotent singular support for $\mu$ being sufficiently antidominant.

    \item Prove a suitable compatibility between $\on{CT}_P$ and Hecke functors. This is a generalization of \cite{braverman2006deformations}, \cite{faergeman2025parabolic} to the relative setting. Using this and the above point, we make an induction argument on $\mu$.
\end{itemize}

The induction argument is surprisingly subtle. The basic issue is that we are inducting on the dominance of $\mu$ while the Hecke action naturally increases the weight with respect to the root order. This leads to some general combinatorical observations about the root system of $G$ of independent interest, which we briefly describe:

We formulate a one-player game, referred to as \textit{a game of polytopes}, which can be played in any lattice together with a fixed set of lattice polytopes. The ordering of $\mu$ in the induction corresponds to a solution to a particular game of polytopes, constructed from the root system of $G$ and the fundamental Weyl polytopes. We give general conditions for the existence of solutions to the abstract game of polytopes and prove that these conditions are always satisfied in the game obtained from the root system of a reductive group. See \S\ref{s:gameofpolytopes} for more details.

\subsection{Beilinson's spectral projector in the relative setting}

\subsubsection{} Following the proof strategy in \S\ref{S:1.3}, to produce the second class of generators, we need Beilinson's spectral projector in the relative setting.

Let $\on{Sh}_{\on{Nilp}_\mathrm{fib}}(\on{Bun}_G(X/S))$ denote the category of fiberwise nilpotent sheaves; that is, sheaves each of whose fibers along $s\in S$ have nilpotent singular support as a sheaf on $\on{Bun}_G(X_s)$. Generalizing \cite[§15]{arinkin2020stack} to the relative setting, we show that there exists a relative version of Beilinson's spectral projector
\[
P_S: \on{Sh}(\on{Bun}_G(X/S))\to \on{Sh}_{\on{Nilp}_\mathrm{fib}}(\on{Bun}_G(X/S)),
\]

\noindent which is a suitable colimit of Hecke functors over $S$. We prove that this functor sends ind-non-characteristic ind-constructible sheaves to sheaves with universal nilpotent singular support. In this way, we may create objects of the automorphic category $\on{Sh}_{\on{Nilp}_u}(\on{Bun}_G(X/S))$ by applying the projector.

\subsubsection{} We have a natural (relative) geometric Whittaker coefficient functor
\[
\on{coeff}_S: \on{Sh}(\on{Bun}_G(X/S))\to \on{Sh}(S)
\]

\noindent corepresented by the Poincaré sheaf
\[
\on{Poinc}_{S}\in\on{Sh}(\on{Bun}_G(X/S)).
\]

\noindent See \S\ref{ss:poinc} for details. We write
\[
\on{Poinc}_{\on{Nilp},S}:=P_S(\on{Poinc}_{S})\in \on{Sh}_{\on{Nilp}_u}(\on{Bun}_G(X/S))
\]

\noindent for the spectral projector applied to the Poincaré sheaf.

\subsubsection{} The following theorem plays a key role in our argument:
\begin{thm}

The functor
\[
\on{coeff}_S: \on{Sh}_{\on{Nilp}_u}(\on{Bun}_G(X/S))\to \on{Sh}(S)
\]

\noindent satisfies that for a cuspidal sheaf $\cF\in \on{Sh}_{\on{Nilp}_u}(\on{Bun}_G(X/S))$, if for every $V\in\on{Rep}(\check{G})$ we have
\[
\on{coeff}_S(\sH_V(\cF))=0,
\]

\noindent then $\cF=0$.
    
\end{thm}

The above theorem is a generalization of the main result of \cite{faergeman2022non} to the relative setting. We prove this result by reducing it to the case $\on{S}=\on{pt}$ where we may apply \emph{loc.cit}.

\subsubsection{} Proving the local constancy property of inner Hom of automorphic sheaves (\ref{eq:innerhom}) ultimately relies on the following result:

\begin{thm}
    The functor 
    \[\on{coeff}_S: \on{Sh}_{\on{Nilp}_u}(\on{Bun}_G(X/S))\to \on{Sh}(S)\]
    factors through the subcategory of locally constant sheaves $\on{Sh}_\mathbf{0}(S)$.
\end{thm}

We prove this by interpreting $\on{coeff}_S$ as a relative micro-stalk. This is a relative version of \cite{nadler2022whittaker}. Following ideas of \cite{guillermou2015quantizationconiclagrangiansubmanifolds}, we prove a general criterion for when a relative micro-stalk of a sheaf forms local systems and verify it in the case of the universal nilpotent cone and the Whittaker functional. The geometric idea is that moving the base of a microstalk a little inside the smooth locus of the singular support does not change the microstalk.

\subsubsection{} Combining the above two theorems shows that the sheaf obtained by applying Beilinson's spectral projector to $\on{Poinc}_S$ has the property that: 1) its inner hom with any other nilpotent sheaf is a local system on $S$, and 2) any nilpotent sheaf that is right orthogonal to it comes from geometric Eisenstein series.

\subsubsection{} We emphasize that in proving Conjecture \ref{c:univ}, we ultimately rely on the result of \cite{faergeman2022non} stating that the geometric Whittaker coefficient functor (up to Hecke translates at a point) does not kill any non-zero nilpotent cuspidal sheaf. This result was a key input in the proof of the geometric Langlands conjecture. That being said, the proof of this result is purely micro-local in the sense that the proof is given by 1) proving an index formula for the Whittaker coefficient functor, 2) estimating the singular support of nilpotent cuspidal sheaves. That is, in proving Conjecture \ref{c:univ}, we do not make use of, say, the Beilinson-Drinfeld machinery of factorization homology and Kac-Moody localization which is cruical in so many places of the proof of the geometric Langlands conjecture. 

\subsection{Additional results}

The previous two subsections explain the main results that go into proving Conjecture \ref{c:univ}. In this section, we list some additional results we prove of independent interest.

\subsubsection{A claim of Drinfeld}

Fix a smooth projective curve $X$ of genus $g\geq 2$. Denote by 
\[
\on{Bun}_G^{\on{vs}}(X)\subset \on{Bun}_G(X)
\]

\noindent the open subset consisting of very stable $G$-bundles; i.e., bundles $\sP_G$ satisfying that any nilpotent Higgs field $\phi\in \Gamma(X,\fg^*_{\sP_G}\otimes\Omega^1_X)$ is zero. Alternatively, the very stable locus is the complement of the image of the projection
\[
\on{Nilp}_X\backslash 0_{\on{Bun}_G(X)}\to \on{Bun}_G(X).
\]

\noindent It is known that very stable bundles are stable.

\subsubsection{} We let $\on{Wob}(X):=\on{Bun}_G(X)\backslash \on{Bun}_G(X)^{\on{vs}}$ denote the complement of the locus of very stable bundles.

In \cite[Remarque 3.6 (ii)]{laumon1988analogue} (which concerns $G=GL_n$), Laumon states that Drinfeld announced in a letter to Deligne that $\on{Wob}(X)$ is of pure codimension one. Since then, a proof for all reductive groups $G$ has not appeared. For $G=GL_2$, a proof appeared in \cite{pal2021wobbly}, and for $G=GL_n$, a proof appeared in \cite{pal2022conjecture}. Both proofs rely on explicit algebro-geometric methods.

In \S\ref{s:codim1}, we give a short proof of this assertion:
\begin{thm}\label{t:pureofcodim11}

$\on{Wob}(X)$ is of pure codimension one.
    
\end{thm}

\begin{rem}
The proof is essentially an immediate consequence of the existence (due to Beilinson-Drinfeld \cite{beilinson1991quantization}) of a D-module on $\on{Bun}_G(X)$ whose singular support is the entire global nilpotent cone, and the results of \cite{faergeman2022non} ensuring that this sheaf is in the heart of the t-structure and irreducible on each connected component of $\on{Bun}_G(X)$.
\end{rem}

\begin{rem}
Using our main theorem, the natural generalization can be proven for $\on{Bun}_G(X/S)$ for a family of smooth projective complex curves $X\to S$, see Remark \ref{r:generalization} for details.
\end{rem}

\subsection{Structure of the paper}

The reader who wants to get an overview of the proof of the local constancy of the automorphic category may jump directly to Section \ref{s:prfmainthm}.

The paper is divided into three parts. The first part of the paper concerns Eisenstein and constant term functors in the family setting. The second part proves the local constancy of the Whittaker functional. The third part concerns Beilinson's spectral projector in families.

\subsection{AI Declaration}

No AI tools were used in this research project.

\subsection{Acknowledgements}

We thank Dustin Clausen, David Nadler, Sam Raskin and Zhiwei Yun for useful discussions.

Special thanks to Dennis Gaitsgory for very helpful discussions regarding preservation of nilpotent singular support under Eisenstein and constant term functors.

\section{Notation and conventions}

For conventions regarding real analytic spaces and real analytic stacks, see Appendix \ref{S:analyticbases}.

\subsubsection{Curves} Let $X\to S$ be a family of smooth proper holomorphic curves over a real analytic manifold $S$. We assume that the genus of $X/S$ is at least $1$.

\subsubsection{Lie theory conventions} Let $G$ be a connected reductive algebraic group over the complex numbers $\bC$. We assume throughout that the commutator subgroup $[G,G]$ is simply-connected.

Let $\Lambda$ denote the cocharacter lattice of $T$. For a parabolic subgroup $P\subset G$ with Levi quotient $M$, we write $\Lambda_{G,P}$ for the quotient of $\Lambda$ by the coroot lattice of $M$. Let $\Lambda_{G,P}^{\on{pos}}\subset \Lambda_{G,P}$ be the subset spanned by non-negative linear combinations of simple roots that do not lie in $M$.

We write $\Gamma_G$ for the set of vertices of the Dynkin diagram of $G$. Similarly, we write $\Gamma_M$ for the set of vertices of the Dynkin diagram for a Levi $M$. 

We fix an algebraically closed field $e$ of characteristic $0$, the Langlands dual group $\check{G}$ will be defined over $e$.

\subsubsection{Categorical conventions} We use the language of higher category theory and higher algebra in the sense of \cite{lurie2009higher}, \cite{lurie2017higher}, \cite{gaitsgory2019study}. For an algebraically closed field $e$ of characeristic zero, we denote by $\on{Pr}_e^L$ the category of presentable $e$-linear categories in which the morphisms are colimit-preserving functors.

For a category $\cC$ equipped with a t-structure, we denote by $\cC^{\heartsuit}$ the heart of the t-structure.

\subsubsection{Derived algebraic geometry conventions} Our derived algebraic geometry conventions follow those of \cite{gaitsgory2019study}. In particular, we will freely use the language of prestacks in the sense of \emph{loc.cit}.

Throughout, all terminology is derived. For example, our schemes and stacks are assumed to be derived. Moreover, when we talk about the category of quasi-coherent sheaves or sheaves of $\bC$-vector spaces on a stack, it will be in the derived sense.

\subsubsection{Sheaf-theoretic conventions}

We follow the sheaf theoretic conventions of \cite[Appendix A]{nadler2021automorphic}. That is, for a real analytic stack $\sY$, we denote by $\on{Sh}(\sY)$ the category of all complexes of $e$-vector spaces on $\sY$ in the sense of \emph{loc.cit}. Here, $e$ denotes an algebraically closed field of characteirstic zero, though of as our coefficient field. Similarly, we refer to \emph{loc.cit} for conventions regarding singular support.

For a sheaf $\cF\in \on{Sh}(Y)$, we write $\on{SS}(\cF)$ for its singular support. By definition, this is a subset of $T^*\sY$, the (underlying classical stack of the) cotangent stack of $\sY$.

For a sheaf $\cF\in\on{Sh}(\sY)$ and a map $\sY\to Z$, we say that $\cF$ is \emph{non-characteristic over} $Z$ if the image of the codifferential map
\[
T^*Z\underset{Z}{\times} \sY\to T^*\sY
\]

\noindent intersects $\on{SS}(\cF)$ in the zero-section. Similarly, we may talk about subsets of $T^*\sY$ that are non-characteristic over $Z$.

\part{}

\section{Automorphic Sheaves in Families}

\subsection{Definition and basic properties}

The category of automorphic sheaves for a family of curves $X/S$, defined by Nadler and Yun \cite{nadler2021automorphic} is a subcategory of sheaves on $\Bun_G(X/S)$ subject to a nilpotence condition on their microsupport. In this section we will recall their basic properties, and give a characterization of the nilpotence condition in terms of Hecke modifications.

\subsubsection{}\label{ss:nilpudef} Let $\mathrm{Nilp}_\mathrm{rel}\subseteq T^*_S\Bun_G(X/S)$ denote the relative nilpotent cone, that is the closed conical substack of the relative cotangent stack $T^*_S\Bun_G(X/S)$ given by the the preimage of the zero-section under the relative Hitchin fibration 
\[h_S:T^*_S\Bun_G(X/S)\to \mathrm{Hitch}_G(X/S).\]
Let $p:T^*\Bun_G(X/S)\to T^*_S\Bun_G(X/S)$ denote the projection from the absolute cotangent stack to the relative cotangent stack. 

\begin{thm}[Nadler-Yun]\label{the nilpotent cone} There exists an $S$-non-characteristic, closed, conic Lagrangian subset $\mathrm{Nilp}_u\subset T^*\Bun_G(X/S)$, which on field-valued points maps bijectively to $\mathrm{Nilp}_\mathrm{rel}$ under $p$.
\end{thm}

Concretely, $\on{Nilp}_u$ is given by transporting the $0$-section along the codifferential of the map $\on{Bun}_B(X/S)\to \on{Bun}_G(X/S)$. That is, an element $\xi\in T^*_{\sP_G}\on{Bun}_G(X/S)$ lies in $\on{Nilp}_{u,G}$ if and only if there exists a $B$-reduction $\sP_B$ of $\sP_G$ such that $\xi$ is killed under the codifferential
\[
T^*_{\sP_G}\on{Bun}_G(X/S)\to T^*_{\sP_B}\on{Bun}_B(X/S).
\]

We refer to \cite{nadler2021automorphic}, \cite{nadler2026global} for further details about the construction of $\on{Nilp}_u$.

\begin{rem}
    Nadler and Yun's construction only provides a moduli description of $\mathrm{Nilp}_u$ on the level of field-valued points. It would be preferable to have a definition of the universal nilpotent cone that is closer to the definition of the relative nilpotent cone  in terms of a universal lift of parts of the relative Hitchin fibration. In particular, this should give a true moduli description of $\mathrm{Nilp}_u$.
\end{rem}

Using this construction, Nadler and Yun give the following definition:

\begin{defin}
    The category of \emph{automorphic sheaves} for $X/S$ is the subcategory
    \[\mathrm{Sh}_{\mathrm{Nilp}_u}(\Bun_G(X/S))\subset \mathrm{Sh}(\Bun_G(X/S))\]
    of sheaves with singular support contained in $\mathrm{Nilp}_u$.
\end{defin}

\begin{rem}

We warn that Nadler and Yun do not use the term 'automorphic sheaves' in their paper to refer to the category of sheaves with singular support contained in $\on{Nilp}_u$. However, we find this terminology convenient.

\end{rem}

\subsubsection{} Any map of real analytic manifolds $i:S'\to S$ induces a map on the moduli spaces of $G$-bundles
$i_{S'}:\mathrm{Bun}_G(X\times_S S'/S')\to \Bun_G(X/S)$. Note that the $S$-non-characteristicness of $\mathrm{Nilp}_u$ implies that the pullback map
\[i_{S'}^*: \mathrm{Sh}(\Bun_G(X/S))\to \mathrm{Sh}(\mathrm{Bun}_G(X\times_S S'/S'))\]
preserves the subcategories of automorphic sheaves. Thus the categories of families of automorphic sheaves assemble into a sheaf of presentable stable categories over $\mathfrak{M}_g$, the moduli stack of smooth proper complex connected genus $g$ curves, with values $\mathrm{Pr}^L_e$.

In particular, it follows that for any $\mathcal{F}\in \mathrm{Sh}_{\mathrm{Nilp}_u}(\Bun_G(X/S))$ and $s\in S$, the sheaf $i^*\mathcal{F}\in \mathrm{Sh}(\Bun_G(X_s))$ has nilpotent singular support.

We will now prove that this property and non-characteristicness characterize automorphic sheaves:

\begin{lem}\label{l:pointwisenilpsuffices}
   Let $\mathcal{F}\in \mathrm{Sh}(\Bun_G(X/S))$ be a weakly analytically constructible sheaf that is non-characteristic over $S$. Suppose that for all points $s\in S$, the sheaf $i_s^*\mathcal{F}\in \mathrm{Sh}(\Bun_G(X_s))$ has nilpotent singular support. Then $SS(\mathcal{F})\subseteq \mathrm{Nilp}_u$.
\end{lem}

\begin{proof}
    The assumption of $\cF$ ensures that the singular support of $\cF$ is a union of closures of conormal bundles to locally closed substacks of $\on{Bun}_G(X/S)$. As such, it suffices to prove that if $L$ is a non-characteristic closed conic analytic Lagrangian subset of $ T^*\Bun_G(X/S)$ such that its image is contained in $\mathrm{Nilp}_\mathrm{rel}$, then $L$ is contained in $\mathrm{Nilp}_u$. The statement is smooth local on $\Bun_G(X/S)$, hence we may reduce to proving the following statement about real analytic manifolds.

    Let $M\to S$ be a submersion of real analytic manifolds. Suppose $\Lambda$ and $L$ are $S$-non-characteristic closed conic Lagrangian analytic subset of $T^*M$. Let $p:T^*M\to T^*_{S} M$ be the projection to the relative cotangent space. Suppose $p(L)\subset p(\Lambda)$. We want to prove that $L\subset \Lambda$.

    Let $L'$ be an irreducible component of $L$. Since $L'$ is closed and conic, its image under $p$ is closed. As $L$ and $\Lambda$ are Lagrangian and $S$-non-characteristic, it follows that $\dim p(L')=\dim L=\dim \Lambda=\dim p(\Lambda)$. This implies that $p(L')\subset p(\Lambda)$ is an irreducible component of $p(\Lambda)$. Let $\Lambda'$ be an irreducible component of $\Lambda$ such that $p(\Lambda')=p(L')$.

    Since $L'$ is a closed, conic, analytic Lagrangian in $T^*M$, $L'$ is contained in the conormal cone of its image in $M$. Since $L'$ is irreducible and has the same dimension as the conormal cone, they must be equal as closed subsets. Similarly, $\Lambda'$ is the conormal cone to its image in $M$. Since $p(\Lambda')=p(L')$, it follows that the images in $M$ are equal and hence $L'=\Lambda'$.   
\end{proof}

\begin{rem}

The above proof essentially amounts to establishing a uniqueness of liftings of relative Lagrangians. This observation appeared independently in the recent preprint \cite{nadler2026global}.
    
\end{rem}

\subsubsection{} From the above lemma above combined with the results of \cite{arinkin2020stack} and \cite{nadler2021automorphic}, we now obtain a Hecke-characterization for non-characteristic weakly analytically constructible sheaves on $\on{Bun}_G(X/S)$ to have singular supported contained in $\on{Nilp}_u$.

For $V\in\on{Rep}(\check{G})$, denote by $\sH_V$ the natural global Hecke functor
\[
\sH_V: \on{Sh}(\on{Bun}_G(X/S))\to \on{Sh}(\on{Bun}_G(X/S)\underset{S}{\times}X).
\]

\noindent We denote by $T^*\on{Bun}_G(X/S)\underset{S}{\times} 0_X\subset T^*(\on{Bun}_G(X/S)\underset{S}{\times} X)$ the closed subset given by the image of the codifferential of the projection map
\[
\on{Bun}_G(X/S)\underset{S}{\times} X\to \on{Bun}_G(X/S).
\]
\begin{thm}\label{t:Heckechar}
Let $\cF\in \on{Sh}(\on{Bun}_G(X/S))$ be weakly analytically constructible and non-characteristic relative to $S$. Then $\on{SS}(\cF)$ is contained in $\on{Nilp}_u$ if and only if
\[
\on{SS}(\sH_V(\cF))\subset T^*\on{Bun}_G(X/S)\underset{S}{\times} 0_X.
\]
\end{thm}

\begin{proof}
The 'only if' direction is \cite{nadler2021automorphic}. When $S=\on{pt}$, the 'if' direction is \cite[Thm. 14.4.4]{arinkin2020stack}. Thus, the theorem follows from Lemma \ref{l:pointwisenilpsuffices} above.
    
\end{proof}

\subsubsection{} Suppose now that $S$ is a complex manifold. In this case, we also have the following variant of Lemma \ref{l:pointwisenilpsuffices} in the ind-constructible sheaf context (see Appendix \ref{A:non-char} for sheaf contexts). In the ind-constructible context, a sheaf is said to be ind-$S$-non-characteristic if it can be written as a filtered colimit of $S$-non-characteristic constructible sheaves.

\begin{lem}\label{criterion for universal nilpotence ind-version}
    Let $\mathcal{F}\in \mathrm{Sh}(\Bun_G(X/S))$ be an ind-$S$-non-characteristic ind-constructible sheaf. If, for all points $s\in S$, the sheaf $i^*_s\mathcal{F}$ has nilpotent singular support, then $SS(\mathcal{F})\subset \mathrm{Nilp}_u$. 
\end{lem}

\begin{proof}
    By definition we have to prove that if $n\in \mathbb{Z}$ and $\mathcal{F}_0\hookrightarrow {}^p\mathcal{H}^n\mathcal{F}$ is a constructible subsheaf, then $SS(\mathcal{F}_0)\subset \mathrm{Nilp}_u$. Since $\mathcal{F}$ is ind-$S$-non-characteristic, it follows that $\mathcal{H}^n\mathcal{F}$ is a filtered colimit of non-characteristic sheaves. Thus, the map $\mathcal{F}_0\into {}^p\mathcal{H}^n\mathcal{F}$ factors through one of the terms in the colimit and hence $\mathcal{F}_0$ is also $S$-non-characteristic. By Lemma \ref{l:pointwisenilpsuffices}, it suffices to prove that $i^*_s\mathcal{F}_0$ has nilpotent singular support for each $s\in S$.

    Write $\mathcal{F}$ as a filtered colimit $\mathcal{F}\simeq \colim \mathcal{F}_i$, where $\mathcal{F}_i$ is $S$-non-characteristic constructible. Thus we have an isomorphism
    \begin{align*}
       i_s^* {}^p\mathcal{H}^n\mathcal{F}\simeq i_s^* {}^p\mathcal{H}^n \colim \mathcal{F}_i \simeq \colim i_s^* {}^p\mathcal{H}^n\mathcal{F}_i \simeq \colim  {}^p\mathcal{H}^ni_s^*\mathcal{F}_i \simeq {}^p\mathcal{H}^ni_s^*\mathcal{F}
    \end{align*}

    \noindent up to a shift. As such, $i^*_s\mathcal{F}_0$ is a perverse subsheaf of ${}^p\mathcal{H}^ni_s^*\mathcal{F}$. By assumption ${}^p\mathcal{H}^ni_s^*\mathcal{F}$ has nilpotent singular support, hence $i^*_s\mathcal{F}_0$ has nilpotent singular support.
\end{proof}

\section{Geometric Eisenstein Series}

Let $P\subset G$ be a parabolic subgroup and $M$ be its Levi quotient. In this paper all parabolics are assumed to contain our fixed Borel subgroup $B\subset G$.

We have the correspondence of stacks
\[
\begin{tikzcd}
       & \Bun_P \arrow[ld, "p"'] \arrow[rd, "q"] &        \\
\Bun_G &                                         & \Bun_M
\end{tikzcd}\]
Geometric Eisenstein series is the functor $\mathrm{Eis}_!=p_!q^*:\mathrm{Sh} (\Bun_M)\to \mathrm{Sh}(\Bun_G)$. 
\subsection{Preservation of nilpotent singular support under Eisenstein series}
In this section we prove that geometric Eisenstein series preserves sheaves with universal nilpotent singular support.

\begin{prop}\label{eisnilpotence}
    If $\mathcal{F}\in \mathrm{Sh}_{\mathrm{Nilp}_u}(\Bun_M)$, then $\mathrm{Eis}_!\mathcal{F}\in \mathrm{Sh}_{\mathrm{Nilp}_u}(\Bun_G)$. 
\end{prop}
The proof is based on analyzing the singularities of the Drinfeld compactification using the \emph{interplay principle} between Drinfeld's compactification and Zastava spaces. We refer to \cite{braverman2002intersection} for definitions and basic properties about these stacks. We remark that the definitions in \emph{loc.cit} carry over to the relative setting. More precisely, every moduli space over $S$ we will encounter is constructed in the same way (see Appendix \ref{ss:B.101}): it is defined when $S=\fM_g$ is the moduli stack of smooth proper complex curves of genus $g$, and then we simply apply base-change for any real analytic manifold $S$ mapping to (the real analytic stack underlying) $\fM_g$.

We will suppress $X/S$ in the notation of the moduli stacks of bundles in this section. Moreover, in the case when $S$ is scheme and $\sY\to S$ is an algebraic stack (resp. scheme) over $S$, we will often still refer to the base change of $\sY$ along a real analytic manifold $S'\to S$ as a stack (resp. scheme) over $S'$

\subsubsection{} Let $\widetilde{\on{\Bun}}_P$ denote Drinfeld's compactification, which fits in the following diagram
\[
\begin{tikzcd}
       & \widetilde{\on{\Bun}}_P \arrow[ld, "\tilde{p}"'] \arrow[rd, "\tilde{q}"] &        \\
\Bun_G & \Bun_P \arrow[l, "p"'] \arrow[r, "q"] \arrow[u, "j"]   & \Bun_M
\end{tikzcd}\]
where $\tilde{p}$ is proper and $j$ is an open embedding. By the projection formula there is an equivalence
\[\mathrm{Eis}_!\simeq \tilde{p}_!(j_! \mathbf{1}\otimes \tilde{q}^*(-)).\]

\subsubsection{Stratifying the Drinfeld compactification} We define a stratification of $\widetilde{\on{\Bun}}_P$ into locally closed substacks mapping submersively to $S$. For $\theta\in \Lambda_{G,P}^\mathrm{pos}$, let $\widetilde{\on{\Bun}}_{P{= \theta}}$ denote the locally closed substack of the Drinfeld compactification given by fixing the degree of the defect divisor to be $\theta$. There is an equivalence of stacks
\[\widetilde{\on{\Bun}}_{P{= \theta}}\simeq \Bun_P\times_{\Bun_M} \mathcal{H}^{+,\theta}_M,\]
see \cite[\S 3.1]{braverman2002intersection} for the definition of $\mathcal{H}^{+,\theta}_M$ (where it is denoted $\mathrm{Mod}_{\mathrm{Bun}_M}^{+,\theta}$).

For $\theta\in \Lambda_{G,P}$, let $X^{\theta}$ denote the scheme parameterizing a point $s\in S$ and a $\Lambda_{G,P}$-colored divisor of $X_s$ of total degree $\theta$. For each partition $\vartheta$ of $\theta$ given by $\theta=\sum_{i\in I}n_i \theta_i$, for distinct $\theta_i\in \Lambda_{G,P}^\mathrm{pos}\setminus\{0\}$, let $X^\vartheta$ denote the partially symmetrized product $\prod_{i\in I} (X/S)^{(n_i)}$. Let $X^{\vartheta,\circ}$ denote the complement of the big diagonal; i.e., the space where all coordinates are pairwise disjoint.

We have a stratification of $X^\theta$ given by $X^{\vartheta,\circ}\hookrightarrow X^\theta$ for each partition $\vartheta$ of $\theta$. Pulling back along the map $\mathcal{H}^{+,\theta}_M\to X^\theta$, we obtain a stratification  $\mathcal{H}^{+,\vartheta}_M\to \mathcal{H}^{+,\theta}_M$. 

Each $\mathcal{H}^{+,\vartheta}_M$ embeds into the Hecke stack over $X^\vartheta$. Pulling back the stratification by relative position gives a stratification of  $\mathcal{H}^{+,\vartheta}_M$ into locally closed substacks $\mathcal{H}^{+,\vartheta,\lambda}_M$ that are smooth over $X^\vartheta$. Thus we obtain a stratification of $\widetilde{\on{\Bun}}_P$ into locally closed substacks $\Bun_P\times_{\Bun_M} \mathcal{H}^{+,\vartheta,\lambda}_M$ that are smooth over $X^\vartheta\times_S \Bun_M$.

\subsubsection{A singular support estimate} For each $\theta\in \Lambda_{G,P}^\mathrm{pos}$, let $j_{\le \theta}:\Bun_P\to\widetilde{\on{\Bun}}_{P{\le \theta}}$ denote the inclusion of $\Bun_P$ to the open union of strata with defect divisor at most $\theta$.

\begin{prop}\label{non-characteristicness}
    The singular support of the sheaf $j_{\le \theta !}\mathbf{1}$ on $\widetilde{\on{\Bun}}_{P{\le \theta}}$ is contained in the union of the conormal bundles of the strata $\Bun_P\times_{\Bun_M} \mathcal{H}^{+,\vartheta,\lambda}_M$. In particular, $j_{\le \theta !}\mathbf{1}$ is non-characteristic relative to $\Bun_M$.
\end{prop}

\begin{rem}
    Note that $\widetilde{\on{\Bun}}_{P,{\le \theta}}$ is a singular stack, see \cite[App. E.6]{arinkin2020stack} for the definition of singular support in this context.
\end{rem}

\begin{rem}
    In the case of a single curve, this estimate appeared in \cite{lin2022poincareseriesmiraculousduality} Section 5.3. However, the proof we present is different from that of \emph{loc.cit} even when $S$ is a point.
\end{rem}

The proposition immediately implies the analogous statement for the full Drinfeld compactification. It is formulated as is for the purpose of proving it by induction on $\theta$.

\begin{proof}[Proof of Proposition \ref{eisnilpotence} assuming Proposition \ref{non-characteristicness}.]

We start by estimating the singular support of $j_!\mathbf{1}\otimes \tilde{q}^*\mathcal{F}$. 
Consider the map
\[\widetilde{\Bun}_P\overset{(\mathrm{id},\tilde{q})}{\to} \widetilde{\Bun}_P \times \Bun_M\]
and the sheaf $j_!\mathbf{1} \boxtimes \mathcal{F}$, so that $j_!\mathbf{1}\otimes \tilde{q}^*\mathcal{F}\simeq (\mathrm{id},\tilde{q})^*(j_!\mathbf{1} \boxtimes \mathcal{F})$. The singular support of $j_!\mathbf{1} \boxtimes \mathcal{F}$ is equal to $SS(j_!\mathbf{1})\times SS(\mathcal{F})$. Since $j_!\mathbf{1}$ is non-characteristic over $\Bun_M$, it follows that $(\mathrm{id},\tilde{q})$ is non-characteristic with respect to $SS(j_!\mathbf{1})\times SS(\mathcal{F})$. Thus the singular support of $j_!\mathbf{1}\otimes \tilde{q}^*\mathcal{F}$ is contained in the pullback of the cone $SS(j_!\mathbf{1})\times SS(\mathcal{F})$ along $(\mathrm{id},\tilde{q})$. Equivalently, the singular support of $j_!\mathbf{1}\otimes \tilde{q}^*\mathcal{F}$ is contained in the point-wise sum $SS(j_!\mathbf{1})+d\tilde{q}SS(\mathcal{F})$.

Since $\tilde{p}$ is proper, if $(\sP_G,v)$ is a point of $T^*\Bun_G(X/S)$ in the singular support of $\mathrm{Eis}_!\mathcal{F}\simeq \tilde{p}_!(j_!\mathbf{1}\otimes \tilde{q}^*\mathcal{F})$, then there is a point $\tilde{\sP}_P$ on $\widetilde{\on{\Bun}}_P$ such that $(\sP_G,v)$ pulls back to a cotangent vector $(\tilde{\sP}_P,w+w')$ in $SS(j_!\mathbf{1})+d\tilde{q}SS(\mathcal{F})$. Assume $\tilde{\sP}_P$ lies in the stratum $\Bun_P\times_{\Bun_M} \mathcal{H}^{+,\vartheta,\lambda}_M$ of $\widetilde{\on{\Bun}}_P$. By Proposition \ref{non-characteristicness} the pullback of $(\tilde{\sP}_P,w)$ to $\Bun_P\times_{\Bun_M} \mathcal{H}^{+,\vartheta,\lambda}_M$ is $0$. Thus the pullback of $(\tilde{\sP}_P,w+w')$ to $\Bun_P\times_{\Bun_M} \mathcal{H}^{+,\vartheta,\lambda}_M$ is pulled back from a nilpotent cotangent vector along the map $\Bun_P\times_{\Bun_M} \mathcal{H}^{+,\vartheta,\lambda}_M\to \Bun_M$ given by the Hecke modification. Let $\sP_P$ denote the image of $\tilde{\sP}_P$ under the projection $\Bun_P\times_{\Bun_M} \mathcal{H}^{+,\vartheta,\lambda}_M\to \Bun_P$. The pullback of $(\sP_G,v)$ to a cotangent vector of $\Bun_P$ at $\sP_P$ is generically on $X_s$ equal to the pullback of a nilpotent cotangent vector of $\Bun_M$ along the map $\Bun_P\to \Bun_M$. Thus we conclude that $v$ is generically nilpotent and hence nilpotent.

\end{proof}

\subsubsection{Zastava spaces} For any $\lambda\in \Lambda_{G,P}$, let $\Bun_M^\lambda$ denote the component of bundles of degree $\lambda$ and $\widetilde{\on{\Bun}}_{P\le \theta}^\lambda$ be the preimage of $\Bun_M^\lambda$ under the map $\widetilde{\Bun}_{P\le \theta}\to \Bun_M$. Recall the following definition of Zastava spaces.

\begin{defin}
    For $\theta,\theta'\in \Lambda_{G,P}^\mathrm{pos}$ with $\theta'\ge \theta$, and $\theta''\in \Lambda_{G,P}$, let
    \[{}^{\theta''}\mathrm{Zast}^{\theta'}_{\le \theta}\subseteq \widetilde{\on{\Bun}}_{P\le \theta}^{\theta'+\theta''}\times_{\Bun_G}\Bun_{P^-}^{\theta''}\]
    denote the open substack classifying points where the generic $P$-reduction is transversal to the $P^-$-reduction. Let $\mathrm{Zast}^{\theta'}_{\le \theta}$ denote the disjoint union over all $\theta''$. 
\end{defin}

Note that there is a contraction map $\mathrm{Zast}^{\theta'}_{\le \theta}\to \mathcal{H}_M^{\theta',+}$ and in particular a map to $X^{\theta'}$, which we refer to as the divisor. See \cite[5.1]{braverman2002intersection} for a construction.

Let $\mathfrak{n}_P$ denote the Lie algebra of the unipotent radical of $P$. Following \cite{braverman2023intersectioncohomologydrinfeldscompactifications}, define $\Bun_M^r\subset \Bun_M$ to be the open substack classifying $M$-bundles $\sP_M$ such that $H^1(X,\mathcal{U}_{\sP_M})=0$ for all $M$-modules $\mathcal{U}$ appearing as a subquotient of $\mathfrak{n}_P$. Let $\Bun_{P^-}^r$ denote the preimage of $\Bun_M^r$ under the map $\Bun_{P^-}\to \Bun_M$. 

\begin{lem}
    The map $\Bun_{P^-}^r\to \Bun_G$ is smooth.
\end{lem}

\begin{proof}
    Since both stacks are smooth over $S$, the statement reduces to the relative case where it follows from \cite[Lemma 3.7]{braverman2023intersectioncohomologydrinfeldscompactifications}.
\end{proof}

\begin{defin}
    Let ${}^r\mathrm{Zast}_{\le \theta}^{\theta'}$ to be the preimage of $\Bun_{P^-}^r$ under the map $\mathrm{Zast}_{\le \theta}^{\theta'}\to \Bun_{P^-}$.
\end{defin}

By the Lemma above, the map ${}^r\mathrm{Zast}_{\le \theta}^{\theta'}\to \widetilde{\on{\Bun}}_P$ is smooth. In fact, by the results of \cite{drinfeld1995B-structuresonG-bundles}, the Zastava spaces smooth locally model the Drinfeld compactification in the sense that any point of the Drinfeld compactification can be lifted to a point of ${}^r\mathrm{Zast}_{\le \theta}^{\theta'}$.

Since the map ${}^r\mathrm{Zast}_{\le \theta}^{\theta'}\to\widetilde{\on{\Bun}}_P$ is smooth, pulling back the stratification of the Drinfeld compactification, we obtain a stratification of ${}^r\mathrm{Zast}_{\le \theta}^{\theta'}$ into locally closed substacks $\mathcal{Z}_i$ that are smooth over $\on{Bun}_M$.

\subsubsection{Factorization property} For a decomposition $\theta'=\theta'_1+\theta'_2$ with $\theta'_1,\theta'_2\in \Lambda_{G,P}^\mathrm{pos}$, let $(X^{\theta'_1}\times_S X^{\theta'_2})^\mathrm{dis}$ denote the open subscheme where the two divisors are disjoint. There is a natural isomorphism recording the factorization property of the Zastava space under the divisor map
\[({}^r\mathrm{Zast}^{\theta_1'}_{\le \theta'_1}\times_{\Bun_M} {}^r\mathrm{Zast}_{\le \theta'_2}^{\theta'_2})\times_{X^{\theta'}}(X^{\theta'_1}\times_S X^{\theta'_2})^\mathrm{dis}\simeq{}^r\mathrm{Zast}^{\theta'}_{\le \theta'}\times_{X^{\theta'}}(X^{\theta'_1}\times_S X^{\theta'_2})^\mathrm{dis}.\]
The factorization is with respect to the map from Zastava space to $\Bun_M$ given by the map $\Bun_{P^-}\to \Bun_M$.

\subsubsection{The interplay principle} We prove Proposition \ref{non-characteristicness} by using the interplay principle to pass back and forth between $\widetilde{\on{\Bun}}_P$ and the Zastava spaces. To do this we introduce the following proposition, which is the Zastava analogue of Proposition \ref{non-characteristicness}.

Let $j_{Z\le \theta}^{\theta'}:{}^r\mathrm{Zast}^{\theta'}_{\le 0}\to {}^r\mathrm{Zast}^{\theta'}_{\le \theta}$ denote the open inclusion and $j_{Z< \theta}^{\theta'}:{}^r\mathrm{Zast}^{\theta'}_{\le 0}\to {}^r\mathrm{Zast}^{\theta'}_{< \theta}$ .

 \begin{prop}\label{non-char Zastava}
Let $\theta,\theta'\in \Lambda_{G,P}^\mathrm{pos}$ with $\theta'\ge \theta$ be given. The sheaf $j_{Z\le \theta!}^{\theta'}\mathbf{1}$ on ${}^r\mathrm{Zast}^{\theta'}_{\le \theta}$ has singular support contained in the union of the conormal bundles of the strata $\mathcal{Z}_i$. In particular, it is non-characteristic relative to $\Bun_M$. 
 \end{prop}

 The main step in the proof is the following lemma, from which we will conclude Proposition \ref{non-characteristicness} by induction on $\theta$.

\begin{lem}\label{goingupanddown}
    Let $\theta\in  \Lambda_{G,P}^\mathrm{pos}$ be given. If Proposition \ref{non-characteristicness} is true for all $\lambda<\theta$, then the following statements are equivalent.
   \begin{enumerate}
       \item[i)] Proposition \ref{non-characteristicness} is true for $\theta$.
       \item[ii)] Proposition \ref{non-char Zastava} is true for $\theta$ and any $\theta'\ge \theta$.
       \item[iii)] Proposition \ref{non-char Zastava} is true for $\theta$ and $\theta'=\theta$.
   \end{enumerate}
\end{lem}

\begin{proof}
    $i)\Rightarrow ii)$. Since the map $\Bun_{P^-}^r\to \Bun_G$ is smooth, it follows that the map $\pi:{}^r\mathrm{Zast}^{\theta'}_{\le \theta}\to \widetilde{\Bun}_P^{\le \theta}$ is smooth. Since $\pi^* j_{\le \theta !}\mathbf{1}\simeq j_{Z\le \theta !}\mathbf{1}$ it follows that the singular support is contained in the conormal bundle of the pullback stratification. 

    $ii)\Rightarrow iii)$. This implication is immediate.

    $iii)\Rightarrow i)$. Let $\tilde{\sP}_P$ be a point of $\widetilde{\on{\Bun}}_P^{\le \theta}$ and $v$ be a cotangent vector at $\tilde{\sP}_P$. We want to prove that if $v$ is in the singular support of $j_{\le \theta!}\mathbf{1}$ and $\tilde{\sP}_P$ is a point of $\Bun_P\times_{\Bun_M} \mathcal{H}^{+,\vartheta,\lambda}_M$, then $v$ maps to $0$ in the cotangent space of $\Bun_P\times_{\Bun_M} \mathcal{H}^{+,\vartheta,\lambda}_M$. Let $\sP_G$ denote the image of $\tilde{\sP}_P$ in $\Bun_G$. Choose a $P^-$-reduction $\sP_{P^-}$ of $\sP_G$, such that $\sP_{P^-}$ is transversal to $\tilde{\sP}_P$ at the defect points of $\tilde{\sP}_P$. We can further choose $\sP_{P^-}$ with induced $M$-bundle in $\Bun_M^r$. Thus $(\tilde{\sP}_P,\sP_{P^-})$ defines a point of ${}^r\mathrm{Zast}_{\le \theta }^{\theta'}$. By the assumption that $\sP_{P^-}$ is transversal to $\tilde{\sP}_P$ at the defect points of the latter reduction and the factorization property of Zastava spaces, it follows that $(\tilde{\sP}_P,\sP_{P^-})$ defines a point of the open subscheme
    \[({}^r\mathrm{Zast}^{\lambda}_{\le \lambda}\times_{\Bun_M} {}^r\mathrm{Zast}_{\le 0}^{\theta'-\lambda})|_\mathrm{dis}\to {}^r\mathrm{Zast}^{\theta'}_{\le \theta},\]
    for some $\lambda\in \Lambda_{G,P}^\mathrm{pos}$ with $\lambda\le \theta$. By $iii)$ and the induction hypothesis, it follows that $j^{\theta'}_{Z\le \theta !}\mathbf{1}$ restricted to this open subscheme has singular support in the pullback stratification. The following composition is submersive
    
\[{}^r\mathrm{Zast}^{\lambda}_{\le \lambda}\times_{\Bun_M} {}^r\mathrm{Zast}_{\le 0}^{\theta'-\lambda}\to {}^r\mathrm{Zast}^{\theta'}_{\le \theta}\to \widetilde{\Bun}_P.^{\le \theta}\]
Since $\tilde{\sP}_P$ is in the image of the composition, it follows that $(\tilde{\sP}_P,v)$ restricted to $\Bun_P\times_{\Bun_M} \mathcal{H}^{+,\vartheta,\lambda}_M$ vanishes after pullback along a smooth morphism. Hence the cotangent vector $(\tilde{\sP}_P,v)$ already vanishes when restricted to the stratum $\Bun_P\times_{\Bun_M} \mathcal{H}^{+,\vartheta,\lambda}_M$.
\end{proof}

To finish the induction argument, it suffices to prove the following.

\begin{lem}\label{contractionstep}
    Let $\theta'\in \Lambda_{G,P}^\mathrm{pos}$ be given. If Proposition \ref{non-char Zastava} is true for $(\theta,\theta')$ for any $\theta<\theta'$, then Proposition \ref{non-char Zastava} is true for $(\theta,\theta')$.
\end{lem}

\begin{proof}
    Recall that ${}^r\mathrm{Zast}^{\theta'}_{\le \theta'}$ contracts onto ${}^r\mathrm{Zast}^{\theta'}_{= \theta'}$, see \cite[5.1]{braverman2002intersection}. Let $i:{}^r\mathrm{Zast}^{\theta'}_{= \theta'}\to {}^r\mathrm{Zast}^{\theta'}_{\le \theta'}$ denote the inclusion and $\pi:{}^r\mathrm{Zast}^{\theta'}_{\le \theta'}\to {}^r\mathrm{Zast}^{\theta'}_{= \theta'}$ denote the projection. We apply \cite[Proposition 19.4.2]{arinkin2020stack} to estimate the singular support of $j_{Z\le \theta' !}\mathbf{1}$ in terms of $SS(j_{Z< \theta' !}\mathbf{1})$. Let $\xi$ be a cotangent vector of ${}^r\mathrm{Zast}^{\theta'}_{\le \theta'}$ at a point $z$ in ${}^r\mathrm{Zast}^{\theta'}_{= \theta'}$ which is contained in the singular support of $j_{Z\le \theta' !}\mathbf{1}$. Let $\xi^0$ be the image of $\xi$ in the cotangent space of ${}^r\mathrm{Zast}^{\theta'}_{= \theta'}$. By \cite[Proposition 19.4.2]{arinkin2020stack}  there exist a point $z'$ of ${}^r\mathrm{Zast}^{\theta'}_{< \theta'}$ which maps to $z$ under $\pi$, such that the pullback of $\xi^0$ to a cotangent vector $\xi'$ at $z'$ is in $SS(j_{Z<\theta'!}\mathbf{1})$.

    Let $\mathcal{Z}$ denote the stratum of ${}^r\mathrm{Zast}^{\theta'}_{= \theta'}$ containing $z$ and $\mathcal{Z}'$ be the stratum of ${}^r\mathrm{Zast}^{\theta'}_{< \theta'}$ which contains $z'$. We claim that the map $\mathcal{Z}'\cap \pi^{-1}\mathcal{Z}\to \mathcal{Z}$ is étale locally a trivial fibration. To see this, let $\vartheta'$ denote the partition of the defect divisor at the point $z$. Consider the diagram
    \[
\begin{tikzcd}
{{}^r\mathrm{Zast}^{\theta'}_{\le \theta'}\times_{X^{\theta'}} X^{\vartheta',\circ}} \arrow[rr, "\pi"] \arrow[rd] &                                               & {{}^r\mathrm{Zast}^{\theta'}_{= \theta'}\times_{X^{\theta'}} X^{\vartheta',\circ}} \arrow[ld] \\
                                                                                                                            & {\mathrm{Bun}_M\times_S X^{\vartheta',\circ}} &                                                                                                        
\end{tikzcd}\]
    For any point of ${\mathrm{Bun}_M\times_S X^{\vartheta',\circ}}$ we may find an open neighborhood $U$ and choose a trivialization of the formal neighborhood of the divisor determined by $X^{\vartheta',\circ}$ and a trivialization of the universal $M$-bundle over these formal neighborhoods over $U$. By Beauville-Laszlo gluing, the vertical morphisms in the diagram above are fiber bundles and this choice determines a local trivialization above $U$. Let $\mathcal{Y}_1$ denote the fiber of the leftmost map and $\mathcal{Y}_2$ denote the fiber of the rightmost map. With respect to this trivialization, the morphism $\pi$ is induced by a map $\pi':\mathcal{Y}_1\to \mathcal{Y}_2$, and the stratifications of source and target are induced from stratifications of $\mathcal{Y}_1$ and $\mathcal{Y}_2$. The map $\pi'$ is equivariant for the commuting $L^+M$-actions on $\mathcal{Y}_1$ and $\mathcal{Y}_2$ given by changing the trivialization of the $M$-bundle at each point of the divisor. These actions act transitively on the strata of $\mathcal{Y}_2$. Since $L^+M$ is a pro-smooth group scheme, this implies that the $\pi'$ a locally trivial fibration after restricting to the strata of $\mathcal{Y}_2$. Hence $\pi$ is also a locally trivial fibration on the strata.

    By the hypothesis, the cotangent vector $\xi'$ maps to $0$ when restricted to $\mathcal{Z}'$ and hence in particular when restricted to $\mathcal{Z}'\cap \pi^{-1}\mathcal{Z}$. Thus $\xi^0$ restricted to $\mathcal{Z}$ pulls back to $0$ under a locally trivial fibration, hence $\xi^0$ maps to $0$ when restricted $\mathcal{Z}$. This concludes that $\xi$ is in the conormal bundle of the stratum $\mathcal{Z}$. 
\end{proof}

Combining the lemma we finish the proof by induction on $\theta\in \Lambda_{G,P}^\mathrm{pos}$.

\begin{proof}[Proof of Proposition \ref{non-characteristicness}]
    For $\theta_0=0$, $\widetilde{\on{\Bun}}_P^{\le \theta_0}=\Bun_P$ and $j_{\le \theta_0!}\mathbf{1}\simeq \mathbf{1}$. Thus the singular support estimate is true in this case. 

    Let $\theta_0\in \Lambda_{G,P}^\mathrm{pos}$ be given and suppose Proposition \ref{non-characteristicness} is true for all $\theta<\theta_0$. By applying Lemma \ref{goingupanddown} it follows that Proposition \ref{non-char Zastava} is true for $\theta'=\theta_0$ and any $\theta<\theta_0$. Using Lemma \ref{contractionstep} this implies that Proposition \ref{non-char Zastava} is true for $\theta'=\theta_0$ and any $\theta\le \theta_0$. Finally the induction hypothesis and Lemma \ref{goingupanddown} implies that Proposition \ref{non-characteristicness} is true for $\theta=\theta_0$.
\end{proof}

This also concludes the proof of Proposition \ref{non-char Zastava}.

\section{Constant term functors}

In this section, we prove that constant term functors preserve nilpotent singular support.

We continue to assume that we are given a family of holomorphic curves $X\to S$ over a real analytic manifold $S$.

\subsection{Preservation of nilpotent singular support under constant term functors}

Fix a parabolic subgroup $P\subset G$ with Levi quotient $M$ and unipotent radical $N_P$. Recall that $[G,G]$ is simply-connected by assumption so that $\mu\in \Lambda_{G,P}=\pi_1^{\on{alg}}(M)\simeq \pi_0(\on{Bun}_M(X/S))\simeq\pi_0(\on{Bun}_P(X/S))$. We write $\on{Bun}_M^{\mu}(X/S)$ for the corresponding connected component of $\on{Bun}_M(X/S)$.

\subsubsection{} Consider the correspondence:
\[\begin{tikzcd}
	& {\mathrm{Bun}_P^{\mu}(X/S)} \\
	{\mathrm{Bun}_M^{\mu}(X/S)} && {\mathrm{Bun}_G(X/S).}
	\arrow["q^{\mu}"', from=1-2, to=2-1]
	\arrow["p^{\mu}", from=1-2, to=2-3]
\end{tikzcd}\]

\noindent We have the constant term functor:
\[
\on{CT}_P^{\mu}: \on{Sh}(\on{Bun}_G(X/S))\to \on{Sh}(\on{Bun}_M^{\mu}(X/S)),\;\; \cF\mapsto q^{\mu}_!\circ p^{\mu,*}(\cF).
\]

\noindent Let
\[
\on{CT}_P=\underset{\mu\in \Lambda_{G,P}}{\bigoplus}\on{CT}_{P}^{\mu}.
\]

The goal of this section is to prove that $\on{CT}_P$ preserves universal nilpotent singular support:
\begin{thm}\label{t:csttermpreservesnilp}

The constant term functor $\on{CT}_P$ restricts to a functor
\[
\on{CT}_P: \on{Sh}_{\on{Nilp}_{u,G}}(\on{Bun}_G(X/S))\to \on{Sh}_{\on{Nilp}_{u,M}}(\on{Bun}_M(X/S)).
\]
    
\end{thm}

\begin{rem}\label{r:CT*}

There is also the functor $\on{CT}_{P,*}$ right adjoint to $\on{Eis}_{P,!}$, defined as a $!$-pullback followed by a $*$-pushforward. The same proof as that of \ref{t:csttermpreservesnilp} shows that $\on{CT}_{P,*}$ preserves universal nilpotent singular support.\footnote{Alternatively, we may apply the second adjunction equivalence $\on{CT}_{P,*}\simeq \on{CT}_{P^-,!}$.}
    
\end{rem}

\subsubsection{} Let $\Lambda^{\bQ}=\bQ\underset{\bZ}{\otimes}\Lambda$ and write $\Lambda_G^{+,\bQ}$ for the subset of $G$-dominant rational coweights; i.e., those rational coweights that pair non-negatively for each positive simple root of $G$. We similarly write $\Lambda_M^{+,\bQ}$ for the set of $M$-dominant rational coweights when $M$ is a Levi of $G$.

\subsubsection{} For $\lambda_1,\lambda_2\in \Lambda^{\bQ}$, we write $\lambda_1\underset{G}{\leq} \lambda_2$ if $\lambda_2-\lambda_1=\underset{i\in\Gamma_G}{\sum}c_i\check{\alpha}_i$ for $c_i\in\bQ_{\geq 0}$. Denote by $G_{\on{ad}}$ the adjoint quotient of $G$. We write $\lambda_1\underset{G_{\on{ad}}}{\leq} \lambda_2$ if $\overline{\lambda}_1\underset{G_{\on{ad}}}{\leq}\overline{\lambda}_2$, where $\overline{\lambda}_i$ is the projection of $\lambda_i$ under the map
\[
\Lambda_G^{\bQ}\to \Lambda^{\bQ}_{G_{\on{ad}}}.
\]

\subsubsection{Harder-Narasimhan strata} Given $\lambda\in \Lambda^{+,\bQ}$, we have a corresponding Harder-Narasimhan stratum
\[
\on{Bun}_G(X/S)^{(\lambda)}\subset \on{Bun}_G(X/S).
\]

\noindent The set of $\Lambda_G^{+,\bQ}$ such that $\on{Bun}_G(X/S)^{(\lambda)}\neq\emptyset$ is discrete in $\Lambda\otimes\bR$. We refer to \cite[\S 7]{drinfeld2015compact} for details and definitions. We note that \emph{loc.cit} works with the non-relative case $S=\on{pt}$, but the definitions extend to the relative situation.\footnote{We note that the generalization of the constructions and results of \cite[\S 7]{drinfeld2015compact} to the relative situation has been used in the literature before, see e.g. \cite[\S 5.2]{gaitsgory2025geometric}.}

For $\lambda\in \Lambda^{\bQ}$, we write
\[
\on{Bun}_G(X/S)^{(\leq\lambda)}
\]

\noindent for the open substack given by the union of the HN-strata corresponding to rational dominant coweights $\mu\underset{G}{\leq} \lambda$. We sometimes also write
\[
\on{Bun}_G(X/S)^{(\leq\lambda)}=\on{Bun}_G(X/S)^{(\underset{G}{\leq}\lambda)}
\]

\noindent to emphasize that the root-order comes from $G$. We let
\[
\on{Bun}_G(X/S)^{(\underset{G_{\on{ad}}}{\leq}\lambda)}
\]

\noindent be the union of HN-strata $\on{Bun}_G(X/S)^{(\mu)}$ corresponding to $\mu\in\Lambda_G^{+,\bQ}$ satisfying $\mu\underset{G_{\on{ad}}}{\leq}\lambda$.

We use similar notation for the HN-strata on $\on{Bun}_M(X/S)$.

\subsubsection{} For an element $\mu\in \pi_1(M)=\Lambda_{G,P}$ and $\lambda\in \Lambda_M^{+,\bQ}$, we write
\[
j_{(\underset{M_{\on{ad}}}{\leq} \lambda),\mu}: \on{Bun}_M^{\mu}(X/S)^{(\underset{M_{\on{ad}}}{\leq}\lambda)}\into \on{Bun}_M(X/S)
\]

\noindent for the intersection of $\on{Bun}_M(X/S)^{(\underset{M_{\on{ad}}}{\leq}\lambda)}$ and $\on{Bun}_M^{\mu}(X/S)$. Let
\[
\on{Bun}_P^{\mu}(X/S)^{(\underset{M_{\on{ad}}}{\leq}\lambda)}\subset \on{Bun}_P(X/S)
\]

\noindent be the preimage of $\on{Bun}_M^{\mu}(X/S)^{(\underset{M_{\on{ad}}}{\leq}\lambda)}$ under $\on{Bun}_P(X/S)\to \on{Bun}_M(X/S)$. We also denote by $j_{(\underset{M_{\on{ad}}}{\leq} \lambda),\mu}$ the inclusion of $\on{Bun}_P^{\mu}(X/S)^{(\underset{M_{\on{ad}}}{\leq}\lambda)}$ into $\on{Bun}_P^{\mu}(X/S)$.

\subsubsection{Rational decomposition}

Denote by $\Lambda_{Z_0(M)}$ the coweight lattice of $Z_0(M)$, the neutral component of the center of $M$. The canonical map
\[
\Lambda_{Z_0(M)}\to \Lambda\to \Lambda_{G,P}
\]

\noindent is an isomorphism on rationalizations:
\[
\Lambda_{Z_0(M)}^{\bQ}\simeq \Lambda_{G,P}^{\bQ}.
\]

As such, we get a decomposition
\begin{equation}\label{eq:decompoflattice}
\Lambda^{\bQ}\simeq \Lambda_{G,P}^{\bQ}\oplus \bQ\Phi_M\simeq \Lambda_{Z_0(M)}^{\bQ}\oplus \bQ\Phi_M,
\end{equation}

\noindent where $\Phi_M$ denotes the set of roots in $M$. By embedding $ \Lambda_{G,P}^{\bQ}$ into $\Lambda^{\bQ}$, we may talk about the pairing
\[
\langle \beta,\mu\rangle\in \bQ
\]

\noindent for every $\mu\in \Lambda_{G,P}^{\bQ}$ and rational weight $\beta\in \check{\Lambda}^{\bQ}$. By construction, $\langle \mu,\beta\rangle=0$ whenever $\beta$ is a root occuring in $\fm=\on{Lie}(M)$.

\begin{prop}\label{p:goodforsmallmu}
There exists some $N\gg0$ such that for any $\mu\in \Lambda_{G,P}$ satisfying $\langle \mu,\alpha\rangle < -N$ for every positive simple root $\alpha\in \Gamma_G-\Gamma_M$, the functor
\[
j_{(\underset{M_{\on{ad}}}{\leq} (2g-2)\rho_M),\mu}^*\circ \on{CT}_P^{\mu}: \on{Sh}_{\on{Nilp}_{u,G}}(\on{Bun}_G(X/S))\to \on{Sh}(\on{Bun}_M^{\mu}(X/S)^{(\underset{M_{\on{ad}}}{\leq} (2g-2)\rho_M)})
\]

\noindent lands in $\on{Sh}_{\on{Nilp}_{u,M}}(\on{Bun}_M^{\mu}(X/S)^{(\underset{M_{\on{ad}}}{\leq}(2g-2)\rho_M)})$.
\end{prop}

\begin{proof}

\step 

First, we claim that the map
\[
\on{Bun}_P(X/S)^{(\underset{M_{\on{ad}}}{\leq} (2g-2)\rho_M),\mu}\to \on{Bun}_G(X/S)
\]

\noindent is smooth. It suffices to show that the map of tangent spaces is surjective. For $(s,\sP_G)$, where $s\in S$ and $\sP_P\in \on{Bun}_P(X/S)$ with induced $G$-bundle $\sP_G$, it follows by reduction to $S=\on{pt}$ using the smoothness of the source and target over $S$, that we have:
\[
\on{coKer}\big(H^0(T_{(s,\sP_P)} \on{Bun}_P(X/S))\to H^0(T_{(s,\sP_G)} \on{Bun}_G(X/S))\big)\simeq H^1(X_s,(\fg/\fp)_{\sP_P}).
\]

The right-hand side has a filtration whose associated graded identifies with $H^1(X_s,(\fg/\fp)_{\sP_M})$. Thus, it suffices to show that the latter vanishes whenever $\sP_M$ lives over some Harder-Narasimhan stratum $\on{Bun}_M(X/S)^{(\lambda)}$ with $\lambda\in \Lambda_M^{+,\bQ}$ $M$-dominant and $\lambda\underset{M_{\on{ad}}}{\leq} (2g-2)\rho_M$.

Let $R$ be the Harder-Narasimhan parabolic with Levi $L$ associated to $\lambda$; that is, the simple roots of $L$ are those simple roots $\alpha$ of $M$ that satisfy $\langle \alpha,\lambda\rangle=0$. Then $\sP_M$ has a canonical $R$-reduction $\sP_R$ with induced semistable $L$-bundle $\sP_L$ of degree $\lambda$. The bundle $(\fg/\fp)_{\sP_M}$ has a filtration with associated graded identified with $(\fg/\fp)_{\sP_L}$. Thus, it suffices to show that $H^1(X_s,(\fg/\fp)_{\sP_L})=0$.

By \cite[Rmk. 10.3.5]{drinfeld2015compact}, it suffices to show that 
\begin{equation}\label{eq:ineq}
\langle \beta,\lambda\rangle < 2-2g
\end{equation}

\noindent for all positive roots $\beta$ not in $M$. Since $\on{Bun}_M(X/S)^{(\lambda)}$ lies in $\on{Bun}_M^{\mu}(X/S)$, it follows that $\mu$ and $\lambda$ have the same images along the projection
\[
\Lambda^{\bQ}\to \Lambda_{G,P}^{\bQ}.
\]

In terms of the decomposition (\ref{eq:decompoflattice}), we may write
\[
\lambda=\mu+\nu
\]

\noindent for some $\nu\in \bQ\Phi_M$. Since $\mu$ is orthogonal to the roots in $M$, it follows that $\nu$ is $M$-dominant. Thus, we may write
\[
\nu=\underset{i\in\Gamma_M}{\sum} c_i\alpha_i,
\]

\noindent with $c_i\geq 0$. Since $\lambda\underset{M_{\on{ad}}}{\leq} (2g-2)\rho_M$, we have $\nu \underset{M}{\leq} (2g-2)\rho_M$. Thus $0\leq c_i\leq c'$ for some bound $c'$ independent of $\mu$. It follows that $\langle \nu,\beta\rangle$, as $\beta$ ranges over all roots in $G$, is bounded independently of $\mu$. Choosing $N$ large enough thus gives the vanishing of $H^1(X_s,(\fg/\fp)_{\sP_L})$.

Define $\on{Nilp}_{u,P}$ to be the set-theoretic image of the codifferential
\[
dp:\on{Nilp}_{u,G}\underset{\on{Bun}_G(X/S)}{\times} \on{Bun}_P(X/S)\to T^*\on{Bun}_P(X/S).
\]

\noindent We say that an element $(\sP_P,\phi)\in T^*\on{Bun}_P(X/S)$ is nilpotent if it lies in $\on{Nilp}_{u,P}$. By smoothness, $\on{Nilp}_{u,P}$ is a closed conic Lagrangian over $\on{Bun}_P(X/S)^{(\underset{M_{\on{ad}}}{\leq} (2g-2)\rho_M),\mu}$, which is moreover non-characteristic over $S$.

As such, $(j_{(\underset{M_{\on{ad}}}{\leq} (2g-2)\rho_M),\mu}\circ p)^{*}$ gives a functor
\[
(j_{(\underset{M_{\on{ad}}}{\leq} (2g-2)\rho_M),\mu}\circ p)^{*}: \on{Sh}_{\on{Nilp}_{u,G}}(\on{Bun}_G(X/S))\to \on{Sh}_{\on{Nilp}_{u,P}}(\on{Bun}_P^{\mu}(X/S)^{(\underset{M_{\on{ad}}}{\leq}(2g-2)\rho_M)}).
\]

\step 

By Step 1, it suffices to show that $q_!^{\mu}$ induces a functor
\[
q_!^{\mu}: \on{Sh}_{\on{Nilp}_{u,P}}(\on{Bun}_P^{\mu}(X/S)^{(\underset{M_{\on{ad}}}{\leq}(2g-2)\rho_M)})\to \on{Sh}_{\on{Nilp}_{u,M}}(\on{Bun}_M^{\mu}(X/S)^{(\underset{M_{\on{ad}}}{\leq}(2g-2)\rho_M)}).
\]

\noindent There is a (trivial) $\bG_m$-action on $\on{Bun}_P(X/S)$ contracting it onto $\on{Bun}_M(X/S)$. Using \cite[Rmk. 10.3.5]{drinfeld2015compact}, the argument in Step 1 shows that choosing $\mu$ sufficiently anti-dominant, we have $H^0(X_s,({\fn}_P)_{\sP_M})=0$ for any $\sP_M\in \on{Bun}_M^{\mu}(X/S)^{(\underset{M_{\on{ad}}}{\leq}(2g-2)\rho_M)}$. This implies that
\[
q^{\mu}: \on{Bun}_P^{\mu}(X/S)^{(\underset{M_{\on{ad}}}{\leq}(2g-2)\rho_M)}\to \on{Bun}_M^{\mu}(X/S)^{(\underset{M_{\on{ad}}}{\leq}(2g-2)\rho_M)}
\]

\noindent is affine and separated (see \cite[Prop. 11.1.4 (c)]{drinfeld2015compact} whose proof readily extends to the relative setting). In particular, the natural map
\[
\iota^{\mu}: \on{Bun}_M^{\mu}(X/S)^{(\underset{M_{\on{ad}}}{\leq}(2g-2)\rho_M)}\to \on{Bun}_P^{\mu}(X/S)^{(\underset{M_{\on{ad}}}{\leq}(2g-2)\rho_M)}
\]

\noindent induced by $M\into P$ is a section of $q^{\mu}$ and hence a closed embedding. Denote by $j^{\mu}: \sU\into \on{Bun}_P^{\mu}(X/S)^{(\underset{M_{\on{ad}}}{\leq}(2g-2)\rho_M)}$ the complement of $\iota^{\mu}$. By the contraction principle, we have a natural isomorphism:
\[
q^{\mu}_!\simeq \iota^{\mu,!}: \on{Sh}_{\on{Nilp}_{u,P}}(\on{Bun}_P^{\mu}(X/S)^{(\underset{M_{\on{ad}}}{\leq}(2g-2)\rho_M)})\to \on{Sh}_{\on{Nilp}_{u,M}}(\on{Bun}_M^{\mu}(X/S)^{(\underset{M_{\on{ad}}}{\leq}(2g-2)\rho_M)}),
\]

\noindent since any object of the left-hand side is automatically $\bG_m$-monodromic. By the triangle
\[
\iota^{\mu}_!\circ \iota^{\mu,!}\to \on{id}\to j^{\mu}_*\circ j^{\mu,*},
\]

\noindent it suffices to show that $j^{\mu}_*$ restricts to a functor
\[
j^{\mu}_*: \on{Sh}_{\on{Nilp}_{u,P}}(\sU^{\mu})\to \on{Sh}_{\on{Nilp}_{u,P}}(\on{Bun}_P^{\mu}(X/S)^{(\underset{M_{\on{ad}}}{\leq}(2g-2)\rho_M)}).
\]

\noindent We claim this follows from \cite[Prop. 19.4.2]{arinkin2020stack}. Indeed, \emph{loc.cit} implies that for any $\cF\in \on{Sh}_{\on{Nilp}_{u,P}}(\sU^{\mu})$ and covector $\xi_P\in \on{SS}(j^ {\mu}_*(\cF))$, its component $\xi_P^0\in T^*\on{Bun}_M(X/S)$ maps to $\on{Nilp}_{u,P}$ under the codifferential induced by $\on{Bun}_P^{\mu}(X/S)\to \on{Bun}_M^{\mu}(X/S)$. By S-non-characteristicness of $\on{Nilp}_{u,P}$ over $\on{Bun}_P^{\mu}(X/S)^{(\underset{M_{\on{ad}}}{\leq}(2g-2)\rho_M)}$, it follows that $\on{SS}(j^{\mu}_*(\cF))$ is non-characteristic over $S$. Moreover, it is clear by the above condition on $\xi_P^0$ that $\xi_P$ is pointwise nilpotent; i.e., defines an element of $\on{Nilp}_{u,P}\underset{S}{\times} \lbrace s \rbrace$. That $\xi_P\in\on{Nilp}_{u,P}$ now follows from the proof of Lemma \ref{l:pointwisenilpsuffices}.

\end{proof}

\subsection{Compatibility between constant term functors and Hecke operators}

We would like to extend the statement of Proposition \ref{p:goodforsmallmu} to all $\mu\in \Lambda_{G,P}$. For this, we need a certain compatibility between constant term functors and Hecke functors, which we formulate for any parabolic $P$.

\subsubsection{} Assume that we have a section $\sigma: S\to X$ of the structure map $X\to S$. Given a representation $V\in \on{Rep}(\check{G})$, we get a Hecke functor:
\[
V_{\sigma}\star -: \on{Sh}(\on{Bun}_G(X/S))\to \on{Sh}(\on{Bun}_G(X/S))
\]

\noindent such that for each $s\in S$, $*$-pulling back along $s\to S$ recovers the usual Hecke operator at $\sigma(s)$:
\[
V_{\sigma(s)}\star-: \on{Sh}(\on{Bun}_G(X_s))\to \on{Sh}(\on{Bun}_G(X_s)).
\]

\noindent We denote by $\Lambda_M^+$ the set of coweights that are dominant for $M$. For $\nu\in \Lambda_M^+$, we denote by $V_{\check{M}}^{\nu}$ the corresponding highest weight representation for $\check{M}$.

\subsubsection{} The goal of this section is to prove the following result:

\begin{thm}\label{t:HeckeCT-formula}
The functor
\[
\on{CT}_P(V_{\sigma}\star-): \on{Sh}(\on{Bun}_G(X/S))\to \on{Sh}(\on{Bun}_M(X/S))
\]

\noindent has a canonical filtration whose associated graded identifies with
\[
\underset{\nu\in\Lambda_M^+}{\bigoplus} (V_{\check{M},\sigma}^{\nu}\star \on{CT}_P(-))\otimes \on{Hom}_{\check{M}}(V_{\check{M}}^{\nu},V)[-\langle 2(\rho_G-\rho_M),\nu\rangle].
\]
\end{thm}

\begin{rem}
Up to the shifts appearing in the formula, this is exactly the formula one would expect from looking at parabolic restriction of representations. Indeed, pull-push along the correspondence
\[\begin{tikzcd}
	{\bB \check{P}} && {\bB\check{G}} \\
	\\
	{\bB \check{M}}
	\arrow[from=1-1, to=1-3]
	\arrow[from=1-1, to=3-1]
\end{tikzcd}\]

\noindent gives a functor $\on{Rep}(\check{G})\to \on{Rep}(\check{M})$. The fact that any $\check{G}$-representation $V$ has a filtration as a $\check{P}$-representation whose associated graded coincides with the restriction of $\on{Res}^{\check{G}}_{\check{M}}(V)$ along $\check{P}\to \check{M}$ gives the analogue of Theorem \ref{t:HeckeCT-formula} in this case.
\end{rem}

\subsubsection{} To prove Theorem \ref{t:HeckeCT-formula}, we use the formalism of Drinfeld-Plücker structures and enhanced Drinfeld-Plücker structures following \cite[§4]{faergeman2025parabolic}. We refer to \emph{loc.cit} for details.

Suppose we are given a category $\sC\in \mathrm{Pr}^L_e$ which is a $(\on{Rep}(\check{M}),\on{Rep}(\check{G}))$-bimodule category.

\subsubsection{Drinfeld-Plücker structures}

Let $\overline{\check{N}_P\backslash \check{G}}$ be the basic affine space; i.e., the affinization of $\check{N}_P\backslash \check{G}$. We may consider the corresponding algebra of functions on the basic affine space
\[
\cO(\overline{\check{N}_P\backslash \check{G}})\in \on{Rep}(\check{M})\otimes \on{Rep}(\check{G}).
\]

\noindent We let
\[
\on{DrPl}_{\check{M},\check{G}}(\sC):=\cO(\overline{\check{N}_P\backslash \check{G}})\on{-mod}(\sC).
\]

\noindent For an object $c\in \sC$, we refer to a lift of c to an object of $\on{DrPl}_{\check{M},\check{G}}(\sC)$ as a Drinfeld-Plücker structure on $c$. Concretely, this amounts to a family of maps
\[
V^{\check{N}_P}\star c\to c\star V,\;\; V\in \on{Rep}(\check{G})
\]

\noindent satisfying the Drinfeld-Plücker identities. We remark here that by the invariants $(-)^{\check{N}_P}$ appearing above, we mean non-derived invariants.

\subsubsection{Enhanced Drinfeld-Plücker structures} Consider the commutative dg-algebra
\[
\cO(\check{N}_P\backslash \check{G}):=R\Gamma(\check{N}_P\backslash \check{G}, \cO)\in \on{Rep}(\check{M})\otimes \on{Rep}(\check{G}).
\]

\noindent We let
\[
\on{EnhDrPl}_{\check{M},\check{G}}(\sC):=\cO(\check{N}_P\backslash \check{G})\on{-mod}(\sC).
\]

\noindent For an object $c\in \sC$, we refer to a lift of $c$ to an object of the above category as an enhanced Drinfeld-Plücker structure on $c$. 

\subsubsection{}\label{ss:ind+oblv} The algebra map $\cO(\overline{\check{N}_P\backslash \check{G}})\to \cO(\check{N}_P\backslash \check{G})$ induces a fully faithful forgetful functor
\[
\on{oblv}_{\on{DrPl}_{\check{M},\check{G}}}^{\on{EnhDrPl}_{\check{M},\check{G}}}: \on{EnhDrPl}_{\check{M},\check{G}}(\sC)\to \on{DrPl}_{\check{M},\check{G}}(\sC)
\]

\noindent with left adjoint
\[
\on{ind}_{\on{DrPl}_{\check{M},\check{G}}}^{\on{EnhDrPl}_{\check{M},\check{G}}}: \on{DrPl}_{\check{M},\check{G}}(\sC)\to \on{EnhDrPl}_{\check{M},\check{G}}(\sC),\;\; c\mapsto \cO(\check{N}_P\backslash \check{G})\underset{\cO(\overline{\check{N}_P\backslash \check{G}})}{\otimes} c.
\]
\subsubsection{} Let
\[
\Omega(\check{\fn}_P):=C^{\bullet}(\check{\fn}_P,\bC)\in\on{Rep}(\check{M})
\]

\noindent be the cohomological Chevalley complex for $\check{\fn}$, considered as a commutative algebra object in $\on{Rep}(\check{M})$. It is shown in \cite[§4.2]{faergeman2025parabolic} that for $c\in\on{EnhDrPl}_{\check{M},\check{G}}(\sC)$ (with underlying object also denoted by $c$), then from the canonical algebra map
\[
\Omega(\check{\fn}_P)\to \cO(\check{N}_P\backslash \check{G}),
\]

\noindent $c$ is equipped with a module structure for $\Omega(\check{\fn}_P)$, and in addition, has a family of canonical isomorphisms:
\[
C^{\bullet}(\check{\fn}_P,V)\underset{\Omega(\check{\fn}_P)}{\star} c\simeq c\star V, \;\; V\in \on{Rep}(\check{G}).
\]

\noindent It follows that for $c$ equipped with an enhanced Drinfeld-Plücker structure, the object 
\[
c\star V
\]

\noindent carries a filtration with associated graded
\[
\underset{\nu\in\Lambda_M^+}{\bigoplus} (V_{\check{M}}^{\nu}\star c)\otimes \on{Hom}_{\check{M}}(V_{\check{M}}^{\nu},V).
\]

To prove Theorem \ref{t:HeckeCT-formula}, it thus remains to equip the suitable kernel realizing $\on{CT}_P$ with an enhanced Drinfeld-Plücker structure.

\subsubsection{} The section $\sigma: S\to X$ defines an action
\[
\on{Rep}(\check{G})\curvearrowright \on{Sh}(\on{Bun}_G(X/S)),
\]

\noindent and similarly an action
\[
\on{Rep}(\check{M})\curvearrowright \on{Sh}(\on{Bun}_M(X/S)).
\]

\subsubsection{Plücker data with poles: absolute case} For $X$ a smooth projective curve over $\bC$ with a point $x\in X$, recall the prestack
\[
\widetilde{\on{Bun}}_{P,\infty\cdot x}
\]

\noindent considered in \cite[§5.2]{faergeman2025parabolic}.\footnote{In \emph{loc.cit}, the stack $\widetilde{\on{Bun}}_{P,\on{pol}}$ was considered. By definition, the stack $\widetilde{\on{Bun}}_{P,\infty\cdot x}$ is the fiber over $x$ of the map $\widetilde{\on{Bun}}_{P,\on{pol}}\to \on{Ran}$.} Concretely, for an affine test scheme $T$, the groupoid $\widetilde{\on{Bun}}_{P,\infty\cdot x}(T)$ classifies the data $(\sP_G,\sP_M,\kappa)$, where $\sP_G$ (resp. $\sP_M$) is a $G$-bundle (resp. $M$-bundle) on $T\times X$, and $\kappa$ is a collection of injective meromorphic maps of coherent sheaves
\[
V^{N_P}_{\sP_M}\to V_{\sP_G}
\]

\noindent defined for each $G$-representation $V$, satisfying the Plücker relations, and such that these maps are defined and non-degenerate away from $T\times x\subset T\times X$. Here, $V_{\sP_G}$ is the vector bundle given by the twist of $\sP_G$ and $V$ (and similarly for $\sP_M$ and $V^{N_P}$), and non-degenerate means that the map is an injective map of vector bundles (i.e., the cokernel is flat on $T\times X$).

The main property of $\widetilde{\on{Bun}}_{P,\infty\cdot x}$ is that it carries commuting Hecke actions for $M$ and $G$. That is, $\on{Sh}(\widetilde{\on{Bun}}_{P,\infty\cdot x})$ is naturally a $(\on{Rep}(\check{M}),\on{Rep}(\check{G}))$-bimodule category.

\subsubsection{Plücker data with poles: relative case}

Let us assume again that $X\to S$ is a family of proper holomorphic curves with a section $\sigma$ and $S$ is a real analytic manifold. We have an analogous prestack
\[
\widetilde{\on{Bun}}_{P,\infty\cdot \sigma}(X/S).
\]

\noindent As in \S\ref{ss:B.101}, it suffices to define the prestack when $S$ is a complex scheme. In this case, for an affine test scheme $T$, the groupoid $\widetilde{\on{Bun}}_{P,\infty\cdot \sigma}(X/S)(T)$ classifies the data $(s,\sP_G,\sP_M,\kappa)$, where:
\begin{itemize}
    \item $s: T \to S$ is a map;
    \item $\sP_G$ is a $G$-bundle on $T\underset{S}{\times} X$ and $\sP_M$ is an $M$-bundle on $T\underset{S}{\times} X$;
    \item $\kappa$ is a collection of injective meromorphic maps of coherent sheaves for each $G$-representation $V$:
  \[
  V^{N_P}_{\sP_M}\to V_{\sP_G},
  \]

  \noindent satisfying the Plücker relations and which are defined and non-degenerate away from the image of the map $T\to T\underset{S}{\times} X$ induced by $\sigma$.
\end{itemize}

By construction, $\widetilde{\on{Bun}}_{P,\infty\cdot \sigma}(X/S)$ carry Hecke modifications for $M$ and $G$ along $\sigma$, and so
\[
\on{Sh}(\widetilde{\on{Bun}}_{P,\infty\cdot \sigma}(X/S))
\]

\noindent is naturally a $(\on{Rep}(\check{M}),\on{Rep}(\check{G}))$-bimodule category.

\subsubsection{}

We have a canonical open embedding
\[
j_{\on{pol}}: \on{Bun}_P(X/S)\into \widetilde{\on{Bun}}_{P,\infty\cdot \sigma}(X/S).
\]

\noindent We define
\[
\mathbf{j}_!^{\on{glob}}:=j_{\on{pol},!}(\on{IC}_{\on{Bun}_P(X/S)/S})\in \on{Sh}(\widetilde{\on{Bun}}_{P,\infty\cdot \sigma}(X/S)),
\]

\noindent where $\on{IC}_{\on{Bun}_P(X/S)/S}$ is the sheaf whose $*$-fiber along $\on{Bun}_P(X_s)=\lbrace s\rbrace\underset{S}{\times} \on{Bun}_P(X/S)\to \on{Bun}_P(X/S)$ is $\on{IC}_{\on{Bun}_P(X_s)}$.

We will prove:
\begin{thm}\label{t:enhdrplonj}
$\mathbf{j}_!^{\on{glob}}$ carries a canonical enhanced Drinfeld-Plücker structure.
\end{thm}

\begin{rem}
A natural Ran-enhancement of the above theorem holds with essentially the same proof. Since we will not need this, however, we will not formulate it.
\end{rem}

\subsubsection{} We claim that Theorem \ref{t:enhdrplonj} implies Theorem \ref{t:HeckeCT-formula}. Indeed, write
\[
p_{\on{pol}}: \widetilde{\on{Bun}}_{P,\infty\cdot \sigma}(X/S)\to \on{Bun}_G(X/S),
\]
\[
q_{\on{pol}}: \widetilde{\on{Bun}}_{P,\infty\cdot \sigma}(X/S)\to \on{Bun}_M(X/S)
\]

\noindent for the canonical projection maps. Then we have a canonical natural isomorphism:
\[
\on{CT}_P\simeq q_{\on{pol},!}\big (\mathbf{j}_!^{\on{glob}}\overset{*}{\otimes} p_{\on{pol}}^*)[-\on{dim} \on{Bun}_P]: \on{Sh}(\on{Bun}_G(X/S))\to\on{Sh}(\on{Bun}_M(X/S)),
\]

\noindent and the functor $p_{\on{pol}}^*$ (resp. $q_{\on{pol},!}$) is $\on{Rep}(\check{G})$-linear (resp. $\on{Rep}(\check{M})$-linear). Here, $\on{dim} \on{Bun}_P$ denotes the locally constant function on $\on{Bun}_P(X/S)$ giving the relative dimension over $S$. Concretely, for a connected component $\mu\in \Lambda_{G,P}$, we have:
\[
(\on{dim} \on{Bun}_P)(\mu)=\on{dim}_S \on{Bun}_M(X/S)+\on{dim}_S\on{Bun}_{N_P}(X/S)+\langle 2(\rho_G-\rho_M),\mu\rangle.
\]

\subsubsection{}

When $S=\on{pt}$ is a point, we remind that Theorem \ref{t:enhdrplonj} was proved in the case in \cite{braverman2006deformations} for $P=B$ a Borel, and for a general parabolic in \cite{faergeman2025parabolic}:
\begin{thm}\label{t:FH}[\cite{faergeman2025parabolic}, Thm 1.5.5.1]

For $s\in S$ a geometric point, the sheaf
\[
j_{s,!}(\on{IC}_{\on{Bun}_P}(X_s))\in \on{Sh}(\widetilde{\on{Bun}}_{P,\infty\cdot \sigma(s)})
\]

\noindent carries a canonical enhanced Drinfeld Plücker structure for the Hecke actions at $\sigma(s)\in X$.

\end{thm}

\subsection{Enhanced Drinfeld-Plücker structure: local story}

In this subsection, we prove a local analogue of Theorem \ref{t:enhdrplonj}. We do so by reducing to the case when $S=\on{pt}$, which is the content of \cite[Prop. 4.5.4.1]{faergeman2025parabolic}.

\subsubsection{} Let $\on{Gr}_{G,\sigma}$ be the affine Grassmannian associated to $\sigma$. To define $\on{Gr}_{G,\sigma}$, it again suffices to consider the case of $S$ algebraic in which case $\on{Gr}_{G,\sigma}(T)$ classifies a map $T\to S$, a $G$-bundle $\sP_G$ on $T\underset{S}{\times}X$ and a trivialization $\alpha$ away from $\sigma: T\to T\underset{S}{\times}X$.

Let $\cL^+_{\sigma}M\to S$ be the group over $S$ given by the relative arc group associated to $\sigma$ satisfying that its fiber over a geometric point $s\in S$ is the usual arc group $\cL^+_{\sigma(s)}M$. Concretely, when $S$ is algebraic, $\cL^+_{\sigma}M(T)$ classifies a map $s:T\to S$ and a map $(T\underset{S}{\times} X)^{\widehat{}}_T\to M$, where $(T\underset{S}{\times} X)^{\widehat{}}_T$ denotes the formal completion of $T\underset{S}{\times} X$ along the graph of $\sigma \circ s: T\to X$ over $S$.

\subsubsection{} The category
\[
\on{Sh}(\cL^+_{\sigma}M\backslash \on{Gr}_{G,\sigma})
\]

\noindent carries commuting Hecke actions for $M$ and $G$; as such it carries commuting actions for $\on{Rep}(\check{M})$ and $\on{Rep}(\check{G})$.

\subsubsection{Twisted Action}\label{ss:twistedaction}
Recall the following automorphism of $\on{Rep}(\check{M})$ considered in \cite[§4.3.1]{faergeman2025parabolic}. The lattice $\Lambda_{G,P}$ coincides with the character lattice of $Z(\check{M})$, the center of $\check{M}$ (which is connected because $[G,G]$ is simply-connected by assumption). As such, we have a canonical decomposition
\[
\on{Rep}(\check{M})\simeq\underset{\theta\in \Lambda_{G,P}}{\bigoplus} \on{Rep}(\check{M})^{\theta}.
\]

\noindent We let $\sS^{\theta}$ denote the endofunctor
\[
\on{Rep}(\check{M})^{\theta}\to \on{Rep}(\check{M})^{\theta},\;\; V\mapsto V[-2\langle \rho_G-\rho_M,\theta\rangle].
\]

\noindent Here, $\rho_G$ (resp. $\rho_M$) denotes the half-sum of the positive roots of $G$ (resp. $M$). This induces a symmetric monoidal autoequivalence
\[
\sS:=\underset{\theta\in\Lambda_{G,P}}{\bigoplus} \sS^{\theta}:\on{Rep}(\check{M})\to \on{Rep}(\check{M}).
\]

\subsubsection{} Consider the twisted Hecke action for $M$:
\begin{equation}\label{eq:twHaction}
\on{Rep}(\check{M})\xrightarrow{\sS} \on{Rep}(\check{M})^{\theta}\curvearrowright \on{Sh}(\cL^+_{\sigma}M\backslash \on{Gr}_{G,\sigma}),
\end{equation}

\noindent where the action $\curvearrowright$ denotes the usual Hecke action via geometric Satake. Henceforth, we will consider $\on{Sh}(\cL^+_{\sigma}M\backslash \on{Gr}_{G,\sigma})$ as a $(\on{Rep}(\check{M}),\on{Rep}(\check{G}))$-bimodule category, where the action of $\on{Rep}(\check{G})$ is the usual Hecke action for $G$, and the action of $\on{Rep}(\check{M})$ is the twisted action (\ref{eq:twHaction}).\footnote{The reason for considering the twisted action is that it is necessary to make Lemma \ref{l:drpl} true.} As such, we may talk about Drinfeld-Plücker structures and enhanced Drinfeld-Plücker structures in the above sense for objects in $\on{Sh}(\cL^+_{\sigma}M\backslash \on{Gr}_{G,\sigma})$ with respect to this bimodule structure.

\subsubsection{}
The unit section induces a closed embedding:
\[
i_{\sigma}: \bB \cL^+_{\sigma}M:=\cL^+_{\sigma}M\backslash S \into \cL^+_{\sigma}M\backslash \on{Gr}_{G,\sigma}.
\]

\noindent Denote by $\omega_{\bB \cL^+_{\sigma}M/S}$ the $!$-pullback of the constant sheaf on $S$ along the structure map $\bB \cL^+_{\sigma}M\to S$. Let
\[
\delta_{\sigma}:=i_{\sigma,!}(\omega_{\bB \cL^+_{\sigma}M/S})\in \on{Sh}(\cL^+_{\sigma}M\backslash \on{Gr}_{G,\sigma}).
\]

The same proof as in \cite[Appendix B]{faergeman2025parabolic} shows that:
\begin{lem}\label{l:drpl}

The sheaf $\delta_{\sigma}$ is equipped with a canonical Drinfeld-Plücker structure.
\end{lem}

\subsubsection{} Consider the sheaf
\[
\on{ind}_{\on{DrPl}_{\check{M},\check{G}}}^{\on{EnhDrPl}_{\check{M},\check{G}}}(\delta_{\sigma})\in \on{Sh}(\cL^+_{\sigma}M\backslash \on{Gr}_{G,\sigma}),
\]

\noindent cf. §\ref{ss:ind+oblv}. We want to give an explicit description of this sheaf.

Consider $\on{Gr}_{N_P,\sigma}$, the affine Grassmannian for $N_P$ along $\sigma$ along with its locally closed embedding
\[
j: \on{Gr}_{N_P,\sigma}\into \on{Gr}_{G,\sigma.}
\]

\noindent Let $\omega_{\on{Gr}_{N_P,\sigma}/S}$ be the sheaf obtained by $!$-pullback of the constant sheaf on $S$ along the structure map $\on{Gr}_{N_P,\sigma}\to S$. Note that it is simply a shift of the dualizing sheaf on $\on{Gr}_{N_P,\sigma}$.

We let
\[
\mathbf{j}_!^{\on{loc}}:= j_!(\omega_{\on{Gr}_{N_P,\sigma}/S}).
\]

It naturally defines an object of $\on{Sh}(\cL^+_{\sigma}M\backslash \on{Gr}_{G,\sigma})$.

\begin{lem}\label{l:unitmap}

We have a canonical equivalence
\[
j^!(\on{ind}_{\on{DrPl}_{\check{M},\check{G}}}^{\on{EnhDrPl}_{\check{M},\check{G}}}(\delta_{\sigma}))\simeq \omega_{\on{Gr}_{N_P,\sigma}/S}
\]

\noindent in such a way that the natural map
\[
\delta_{\sigma}\to j_!(\omega_{\on{Gr}_{N_P,\sigma}/S})\simeq j_!\circ j^!(\on{ind}_{\on{DrPl}_{\check{M},\check{G}}}^{\on{EnhDrPl}_{\check{M},\check{G}}}(\delta_{\sigma}))\to \on{ind}_{\on{DrPl}_{\check{M},\check{G}}}^{\on{EnhDrPl}_{\check{M},\check{G}}}(\delta_{\sigma})
\]

\noindent is the unit map coming from the $(\on{ind}_{\on{DrPl}_{\check{M},\check{G}}}^{\on{EnhDrPl}_{\check{M},\check{G}}},\on{oblv}_{\on{DrPl}_{\check{M},\check{G}}}^{\on{EnhDrPl}_{\check{M},\check{G}}})$-adjunction.
\end{lem}

\begin{proof}
The analogue of Lemma \ref{l:unitmap} in the pointwise case (i.e., when $S=\on{pt}$) was proven in \cite[Lemma 4.4.8.1 + Prop. 4.5.4.1]{faergeman2025parabolic}. We are going to make repeated use of this in what follows.

For simplicity, we write $\cF:=j^!\circ\on{ind}_{\on{DrPl}_{\check{M},\check{G}}}^{\on{EnhDrPl}_{\check{M},\check{G}}}(\delta_{\sigma})$. It follows from the pointwise case that $\on{ind}_{\on{DrPl}_{\check{M},\check{G}}}^{\on{EnhDrPl}_{\check{M},\check{G}}}(\delta_{\sigma})$ is supported on $\overline{\on{Gr}}_{N_P,\sigma}$, the closure of $\on{Gr}_{N_P,\sigma}$ in $\on{Gr}_{G,\sigma}$; and hence $\on{Gr}_{N_P,\sigma}$ is open in the support of $\on{ind}_{\on{DrPl}_{\check{M},\check{G}}}^{\on{EnhDrPl}_{\check{M},\check{G}}}(\delta_{\sigma})$. In particular, since the latter sheaf is ind-non-characteristic over $S$ (being a colimit of Hecke functors for $M$ and $G$ applied to the non-characteristic sheaf $\delta_{\sigma}$), $\cF$ itself is ind-non-characteristic over $S$.

\step Consider the two maps
\[
p:\on{Gr}_{N_P,\sigma}\to S,\;\; i:S\to \on{Gr}_{N_P,\sigma}
\]

\noindent given by the structure map and the unit section, respectively. We first claim that there is a canonical isomorphism $\cF\simeq p^!i^!(\cF)$.

Applying $i^!$ to the unit map $\cF\to p^!p_!(\cF)$, we get a map
\[
\phi: i^!(\cF)\to p_!(\cF).
\]

\noindent We claim that $\phi$ is an isomorphism. We may check this by checking that $i_s^*\circ \phi: i_s^*\circ i^!(\cF)\to i_s^*\circ p_!(\cF)$ is an isomorphism for all $s\in S$, where $i_s: \lbrace s\rbrace\to S$ is a geometric point.

Note that $i^!(\cF)$ is a colimit of lisse sheaves on $S$. Indeed, by definition, the sheaf $\on{ind}_{\on{DrPl}_{\check{M},\check{G}}}^{\on{EnhDrPl}_{\check{M},\check{G}}}(\delta_{\sigma})$ is a colimit of Hecke functors for $M$ and $G$ applied to $\delta_{\sigma}$. In the notation of \cite[\S 4.4.5]{nadler2021automorphic}, this implies that $\on{ind}_{\on{DrPl}_{\check{M},\check{G}}}^{\on{EnhDrPl}_{\check{M},\check{G}}}(\delta_{\sigma})$ is a colimit of sheaves $*$-pulled back along
\[
\cL^+_{\sigma}M\backslash \on{Gr}_{G,\sigma}\to \frac{M[[z]]\backslash G((z))/G[[z]]}{\on{Aut}(\bC[[z]])}.
\]

\noindent Any sheaf on the right-hand side with compact support on $G((z))/G[[z]]$ descends to $\frac{M[[z]]\backslash G((z))/G[[z]]}{\on{Aut}^n(\bC[[z]])}$, where $\on{Aut}^n(\bC[[z]])$ denotes the smooth group scheme of finite type of automorphisms of $\on{Spec}(\bC[[t]]/(t^n))$. Since the map $\cL^+_{\sigma}M\backslash \on{Gr}_{G,\sigma}\to \frac{M[[z]]\backslash G((z))/G[[z]]}{\on{Aut}^n(\bC[[z]])}$ is smooth, and in particular $*$-restriction coincides with $!$-restriction up to a shift, this proves that $i^!(\cF)$ is a colimit of lisse sheaves. In particular, $i_s^*\circ i^!(\cF)\simeq i_s^!\circ i^!(\cF)[2\on{dim}S]\simeq i^!\circ i_s^!(\cF)[2\on{dim}S]$. Combining with the fact that $\cF$ is non-characteristic over $S$, so that $i_s^*(\cF)\simeq i_s^!(\cF)[2\on{dim}S]$, the pointwise case of \cite{faergeman2025parabolic} implies that $i_s^*\circ \phi: i_s^*\circ i^!(\cF)\to i_s^*\circ p_!(\cF)$ is an isomorphism.

Since $\phi$ is an isomorphism, we get a map
\[
\psi: \cF\to p^!p_!(\cF)\simeq p^!i^!(\cF).
\]

\noindent From the pointwise case, we again know that $i_s^!\circ \psi$ is an isomorphism for all $s\in S$, and hence $\psi$ is an isomorphism.

\step

Note that the unit map $\delta_{\sigma}\to \on{ind}_{\on{DrPl}_{\check{M},\check{G}}}^{\on{EnhDrPl}_{\check{M},\check{G}}}(\delta_{\sigma})$ factors through a map
\begin{equation}\label{eq:natmap1}
\delta_{\sigma}\to \cF.
\end{equation}

\noindent By the above step, (\ref{eq:natmap1}) further factors as:
\[
\delta_{\sigma}\to p^!p_!(\delta_{\sigma})\simeq \omega_{\on{Gr}_{N_P,\sigma}/S}\to \cF.
\]

\noindent As above, to check that the induced map $\omega_{\on{Gr}_{N_P,\sigma}/S}\to \cF$ is an isomorphism, we may do so after applying $i_s^!$ for all $s\in S$ where it follows from \cite{faergeman2025parabolic}.
\end{proof}

\subsubsection{} By Lemma \ref{l:unitmap}, we get a natural map
\begin{equation}\label{eq:mainmap}
\mathbf{j}_!^{\on{loc}}\to\on{ind}_{\on{DrPl}_{\check{M},\check{G}}}^{\on{EnhDrPl}_{\check{M},\check{G}}}(\delta_{\sigma}).
\end{equation}

\begin{cor}\label{c:enhdrpl}

The map (\ref{eq:mainmap}) is an isomorphism.
\end{cor}

\begin{proof}

Using that $i_s^*$ commutes with Hecke functors for $M$ and $G$, we may check that (\ref{eq:mainmap}) is an isomorphism after applying $i_s^*$ for all $s\in S$. Here it follows from \cite[Prop. 4.5.4.1]{faergeman2025parabolic}
    
\end{proof}

\subsection{Enhanced Drinfeld-Plücker structure: global story}

In this subsection, we prove Theorem \ref{t:enhdrplonj}.

\subsubsection{} Consider the map
\[
\on{Bun}_M(X/S)\to \bB \cL^+_{\sigma}M
\]

\noindent given by restricting an $M$-bundle along the formal completion of the section $\sigma: S\to X$. We form the prestack
\[
{}_{\on{Bun}_M}\on{Gr}_{G,\sigma}:=\on{Bun}_M(X/S)\underset{\bB \cL^+_{\sigma}M}{\times} \cL^+_{\sigma}M\backslash \on{Gr}_{G,\sigma.}
\]

\noindent Concretely when $S$ is algebraic, for a test scheme $T$, ${}_{\on{Bun}_M}\on{Gr}_{G,\sigma}(T)$ classifies the data $(s,\sP_G,\sP_M,\alpha)$, where $s:T\to S$ is a map, $\sP_G$ (resp. $\sP_M$) is a $G$-bundle (resp. $M$-bundle) on $T\underset{S}{\times} X$, and $\alpha$ is an identification of $\sP_M\overset{M}{\times}G$ and $\sP_G$ away from the graph of $\sigma\circ s: T\to X$.

Note that we have a natural map
\[
r:{}_{\on{Bun}_M}\on{Gr}_{G,\sigma}\to \widetilde{\on{Bun}}_{P,\infty\cdot \sigma}(X/S)
\]

\noindent sending $(s,\sP_G,\sP_M,\alpha)$ to the corresponding meromorphic Plücker data.

Observe that ${}_{\on{Bun}_M}\on{Gr}_{G,\sigma}$ carries Hecke modifications for $M$ and $G$, and $r$ commutes with these modifications.

\subsubsection{} Let
\[
\on{pr}: {}_{\on{Bun}_M}\on{Gr}_{G,\sigma}\to \cL^+_{\sigma}M\backslash \on{Gr}_{G,\sigma}
\]

\noindent be the projection and set
\[
{}_{\on{Bun}_M}\mathbf{j}_!^{\on{loc}}:=\on{pr}^!(\mathbf{j}_!^{\on{loc}})\in\on{Sh}({}_{\on{Bun}_M}\on{Gr}_{G,\sigma}).
\]

\noindent It is equipped with an enhanced Drinfeld-Plücker structure by Corollary \ref{c:enhdrpl}.

\subsubsection{} Similarly, we let
\[
{}_{\on{Bun}_M}\on{Gr}_{P,\sigma}:=\on{Bun}_M(X/S)\underset{\bB \cL_{\sigma}^+M}{\times} \cL_{\sigma}^+M\backslash \on{Gr}_{N_P,\sigma.}
\]

\noindent It has a natural map
\[
{}_{\on{Bun}_M}j: {}_{\on{Bun}_M}\on{Gr}_{P,\sigma}\to {}_{\on{Bun}_M}\on{Gr}_{G,\sigma}.
\]

\subsubsection{} Let $\omega_{{}_{\on{Bun}_M}\on{Gr}_{P,\sigma}/S}$ be the $!$-pullback of the constant sheaf on $S$ along the projection ${}_{\on{Bun}_M}\on{Gr}_{P,\sigma}\to S$. Note that
\begin{equation}\label{eq:kindaobv}
{}_{\on{Bun}_M}j_!(\omega_{{}_{\on{Bun}_M}\on{Gr}_{P,\sigma}/S})\simeq {}_{\on{Bun}_M}\mathbf{j}_!^{\on{loc}}.
\end{equation}

Indeed, the map $\on{pr}$ is pro-smooth, and so we may reduce the validity of base-change to a finite-dimensional one as in Step 1 of the proof of Lemma \ref{l:unitmap}.

\subsubsection{} We have a natural map
\[
\pi: {}_{\on{Bun}_M}\on{Gr}_{P,\sigma}\to \on{Bun}_P(X/S)
\]

\noindent sitting in a commutative diagram
\[\begin{tikzcd}
	{{}_{\mathrm{Bun}_M}\mathrm{Gr}_{P,\sigma}} && {{}_{\mathrm{Bun}_M}\mathrm{Gr}_{G,\sigma}} \\
	\\
	{\mathrm{Bun}_P(X/S)} && {\widetilde{\mathrm{Bun}}_{P,\infty\cdot \sigma}(X/S).}
	\arrow["{}_{\on{Bun}_M}j", from=1-1, to=1-3]
	\arrow["\pi"', from=1-1, to=3-1]
	\arrow["r", from=1-3, to=3-3]
	\arrow["{j_{\mathrm{pol}}}"', from=3-1, to=3-3]
\end{tikzcd}\]

\subsubsection{} The same proof as \cite[Corollary 5.2.8.3]{faergeman2025parabolic} shows that:

\begin{lem}\label{l:univhom}

The map $\pi$ is universally homologically contractible. That is, for a constructible sheaf $\cF\in \on{Sh}(\on{Bun}_P(X/S))$, the counit map
\[
\pi_!\pi^!(\cF)\to \cF
\]

\noindent is an isomorphism.
    
\end{lem}

\begin{prop}\label{p:loctoglob}

We have a canonical isomorphism:
\[
r_!({}_{\on{Bun}_M}\mathbf{j}_!^{\on{loc}})\simeq \mathbf{j}_!^{\on{glob}}[\on{dim}\on{Bun}_P].
\]
    
\end{prop}

\begin{proof}

Note that $\mathbf{j}_!^{\on{glob}}[\on{dim}\on{Bun}_P]\simeq j_{\on{pol},!}(\omega_{\on{Bun}_P(X/S)/S})$, where $\omega_{\on{Bun}_P(X/S)/S}$ denotes the $!$-pullback of the constant sheaf on $S$ along the projection $\on{Bun}_P(X/S)\to S$.

By Lemma \ref{l:univhom}, we have:
\[
\pi_!(\omega_{{}_{\on{Bun}_M}\on{Gr}_{P,\sigma}/S})\simeq \omega_{\on{Bun}_P(X/S)/S}.
\]

\noindent Combining with (\ref{eq:kindaobv}), we get the desired assertion.
    
\end{proof}

\emph{Proof of Theorem \ref{t:enhdrplonj}}\\

Since ${}_{\on{Bun}_M}\mathbf{j}_!^{\on{loc}}$ is equipped with an enhanced Drinfeld-Plücker structure, by Proposition \ref{p:loctoglob}, so is $\mathbf{j}_!^{\on{glob}}[\on{dim}\on{Bun}_P]$. We remind that these are enhanced Drinfeld-Plücker structures with respect to the twisted action of $\on{Rep}(\check{M})$, cf. §\ref{ss:twistedaction}.

Recall that for $\mu\in \Lambda_{G,P}$:
\[
(\on{dim} \on{Bun}_P)(\mu)=\on{dim}_S \on{Bun}_M(X/S)+\on{dim}_S\on{Bun}_{N_P}(X/S)+\langle 2(\rho_G-\rho_M),\mu\rangle.
\]

\noindent Unwinding the definitions, this means that $\mathbf{j}_!^{\on{glob}}[\on{dim}_S \on{Bun}_M(X/S)+\on{dim}_S\on{Bun}_{N_P}(X/S)]$ is equipped with an enhanced Drinfeld-Plücker structure for the usual Hecke action of $M$ and $G$ on $\on{Sh}(\widetilde{\on{Bun}}_{P,\infty\cdot \sigma})$. Since $\on{dim}_S \on{Bun}_M(X/S)+\on{dim}_S\on{Bun}_{N_P}(X/S)$ is a constant (i.e., independent of $\mu\in \Lambda_{G,P}$), $\mathbf{j}_!^{\on{glob}}$ also carries an enhanced Drinfeld-Plücker structure.

\qed

\subsection{Reduction to polytopes} The purpose of the next two subsections is to prove the following reduction step:

\begin{prop}\label{nilpotenceofCT}
    Let $N$ be a positive integer and suppose that for all $\mu\in \Lambda_{G,P}$ satisfying $\langle \mu,\alpha\rangle < -N$, the functor $\mathrm{CT}_P^\mu$ preserves universal nilpotent singular support. Then $\on{CT}_P$ preserves universal nilpotent singular support.
\end{prop}

\subsubsection{} We will prove the above proposition by induction on $\mu\in \Lambda_{G,P}$. However, exactly what to induct on is not obvious and will ultimately correspond to a solution of a particular \textit{game of polytopes}, which will be described in the next section. We will think of $\Lambda_{G,P}$ as an integral lattice in $V:=\Lambda_{G,P}\otimes_\mathbb{Z} \mathbb{R}$.

Let $C_0$ denote the set of elements $\mu$ in $\Lambda_{G,P}$ where $CT_P^{\mu}$ preserves universal nilpotent singular support. By Proposition \ref{p:goodforsmallmu}, $C_0$ is the integral points of an open affine cone in $V$. The set $C_0$ will serve as the base case of our induction. 

The induction step will be based on a certain set of lattice polytopes in $V$. For each $\nu\in \Lambda_G^+$, let $V^\nu$ be the corresponding irreducible representation of $\check{G}$ with highest weight $\nu$. Let $P_\nu'$ be the image of the set of highest weights for $V^\nu |_{\check{M}}$ under the map\footnote{We add the inversion map to make the statement of Lemma \ref{l:allbutone} nicer.}
\[\Lambda_M^+\to \Lambda_{G,P}\overset{\lambda\mapsto-\lambda}{\to} \Lambda_{G,P}.\]
Let $P_\nu=P_\nu'\cup\{0\}$ in $\Lambda_{G,P}$, and for any $v\in \Lambda_{G,P}$ let $P_\nu+v$, respectively $P_\nu'+v$, denote the translation of the set $P_\nu$, respectively $P_\nu'$, by the vector $v$. The idea of the induction step is as follows: 

\begin{lem}\label{l:allbutone}
    Suppose $\mathrm{CT}_P^\mu$ preserves universal nilpotent singular support for all $\mu\in P_\nu+v$ except possibly one element $\mu_0\ne v$. Then $\mathrm{CT}_P^{\mu_0}$ preserves universal nilpotent singular support.
\end{lem}
\begin{proof}
    If $\mathcal{F}$ is a sheaf with a filtration with associated graded pieces $\mathcal{F}_i$ and $\mathcal{F}$ and all of the sheaves $\mathcal{F}_i$ except at most one has singular support in a cone $N$, then by the triangle inequalities for singular support in a fiber sequence, the remaining sheaf has singular support in $N$. Since $\mu_0\neq v$ and Hecke operators preserve nilpotent singular support, it follows that $\mathrm{CT}_P^v(V^\nu\star-)$ preserves nilpotent singular support. Applying Theorem \ref{t:HeckeCT-formula} to $\mathrm{CT}_P^v(V^\nu\star-)$, we obtain a filtration whose associated graded, up to a shift and a non-zero multiplicity, are
    \[V_{\check{M}}^{\mu'}\star \mathrm{CT}_P^{v-\mu'},\]
    for $\mu'\in P_{\nu}'$. Since Hecke operators preserve universal nilpotent singular support, it follows that all associated graded functors except at most one preserve universal nilpotent singular support. On the other hand, if one of the $V_M^{\mu'}\star \mathrm{CT}_P^{v-\mu'}$ preserves nilptent singular support, then so does $\mathrm{CT}_P^{v-\mu'}$, as it is a summand of 
    \[(V_{\check{M}}^{\mu'})^\vee\star V_{\check{M}}^{\mu'}\star \mathrm{CT}_P^{v-\mu'},\] 
    which preserves universal nilpotent singular support. Thus using the triangle inequalities, it follows that $\mathrm{CT}_P^{\mu_0}$ preserves universal nilpotent singular support.
\end{proof}

Thus starting from $C_0$ we can inductively prove Proposition \ref{nilpotenceofCT} for more and more $\mu$ by choosing the right $v$ and $\nu$ so that Proposition \ref{nilpotenceofCT} is already known for all elements of $P_\nu+v$ except one element of $P_\nu'+v$, thus concluding the statement for one more element of $\Lambda_{G,P}$. Continuing like that, the problem is then reduced to proving that all elements of $\Lambda_{G,P}$ are eventually reached by such a sequence of moves starting from $C_0$. This leads us to consider the abstract setup explained in the next subsection.

\subsubsection{} To prove that every element of $\Lambda_{G,P}$ is eventually reached this way, we need the following input. In the next section we will prove that this is sufficient, see Proposition \ref{round and conic}.

\begin{prop} \label{weightsarejointlyround}
Suppose $G$ is semi-simple. For each functional $f:V=\Lambda_{G,P}\otimes_\mathbb{Z}\mathbb{R}\to \mathbb{R}$, there exist a fundamental weight $\omega$ such that $f$ has a unique maximum on the set $P_\omega$.
\end{prop}

\begin{proof}
    First assume that $P=B$, in this case $\Lambda_{G,P}=\mathbb{X}_\bullet(T)$.  Write $f$ as $\langle v,-\rangle$ for some real coweight $v\in \mathbb{X}_\bullet(T)\otimes_{\mathbb{Z}} \mathbb{R}$. By Weyl group invariance of the inner product, we may assume that $v$ lies in the cone of dominant coweights. By semisimplicity of $G$, choose a fundamental weight $\omega$ such that if $\langle \alpha,v\rangle=0$, for some root $\alpha$, then $\langle \omega,\alpha^\vee\rangle=0$. Since the weights of $V^{\omega}$ are contained in the convex hull of $W\omega$, it suffices to prove that $f$ has a unique maximum on the Weyl orbit of $\omega$. Suppose $\langle \omega,v\rangle =\langle w\cdot \omega,v\rangle$ for some $w\in W$. We have to prove that $w\cdot\omega=\omega$. We will prove the statement by induction on the length of $w$.

    If $l(w)=0$, the statement is trivial. For the induction step, suppose the statement is true for all $w'$ with $l(w')< l(w)$. Choose a simple reflection $s_\alpha$ such that $w=w's_\alpha$, for $\alpha$ a simple root and $w'$ a weyl element with $l(w')=l(w)-1$. Then
    \[0=\langle \omega-w\cdot \omega,v\rangle=\langle \omega-w'\cdot \omega,v\rangle  + \langle w'\cdot \omega -w\cdot \omega,v\rangle=\langle \omega-w'\cdot \omega,v\rangle  + \langle \omega,\alpha^\vee \rangle \langle  w'\cdot \alpha ,v\rangle\]
    Since $\omega$ is dominant, $w'\cdot \omega\le \omega$, thus since $v$ is dominant it follows that the first term is non-negative. By \cite[Prop. 4.4.6]{bjorner2005combinatorics}, the root $w'\cdot \alpha$ is positive and thus the second term is also non-negative. From this we conclude that $\langle \omega-w'\cdot \omega,v\rangle =0$ and $\langle \omega,\alpha^\vee \rangle \langle  w'\cdot \alpha ,v\rangle=0$. From the induction hypothesis, we conclude that $\omega=w'\cdot \omega$. Moreover, either $\langle \omega,\alpha^\vee \rangle=0$ or $\langle  w'\cdot \alpha ,v\rangle=0$.  If $\langle  w'\cdot \alpha ,v\rangle=0$, then by the assumption on $\omega$, it follows that $0=\langle \omega, w'\cdot \alpha^\vee \rangle=\langle \omega,\alpha^\vee \rangle$. In any case, we conclude that $0=\langle \omega,\alpha^\vee\rangle$. Hence
    \[ \omega- w\cdot \omega =  \langle \omega,\alpha^\vee\rangle w'\cdot \alpha =0, \]
which finishes the induction.

For a general parabolic $P$ consider the map $\Lambda\to \Lambda_{G,P}$. Let $P_{\omega,B}\subset \Lambda$ denote the set of all weights of the irreducible highest weight representation $V^\omega$. Then $P_\omega\subset \Lambda_{G,P}$ is the image of $P_{\omega,B}$ under the quotient map. Thus the statement for a general parabolic reduces to the statement for a Borel subgroup, by pulling back the functional along the quotient map.
    
\end{proof}

\subsection{A game of polytopes}\label{s:gameofpolytopes}

\subsubsection{} Let $V$ be a finite dimensional real vector space and $L$ be an integral lattice in $V$.

\begin{defin}
    A \textit{lattice polytope} in $(V,L)$ is the convex hull of a finite subset of $L$.
\end{defin}

Suppose $\mathcal{P}$ is a finite collection of lattice polytopes in $(V,L)$. Given this we will play the following one-player game: 
\begin{itemize}
    \item A state of the game is a subset $C$ of $L$ consisting of vertices that have been “coloured”.
    \item Given a state $C$, a move consists of a polytope $P\in \mathcal{P}$ and an element $v\in L$ such that all lattice elements of $v+P$ except one vertex $x\neq v$, is contained in $C$. This move brings you to the new state $C\cup \{x\}$.
\end{itemize}

Given a state $C_0$, we say that a lattice point $x$ can be coloured starting from $C_0$ if there exists a sequence of moves and states starting from the state $C_0$ 
\[C_0\overset{(P_0,v_0)}{\to} C_1 \overset{(P_1,v_1)}{\to}\cdots \overset{(P_{n-1},v_{n-1})}{\to} C_n\]
such that $x\in C_n$. We are interested in finding sufficient conditions on $C_0$ such that all lattice points can be coloured from $C_0$.

\begin{defin}
    The set $\mathcal{P}$ is called \textit{jointly round} if for all non-zero functionals $f:V\to \mathbb{R}$, there exists a polytope $P\in \mathcal{P}$ such that the function $f|_{P}$ has a unique maximum. We refer to the point realizing the maximum as the extremal point of $P$ in the direction of $f$.
\end{defin}

By an open cone in $V$, we mean an open subset $\sigma$ stable under $\bR_{>0}$. Our goal is to prove the following.

\begin{prop}\label{round and conic}
    Let $V$ be a finite dimensional real vector space, $L$ a lattice in $V$ and $\mathcal{P}$ a finite set of lattice polytopes. If $\mathcal{P}$ is jointly round and $C_0$ contains $\sigma\cap L$, where $\sigma$  is an open cone of $V$, then every point of $L$ can be coloured starting from $C_0$.
\end{prop}

We will prove this proposition by proving another sufficient condition, which requires us to choose an inner product $\langle -,-\rangle$ on $V$.

\begin{prop}\label{round and sufficiently large}
    Let $V$, $L$, $\mathcal{P}$ and $\langle -,-\rangle$ be given. If $\mathcal{P}$ is finite and jointly round, then there exists a constant $d\in \mathbb{R}_{>0}$ such that if $C_0$ contains all lattice points in a ball of radius $d$, then every lattice point can be coloured starting from $C_0$.
\end{prop}

If a set contains an open cone, then it contains balls of arbitrarily large radius, thus we immediately conclude the first proposition form the second.

To prove the second proposition, we will need a few lemmas about convex geometry.

\begin{lem}
    Let $V$, $\mathcal{P}$ and $\langle -,-\rangle$ be given. If $\mathcal{P}$ is finite and jointly round, there exists an angle $\theta\in (0,\pi/2]$ such that for all non-zero functionals $f:V\to \mathbb{R}$, there exists a polytope $P\in \mathcal{P}$ such that the angle between the hyperplane $f=0$ and any hyperplane supporting a positive dimensional face of $P$ is bounded below by $\theta$.
\end{lem}

\begin{proof}
    Let $SV\subset V$ denote the unit sphere. The statement is invariant under changing the functional $f$ by a positive multiple, hence we may consider the functionals of the form $f_v:=\langle v,-\rangle$ for a unit vector $v\in SV$. For a given polytope $P\in \mathcal{P}$, let $K_P$ denote the closed set of unit vectors $v\in SV$ such that the functional $f_v$ does not have a unique maximum on $P$. By the assumption that $\mathcal{P}$ is jointly round,
    \[\bigcap_{P\in \mathcal{P}} K_P=\emptyset.\]
    Since the unit sphere in $V$ is a normal space, it follows that there exist open neighborhoods $U_P$ of $K_P$ such that
    \[\bigcap_{P\in \mathcal{P}} U_P=\emptyset.\]
    The topology on the unit sphere is metrizable with metric given by the angle between two unit vectors. Since $K_P$ is compact, the distance in the angle metric from $K_P$ to $SV\setminus U_P$, denote it $\theta_P$, is strictly positive. Let $\theta$ be the minimum of $\theta_P$, where $P$ ranges over $\mathcal{P}$.

    Since the sets $SV\setminus U_P$ form a cover of the unit sphere, given a functional $f_v$ there exist a polytope $P_v\in \mathcal{P}$ such that $v\in SV\setminus U_{P_v}$. We will now prove the lower bound on the angle for any face of this polytope $P_v$ and the hyperplane given by the functional $f_v$. The set $K_{P_v}$ is exactly the subset of functionals that support a positive dimensional face of $P_v$. The angle $\theta$ was chosen to be at most the distance from $v$ to $K_{P_v}$ and thus in particular smaller than the angle between $f_v=0$ and the supporting hyperplanes of the faces of $P_v$. 
\end{proof}

\begin{lem}
    Let $V$ be a finite dimensional Euclidean vector space, $P$ be a polytope in $V$ and $H$ be a linear hyperplane. Suppose that $H$ intersects $P$ in a unique point $x_0$. If the angle between $H$ and any facet of $P$ containing $x$ is bounded below by $\theta\in (0,\pi/2]$, then for all $x\in P$ the angle between $H$ and the vector from $x_0$ to $x$ is bounded below by $\theta$.
\end{lem}

\begin{proof}
    We may center the picture so that $x_0=0$. Let $x_H$ denote the orthogonal projection of $x$ onto $H$. The line segment from $x$ to $x_H$ intersects a hyperplane supported by a facet of $P$ adjecent to the vertex $x_0$. Let $v$ denote this intersection point. Then the angle between $x$ and $x_H$ is the angle between $x$ and $v$ plus the angle between $v$ and $x_H$. The angle between $v$ and $x_H$ is bounded below by $\theta$ and the angle between $x$ and $v$ is positive, hence the statement. 
\end{proof}

\begin{lem}
    Let $V$, $L$, $\mathcal{P}$ and $\langle -,-\rangle$ be given. If $\mathcal{P}$ is finite and jointly round, there exist a $d\in \mathbb{R}_{>0}$ such that for all $d'\ge d$ and any vector $v\in L$ with $||v||=d'$, there exist a polytope $P\in \mathcal{P}$ and a lattice element $t\in L$ such that $v$ belongs to $t+P$ and so that $(t+P)\setminus\{v\}$ is contained in the open ball of radius $d'$ centered at $0$.
\end{lem}

\begin{proof}
   Let $\theta$ be chosen as in the previous lemma. The number $d$ will be chosen to solve the following geometric problem. Let $r$ denote the maximum diameter of the polytopes in $\mathcal{P}$. Choose any hyperplane $H$ in $V$ with a point $x_0\in H$ and a half space $H_+$; the number $d$ will be independent of these choices. Consider the set $K$ of vectors $x$ in $H_+$ such that the angle between $H$ and the vector from $x_0$ to $x$ is bounded below by $\theta$ and the distance between $x_0$ and $x$ is less than $r$. Choose $d$ to be any number such that the closed ball inside $H_+$ of radius $d$ tangent to $H$ at $x_0$ contains $K$. It is clear that such a $d$ exists, since $K$ is compact and such balls tangent to $H$ at $x_0$ forms an increasing family whose union is $H_+$. Since affine isometries act transitively on the choices involved in defining $d$, this number is independent of the hyperplane $H$, the point $x_0$ and the half-space $H_+$.

   Now let $d'\ge d$ be given and $v\in L$ be a lattice point with $||v||=d'$.  Since $\mathcal{P}$ is jointly round, there exist a polytope $P\in \mathcal{P}$ such that the functional $f_v:=\langle v,-\rangle$ has a unique maximum on $P$, say at a vertex $m\in P$. By the previous two lemmas we can further assume that the angle between the vector from $m$ to any point of $P$ and the hyperplane $f_v=0$ is bounded below by  $\theta$. Since $P$ is assumed to be a lattice polytope in fact $m\in L$. Thus the choice $t=v-m$ is an element of $L$. Clearly $v\in t+P$, thus it remains to show that every element of $(t+P)\setminus\{v\}$ has norm strictly less than $d'$.

   The hyperplane $f_v=d'^2$ passes through $v$ and $t+P$ is contained in the half-space $f_v\le d'^2$. By assumption every angle between the hyperplane $f_v=d'^2$ and the vector from $v$ to a point of $t+P$ is at least $\theta$. Hence $t+P$ is contained in the set $K$ used to define $d$, where we choose $H$ to be $f_v=d'^2$, $x_0=v$ and $H_+$ to be the set $f_v\le d'^2$. By construction of $d$ this implies that $t+P\subset K$ is contained in the closed ball of radius $d$ tangent to $H$ at $v$, thus in particular it is contained in the same ball of radius $d'\ge d$. This ball is exactly the closed ball of radius $d'$ centered at $0$. Since $v$ is the only intersection point between $K$ and the ball of radius $d'$, it follows that $v$ is the only point of $P$ on the boundary of the ball of radius $d'$ centered at $0$. Hence $(t+P)\setminus \{v\}$ is contained in the open ball of radius $d'$ centered at $0$.
\end{proof}

With these lemma we now conclude that for jointly round sets of lattice polytopes, every point can be coloured if the starting state contains a ball of sufficiently large radius.

\begin{proof}[Proof of Proposition \ref{round and sufficiently large}]
Let $d$ be chosen as in the previous lemma. We may assume that the ball of radius $d$, which is contained in $C_0$, is centered at the origin. The image of the norm map $||\cdot||:L\to \mathbb{R}_{\ge 0}$ forms a discrete subset. We will prove that every lattice point can be coloured by induction on the distance from the origin. We assume that the statement is known for $d\in \mathbb{R}_{>0}$, which gives the base case of the induction. For the induction step we suppose that every lattice point in the interior of a ball of radius $d'$ can be coloured, for some $d'\ge d$. We have to prove that the lattice points on the boundary can also be coloured. Let $v$ be a lattice point on the boundary of the ball of radius $d'$. By the previous lemma we can choose a polytope $P\in \mathcal{P}$ and a lattice element $t\in L$ such that $v\in t+P$ and $(t+P)\setminus \{v\}$ is contained in the interior of this ball. Since the interior can be coloured given $C_0$ by the induction hypothesis, the move given by $P$ and $t$ shows that $v$ can also be coloured given $C_0$. Thus all points on the boundary of the ball of radius $d'$ centered at $0$ can be coloured, which proves the induction step. Thus all lattice points can be coloured given $C_0$.
\end{proof}

\begin{proof}[Proof of Proposition \ref{nilpotenceofCT}]

For a reductive group $G$, let $V=\Lambda_{G,P}\otimes_{\bZ}\bR$, $L=\Lambda_{G,P}$ and $\cP=\lbrace \mathrm{conv}P_{\omega}\rbrace$, where $\mathrm{conv}P_{\omega}$ denotes the convex hull of $P_\omega$ and $\omega$ ranges over the fundamental weights. Let $C_0\subset \Lambda_{G,P}$ be the subset consisting of $\mu$ such that $\on{CT}_P$ preserves universal nilpotent singular support.

We refer to this game as the \emph{game of polytopes associated to $G$}. By Lemma \ref{l:allbutone}, solving the game of polytopes for $G$ implies Proposition \ref{nilpotenceofCT}. It is not difficult to see that solving the game of polytopes associated to $G$ reduces to solving the game of polytopes associated to its adjoint quotient $G^\mathrm{ad}$. When $G$ is semi-simple, Proposition \ref{weightsarejointlyround} implies that $\cP$ is jointly round. By assumption, $C_0$ contains $\sigma\cap L$, where $\sigma\subset V$ is an open cone. We conclude by Proposition \ref{round and conic}.
\end{proof}

\subsubsection{} We end this subsection by proving Theorem \ref{t:csttermpreservesnilp} in the principal case $P=B$:

\begin{thm}\label{t:princcase}

The functor $\on{CT}_B$ preserves universal nilpotent singular support.
    
\end{thm}

\begin{proof}

This follows by combining Proposition \ref{p:goodforsmallmu} and Proposition \ref{nilpotenceofCT}.

\end{proof}

\subsection{Proof for a general parabolic}

In this subsection, we finish the proof of Theorem \ref{t:csttermpreservesnilp} for a general parabolic. Thus, fix a parabolic $P$ properly contained in $G$ with Levi $M$. By Theorem \ref{t:princcase}, we assume that $P$ is not a Borel subgroup.

We assume the following induction hypotheses:
\begin{enumerate}
    \item For any parabolic $Q$ properly contained in $P$, $\on{CT}_Q$ preserves universal nilpotent singular support;

    \item For a parabolic $R$ of $M$, the functor
    \[
    \on{CT}_R: \on{Sh}(\on{Bun}_M(X/S))\to \on{Sh}(\on{Bun}_L(X/S))
    \]

    \noindent preserves universal nilpotent singular support.
\end{enumerate}

\subsubsection{} Fix $\mu\in\Lambda_{G,P}$ as in the statement of Proposition \ref{p:goodforsmallmu}. By Proposition \ref{nilpotenceofCT}, it suffices to prove that $\on{CT}_P^{\mu}$ preserves universal nilpotent singular support. Applying Proposition \ref{p:goodforsmallmu}, we let $j:U\subset \on{Bun}_M^{\mu}(X/S)$ be an open substack given by a union of Harder-Narasimhan strata such that
\[
\on{Bun}_M^{\mu}(X/S)^{(\underset{M_{\on{ad}}}{\leq}(2g-2)\rho_M)}\subset U,
\]

\noindent and such that the functor $j^*\circ \on{CT}_P^{\mu}$ preserves universal nilpotent singular support. Let $i:Z\subset \on{Bun}_M^{\mu}(X/S)$ be the complement of $U$ in $\on{Bun}_M^{\mu}(X/S)$. 

We need the following theorem proved in \cite{arinkin2020stack}:\footnote{The result was proved when $S=\on{pt}$ but the same proof works in the relative setting.}
\begin{thm}\label{t:relagkrrv}[{\cite[Thm. 19.1.3]{arinkin2020stack}}]

The functor
\[
j_!: \on{Sh}(U)\to \on{Sh}(\on{Bun}_M(X/S))
\]

\noindent restricts to a functor
\[
j_!: \on{Sh}_{\on{Nilp}_{u,M}}(U)\to \on{Sh}_{\on{Nilp}_{u,M}}(\on{Bun}_M(X/S)).
\]
    
\end{thm}

\subsubsection{} Henceforth, fix $\cF \in\on{Sh}_{\on{Nilp}_{u,G}}(\on{Bun}_G(X/S))$. Let $\cG:=i_*i^*(\on{CT}_P(\cF))$. From the above theorem and the triangle
\[
j_!j^*\to \on{id}\to i_*i^*,
\]

\noindent it suffices to prove that $\on{SS}(\cG)\subset \on{Nilp}_{u,M}$. More precisely, let $\on{Bun}_M(X/S)^{(\lambda)}$ be a Harder-Narasimhan stratum for some $\lambda\in \Lambda_M^{+,\bQ}$ whose intersection with $\on{Bun}_M^{\mu}(X/S)$ is open in $Z$, and let $\xi\in \on{SS}(\cG)$ lie over $\on{Bun}_M(X/S)^{(\lambda)}$. We need to show that $\xi$ is nilpotent. The idea is that because $\xi$ is sufficiently far out in the Harder-Narasimhan stratification, we may further reduce to a smaller parabolic, where we may apply the induction hypothesis.

\subsubsection{} Let $R\subset M$ be the parabolic with Levi $L$ whose simple roots $\alpha$ are exactly those that satisfy $\langle \lambda,\alpha\rangle \leq 2g-2$. Note that $R$ is properly contained in $M$ since we assumed that $M$ was not a torus (otherwise, we would have $\lambda\underset{M_{\on{ad}}}{\leq} (2g-2)\rho_M$).

Let
\[
f': \on{Bun}_R(X/S)\to \on{Bun}_M(X/S)
\]

\noindent be the evident map. We write $\on{Bun}_L(X/S)^{(\lambda)}$ for the Harder-Narasimhan stratum in $\on{Bun}_L(X/S)$ corresponding to $\lambda$.

Let $A\subset \Lambda_{M,R}$ be the (finite) set corresponding to components of $\on{Bun}_L(X/S)$ that intersects $\on{Bun}_L(X/S)^{(\lambda)}$. Concretely, these are the $\mu\in\Lambda_{M,R}$ such that the image of $\mu$ along
\[
\Lambda_{M,R}\to \Lambda_{M,R}^{\bQ}
\]

\noindent coincides with the image of $\lambda$ under
\[
\Lambda^{\bQ}\to \Lambda_{M,R}^{\bQ}.
\]

Write
\[
\on{Bun}_R^A(X/S)
\]

\noindent for the union of components indexed by $A$, and let
\[
f: \on{Bun}_R^A(X/S)\to \on{Bun}_M(X/S)
\]

\noindent be the restriction of $f'$ to $\on{Bun}_R^A(X/S)$. Let
\begin{equation}\label{eq:defofg}
g: \on{Bun}_R^A(X/S)\to \on{Bun}_L(X/S)
\end{equation}

\noindent be the evident map. Let
\[
f^{-1,A}(\on{Bun}_M(X/S)^{(\lambda)}):=f^{-1}(\on{Bun}_M(X/S)^{(\lambda)})\cap \on{Bun}_R^A(X/S).
\]

It is easy to see that we have an inclusion (see e.g. \cite[Lemma B.3.2]{drinfeld2016geometric}):
\[
g^{-1}(\on{Bun}_L(X/S)^{(\lambda)})\subset f^{-1,A}(\on{Bun}_M(X/S)^{(\lambda)}).
\]

\begin{lem}\label{l:incisan=}

The inclusion
\[
g^{-1}(\on{Bun}_L(X/S)^{(\lambda)})\subset f^{-1,A}(\on{Bun}_M(X/S)^{(\lambda)})
\]

\noindent is an equality.
\end{lem}

\begin{proof}

If $X$ has genus $g=0$, then $R$ is a Borel subgroup and $\on{Bun}_R^A(X/S)=g^{-1}(\on{Bun}_L(X/S)^{(\lambda)})$, so the assertion is tautological.

Assume that $g>0$. Let $\sP_R\in f^{-1,A}(\on{Bun}_M(X/S)^{(\lambda)})\subset \on{Bun}_R(X/S)$. Let $\sP_L,\sP_M$ be the induced $L$ and $M$-bundles. Let $\nu\in \Lambda_L^{+,\bQ}$ be the Harder-Narasimhan type of $\sP_L$; i.e., $\sP_L\in \on{Bun}_L(X/S)^{(\nu)}$. We need to show that $\nu=\lambda$.

\step We first show that $\nu\underset{L}{\leq}\lambda$. Let $Q_L\subset L$ be the Harder-Narasimhan parabolic associated to $\nu$ (i.e., the standard parabolic associated to the subset of simple roots $\alpha$ of $L$ satisfying $\langle \nu,\alpha\rangle=0$), and let $\sP_{Q_L}$ be the canonical reduction of $\sP_L$ to $Q_L$ (cf. \cite[Thm. 7.4.3]{drinfeld2015compact}). Denote by $K$ the Levi quotient of $Q_L$.

Let $Q_M\subset M$ be the parabolic of $M$ whose Levi quotient coincides with $K$. That is, the parabolic defined by the Cartesian diagram:
\[\begin{tikzcd}
	{Q_M} && R \\
	\\
	{Q_L} && L.
	\arrow[from=1-1, to=1-3]
	\arrow[from=1-1, to=3-1]
	\arrow[from=1-3, to=3-3]
	\arrow[from=3-1, to=3-3]
\end{tikzcd}\]

In particular, the triple $(\sP_{Q_L},\sP_L,\sP_R)$ gives a $Q_M$-bundle $\sP_{Q_M}$ whose induction to $R$ coincides with $\sP_R$. Since $\sP_M$ has Harder-Narasimhan type $\lambda$, it follows from \cite[Thm. 4.5.1]{schieder2015harder} that $\nu\underset{M}{\leq}\lambda$. Moreover, since $\sP_R\in \on{Bun}_R^A(X/S)$, we get that the images of $\nu$ and $\lambda$ under
\[
\Lambda^{\bQ}\to \Lambda_{M,R}^{\bQ}
\]

\noindent coincide. Thus, $\nu\underset{L}{\leq}\lambda$.

\step Next, we prove that $\lambda\underset{M}{\leq} \nu$, which together with Step 1 proves the lemma.

We claim that $\nu$ is $M$-dominant and $Q_M$-regular. For $\alpha\in \Gamma_K$, we have $\langle \nu,\alpha\rangle=0$, and for $\alpha\in \Gamma_L-\Gamma_K$, we have $\langle \nu,\alpha\rangle>0$ by definition of $Q_L$. For $\alpha\in\Gamma_M-\Gamma_L$, we have $\langle \nu,\alpha\rangle>0$ since $\lambda$ is $M$-dominant and $\nu\underset{L}{\leq} \lambda$ by Step 1.

Applying \cite[Thm. 4.5.1]{schieder2015harder} again, we see that $\lambda\underset{M}{\leq} \nu$, as desired.

\end{proof}

\begin{lem}\label{l:(1)}

We have an equality
\[
\on{SS}(f^*(\cG))\vert_{f^{-1,A}(\on{Bun}_M(X/S)^{(\lambda)})}=df(\on{SS}(\cG))\vert_{f^{-1,A}(\on{Bun}_M(X/S)^{(\lambda)})}.
\]
    
\end{lem}

\begin{proof}

Since $\on{Bun}_M(X/S)^{(\lambda)}$ is open in the support of $\cG$, it suffices to show that
\begin{equation}\label{eq:fisaniso}
f: f^{-1,A}(\on{Bun}_M(X/S)^{(\lambda)})\to \on{Bun}_M(X/S)^{(\lambda)}
\end{equation}

\noindent is smooth. Suppose first that the genus of $X\to S$ is zero, so that $R$ is a Borel subgroup and $L$ is the Cartan. Then $\on{Bun}_M(X/S)^{(\lambda)}$ and $f^{-1}(\on{Bun}_M(X/S)^{(\lambda)})$ are both classifying stacks for families of smooth algebraic groups $S$ (for the latter, this follows because the fiber of $\on{Bun}_R^{\lambda}(X/S)\to \on{Bun}_M^{\lambda}(X/S)$ consists of a single isomorphism class of geometric point whenever $\lambda$ is dominant).

Next, suppose that the genus is $>0$. In this case, we claim that $(\ref{eq:fisaniso})$ is an isomorphism. Indeed, this follows from \cite[Lemma B.3.2(3)]{drinfeld2016geometric} and Lemma \ref{l:incisan=}.
    
\end{proof}

\subsubsection{} Let $\on{Nilp}_{u,R}$ be the set-theoretic image of the codifferential
\[
df': \on{Nilp}_{u.M}\underset{\on{Bun}_M(X/S)}{\times} \on{Bun}_R(X/S)\to  T^*\on{Bun}_R(X/S).
\]

We say that a covector $\xi_P\in T^*\on{Bun}_R(X/S)$ is nilpotent if it lies in $\on{Nilp}_{u,R}$.

\begin{cor}\label{c:laststep}

We have $\on{SS}(\cG)\vert_{\on{Bun}_M(X/S)^{(\lambda)}}\subset \on{Nilp}_{u,M}$.
    
\end{cor}

\begin{proof}

Write
\[
\on{CT}^M_R: \on{Sh}(\on{Bun}_M(X/S))\to \on{Sh}(\on{Bun}_L(X/S))
\]

\noindent for the constant term functor associated to $R$. Consider the map
\[
g: \on{Bun}_R(X/S)\to \on{Bun}_L(X/S).
\]

Hence
\[
\on{CT}^M_R(\cG)\vert_{\on{Bun}_L(X/S)^{(\lambda)}}\simeq g_!f^*(\cG)\vert_{\on{Bun}_L(X/S)^{(\lambda)}}.
\]

\step First, we claim that
\begin{equation}\label{eq:firstest}
\on{SS}(g_!f^*(\cG))\vert_{\on{Bun}_L(X/S)^{(\lambda)}}
\end{equation}

\noindent lies in $\on{Nilp}_{u,L}$. Indeed, note that $\on{CT}^M_R\circ \on{CT}_P\simeq \on{CT}_{R'}$, where $R'\subset G$ is the parabolic whose Levi coincides with that of $R$. As noted above, $j_!j^*(\on{CT}_P(\cF))$ has nilpotent singular support. As such, applying the induction hypothesis (1)+(2) (using that $R$ is a proper parabolic of $M$), the two first terms of the triangle
\[
\on{CT}^M_R(j_!j^*(\on{CT}_P(\cF)))\to \on{CT}^M_R\circ\on{CT}_P(\cF)\to \on{CT}^M_R(\cG)
\]

\noindent have singular support contained in $\on{Nilp}_{u,L}$. This shows that (\ref{eq:firstest}) is nilpotent.

\step

By \cite[Lemma B.3.5]{drinfeld2016geometric}, the map
\[
g: f^{-1,A}(\on{Bun}_M(X/S)^{(\lambda)}) \to \on{Bun}_L(X/S)^{(\lambda)}
\]

\noindent satisfies that its fibers are unipotent gerbes and that $g_!$ induces an equivalence on sheaves. By Lemma \ref{l:(1)} and Step 1, this shows that for $\sP_R\in f^{-1,A}(\on{Bun}_M(X/S)^{(\lambda)})$ and $\xi\in \on{SS}(\cG)\vert_{\on{Bun}_M(X/S)^{(\lambda)}}$ lying over $\sP_M=\sP_R\overset{R}{\times}M$, the image of $\xi$ under
\[
T^*_{\sP_M} \on{Bun}_M(X/S)\to T^*_{\sP_R} \on{Bun}_R(X/S)
\]

\noindent lies in the image of
\[
\on{Nilp}_{u,L}\underset{\on{Bun}_L(X/S)}{\times} \lbrace \sP_L\rbrace \to T^*_{\sP_R} \on{Bun}_R(X/S),
\]

\noindent where $\sP_L=\sP_R\overset{R}{\times}L$. From the definition of the universal nilpotent cone in \S \ref{ss:nilpudef}, this shows that $\xi$ lies in $\on{Nilp}_{u,M}$.
\end{proof}

\part{}

This part begins our study of the Whittaker functional for a family of curves $X/S$. Our goal in the next few sections is to show that the Whittaker functional sends nilpotent automorphic sheaves to locally constant sheaves on $S$. In the case of a curve over a point, Nadler and Taylor proved that the Whittaker functional is a shifted microstalk. We will prove the analogous statement for a family of curves, namely that the Whittaker functional is given by a relative microstalk, in the sense that we will define in the next section. We then prove a general criterion for when relative microstalks form local systems. From this criterion, we conclude that the Whittaker coefficients of a nilpotent automorphic sheaf forms a local system on the base.

\section{The Whittaker Functional}\label{S:Whittakerfunctional}

    In this section we set up the definitions of the Poincaré sheaf and the Whittaker functional in the context of a family of curves.

\subsection{The Poincare sheaf}\label{ss:poinc} Throughout, let $f:X \to S$ be a family of smooth proper complex curves over a real analytic manifold $S$, and assume we have a section $x:S\to X$ and a square root $\Omega^{1/2}$ of $\Omega_{X/S}$. Let $\check{\rho}$ be the half-sum of positive coroots of $G$ and $\check{\rho}(\Omega):=2\check{\rho}(\Omega^{1/2})$ denote the induced $T$-bundle on $X/S$.

Consider the $T$-bundle $\check{\rho}(\Omega)$ as a section $\check{\rho}(\Omega):S\to \Bun_T(X/S)$ and let $\Bun_N^\Omega(X/S)$ denote the fiber product
\[\Bun_N^{\check{\rho}(\Omega)}(X/S):\simeq \Bun_B(X/S)\times_{\Bun_T(X/S)} S.\]
This stack admits a forgetful map
\begin{equation}\label{eq:NtoBtoG}
a:\Bun_N^{\check{\rho}(\Omega)}(X/S)\to \Bun_B^{(g-1)2\check{\rho}}(X/S)\to \Bun_G(X/S).
\end{equation}

For each simple root $\alpha:T\to \mathbb{G}_m$, $\langle\check{\rho},\alpha\rangle=1$ and thus the induced $\mathbb{G}_m$-bundle $\alpha\circ\check{\rho}(\Omega)$ is isomorphic to $\Omega$. Using Serre duality, we obtain a map
\[\chi_\alpha:\Bun_N^{\check{\rho}(\Omega)}(X/S)\to \mathbb{V}(f_*\mathcal{O}_X)\simeq \mathbb{A}^1_S,\]
where $\mathbb{V}(f_*\mathcal{O}_X)$ denotes the line bundle $f_*\mathcal{O}_X$ considered as an affine space over $S$. Let $r$ be the semi-simple rank of $G$, combining all simple roots of $G$ gives a map
\[\chi:\Bun_N^{\check{\rho}(\Omega)}(X/S)\to \mathbb{A}^r_S.\]
The rational cocharacter $\check{\rho}$ defines an action of $\mathbb{G}_m$ on $B$ by conjugation that acts trivially on $T\subset B$. This in turn defines an action of $\mathbb{G}_m$ on the stack $\Bun_N^{\check{\rho}(\Omega)}(X/S)$. Let $+:\mathbb{A}^r_S\to \mathbb{A}^1_S$ denote the sum of coordinates map. Composing $\chi$ with $+$ we obtain a map which we denote $\varphi$.
This map is $\mathbb{G}_m$-equivariant, where $\mathbb{G}_m$ acts on the target by the weight $1$ scaling action, and hence it descends to a map
\[\varphi/\mathbb{G}_m:\Bun_N^{\check{\rho}(\Omega)}(X/S)/\mathbb{G}_m\to \mathbb{A}^1_S/\mathbb{G}_m.\]

\begin{defin}
   Let $\exp_!/\mathbb{G}_m\in \mathrm{Sh}(\mathbb{A}_S^1/\mathbb{G}_m)$ denote the sheaf given by $j_*\mathbf{1}[-1]$, where $j$ is the open embedding $j:S\simeq (\mathbb{A}_S^1\setminus \{0\})/\mathbb{G}_m\to \mathbb{A}_S^1/\mathbb{G}_m$.
\end{defin}

\begin{defin}\label{the poincare sheaf}
    The \textit{Poincaré sheaf} is the sheaf on $\Bun_G(X/S)$ given by
    \[\mathrm{Poinc}_!:\simeq (a/\mathbb{G}_m)_!(\varphi/\mathbb{G}_m)^*\mathrm{exp}_!/\mathbb{G}_m\in \mathrm{Sh}(\Bun_G(X/S)),\]
    where $a/\mathbb{G}_m$ is the map
    \[a/\mathbb{G}_m:\Bun_N^{\check{\rho}(\Omega)}(X/S)/\mathbb{G}_m\to \Bun_G.\]
    We note that the $\textit{Poincaré sheaf}$ is constructible.
    
    The \textit{Whittaker functional} is the functor
    \[\mathrm{Whit}_{X/S}:\mathrm{Sh}(\Bun_G(X/S))\to \mathrm{Sh}(S)\]
    given by $\Psi_*\mathcal{H}om(\mathrm{Poinc}_!,-)$. Recall that $\Psi:\mathrm{Bun}_G(X/S)\to S$ is the structure map.
\end{defin}

The objective of the next few sections is to prove the following local constancy result for the Whittaker functional.

\begin{prop}\label{Localconstancyofwhittaker}
    For any nilpotent automorphic sheaf $\mathcal{F}\in \mathrm{Sh}_{\mathrm{Nilp}_u}(\Bun_G(X/S))$, the Whittaker functional $\mathrm{Whit}_{X/S}\mathcal{F}$ is a local system on $S$. Furthermore, there is an identification of the stalk of $\mathrm{Whit}_{X/S}\mathcal{F}$ at a point $s$ in $S$ with $\mathrm{Whit}_{X_s} i_s^*\mathcal{F}$.
\end{prop}

The Whittaker functional can also be described in terms of vanishing cycles. Taking vanishing cycles with respect to $\varphi$, considered as a holomorphic function, defines a functor
\[\Phi_\varphi:\mathrm{Sh}(\Bun_N^{\check{\rho}(\Omega)}(X/S))\to \mathrm{Sh}(\varphi^{-1}(0)).\]
For $\mathbb{G}_m$-equivariant sheaves, the vanishing cycles with respect to $\varphi$ is naturally equivalent to the real vanishing cycles with respect to the $C^\infty$-function $\operatorname{Re} \varphi$, this follows by arguing as in \cite[VIII Exercise 13]{kashiwara1990sheaves} using that $\varphi_*\mathcal{F}$ is weakly constructible for such sheaves. We will freely move between the complex analytic and the smooth perspective.

Note that $\varphi\circ \check{\rho}(\Omega)=0$, hence $\check{\rho}(\Omega):S\to \Bun_N^{\check{\rho}(\Omega)}(X/S)$ factors through $\varphi^{-1}(0)$. Write $\check{\rho}(\Omega)^* \Phi_\varphi a^!=:\Phi_{\varphi,\check{\rho}(\Omega)}a^!$.

\begin{prop}\label{whittakerisvanishing}
    There is a natural equivalence \[\mathrm{Whit}_{X/S}\simeq \Phi_{\varphi,\check{\rho}(\Omega)}a^!.\]
\end{prop}
\begin{proof}
    Let $i:\varphi^{-1}(0)\to \Bun_N^{\check{\rho}(\Omega)}(X/S)$ be the inclusion of the zero fiber and $e:\mathbb{A}_\mathbb{C}^1\to \mathbb{A}_\mathbb{C}^1$ be the holomorphic map given by the universal cover of $\mathbb{A}_\mathbb{C}^1\setminus\{0\}\subset \mathbb{A}_\mathbb{C}^1$. By definition, the complex vanishing cycles of $a^!\mathcal{F}$ is given by $i^*\mathcal{H}om(\varphi^*K,a^!\mathcal{F})$, where $K$ is the sheaf on $\mathbb{A}_\mathbb{C}^1$ given by the cofiber of the map $e_!e^!\mathbf{1}\to \mathbf{1}$. Since $a^!\mathcal{F}$ is naturally $\mathbb{G}_m$-equivariant and the $\mathbb{G}_m$-action on $\Bun_N^{\check{\rho}(\Omega)}(X/S)$ contracts the stack to the section $\check{\rho}(\Omega)$, pulling back along the section $\check{\rho}(\Omega):S\to \Bun_N^{\check{\rho}(\Omega)}(X/S)$ is naturally equivalent to pushing forward to $S$ along the structure map.

    Let $p:\mathbb{A}^1_\mathbb{C}\to \mathbb{A}_\mathbb{C}^1/\mathbb{G}_m$ denote the quotient map. By adjoining $a^!$ over in the formula for vanishing cycles, the problem reduces to providing an equivalence between $p_!K$ and $\mathrm{exp}_!/\mathbb{G}_m$. To do this, consider the following diagram with Cartesian square
\[
\begin{tikzcd}
\mathbb{A}^1_{\mathbb{C}} \arrow["e'"',rd] \arrow[r, "e"] & \mathbb{A}^1_{\mathbb{C}} \arrow[r, "p"]                                & \mathbb{A}^1_{\mathbb{C}}/\mathbb{G}_m                                        \\
                                       & \mathbb{A}^1_{\mathbb{C}}\setminus\{0\} \arrow[u, "j'"] \arrow[r, "p'"] & *\simeq (\mathbb{A}^1_{\mathbb{C}}\setminus\{0\})/\mathbb{G}_m \arrow[u, "j"]
\end{tikzcd}\]
    We have a natural map $p_!\mathbf{1}\to \mathrm{exp}_!/\mathbb{G}_m$ defined as follows. Note that $\tau_{\le -1}p'_!\mathbf{1}\simeq \mathbf{1}[-1]$, then
    \begin{align*}
        \mathrm{Map}(p'_!\mathbf{1},\tau_{\le -1}p'_!\mathbf{1})\simeq \mathrm{Map}(p'_!\mathbf{1},\mathbf{1}[-1])\simeq  \mathrm{Map}(j^*p_!\mathbf{1},\mathbf{1}[-1])\simeq \mathrm{Map}(p_!\mathbf{1},j_*\mathbf{1}[-1])
    \end{align*}
   Thus the map $p'_!\mathbf{1}\to\tau_{\le -1}p'_!\mathbf{1}$ given by truncation defines a map $p_!\mathbf{1}\to \mathrm{exp}_!/\mathbb{G}_m$. It suffices to prove that 
   \[p_!e_!e^!\mathbf{1}\to p_!\mathbf{1}\to \mathrm{exp}_!/\mathbb{G}_m\]
   is a fiber sequence. This can be checked on stalks and hence it suffices to prove the statement after pulling back along $j$ or its closed complement $i:B\mathbb{G}_m\to \mathbb{A}^1/\mathbb{G}_m$. In both cases the statement reduces to simple cohomology computations.
\end{proof}

To prove Proposition \ref{Localconstancyofwhittaker} we will first prove a general result in the $C^\infty$-setting on the local constancy of microstalks in families, see Proposition \ref{localcontancyofmicro}. Then we will follow the same strategy as Nadler and Taylor to prove that the Whittaker functional is a microstalk in a relative sense. Together these two results will give a proof of Proposition \ref{Localconstancyofwhittaker}. 

\subsubsection{Non-characteristicness of the Poincaré Sheaf}

We recall the following result proved in \cite{gaitsgory2025geometric}:
\begin{prop}(\cite[Thm. 4.4.2]{gaitsgory2025geometric})\label{poincare non-char}
The Poincaré sheaf
\[
\on{Poinc}_!\in \on{Sh}(\on{Bun}_G(X/S))
\]

\noindent is non-characteristic over $S$.
\end{prop}

\begin{rem}
    In truth, they prove that $\on{Poinc}_!$ is ULA over $S$, which is weaker that non-characteristicness. However, the same proof goes through to show the latter. For completeness, we sketch their proof below.
\end{rem}

\begin{proof}

Assume that the genus of $X$ is at least $2$ (the case case $g=0,1$ are treated separately in \cite{gaitsgory2025geometric}). 

Consider the map
\[
r: \Bun_N^{\check{\rho}(\Omega)}(X/S)/T\to \Bun_B^{(g-1)2\check{\rho}}(X/S)\to \on{Bun}_G(X/S)^{((g-1)2\check{\rho})}\into \on{Bun}_G(X/S)^{(\leq(g-1)2\check{\rho})},
\]

\noindent where $\on{Bun}_G(X/S)^{(\lambda)}$ (resp. $\on{Bun}_G(X/S)^{(\leq\lambda)}$) denotes the locally closed (resp. open) substack of $\on{Bun}_G(X/S)$ given by the (relative) Harder-Narasimhan strata corresponding to the rational dominant coweight $\lambda$ (resp. union of HN-strata $\lambda'\leq \lambda$); see also \cite[§5.2]{gaitsgory2025geometric}. The first map above is closed by definition. The second is an isomorphism by \cite[Thm. 7.4.3 (1')]{drinfeld2015compact}, and the third is also closed, see \cite[Corollary 7.4.11]{drinfeld2015compact}. It follows that $r_!(\phi/\bG_m)^*(\on{exp}_!/\bG_m)$ is non-characteristic over $S$.

Denote by $j$ the inclusion
\[
\on{Bun}_G(X/S)^{(\leq(g-1)2\check{\rho})}\into \on{Bun}_G(X/S).
\]

\noindent It remains to prove that $j_!$ preserves sheaves non-characteristic over $S$. This in turn follows from the fact that $j$ is contractive in the sense of \cite[§5]{drinfeld2015compact}. See \cite[§19]{arinkin2020stack} for a proof that this implies that $j_!$ preserves non-characteristicness.
    
\end{proof}

\section{Microstalks in Families}
In this section we temporarily turn to the setting of $C^\infty$-manifolds. Throughout $S$ will denote $C^\infty$-manifold and $M\to S$ will be a submersion.

\subsection{Local constancy of the microstalk}
Our goal in this subsection is to prove a local constancy result for mircostalks in families.

Let $\mathcal{F}\in \mathrm{Sh}(M)$ be a sheaf on the total space of a submersive family $f:M\to S$ of $C^\infty$-manifolds. For every point $s\in S$, denote by $M_s$ the fiber over $s$ and fix a non-zero cotangent vector $v_s\in T^*M_s$. In this section we study a general criterion for when the microstalk of $\mathcal{F}|_{M_s}$ in the direction of $v_s$ varies in a local system on $S$.

\begin{rem}
The content of this section is closely related to the results in section 7 of \cite{guillermou2015quantizationconiclagrangiansubmanifolds}. In our language, the results of \cite{guillermou2015quantizationconiclagrangiansubmanifolds} implies the results in this section in the case when $M\to S$ is a constant family $M_0\times S\to S$ and the sheaf $\mathcal{F}$ is pulled back from a sheaf $\mathcal{F}_0$ on $M_0$. The ideas of the proofs in this section are essentially the same as in \cite{guillermou2015quantizationconiclagrangiansubmanifolds}, only minor modifications are required to allow non-constant families.    
\end{rem}

\subsubsection{Relative micro-stalks}
Suppose $\mathrm{SS}(\mathcal{F})\subset \Lambda$, where $\Lambda\subset T^*M$ is a closed conic Lagrangian stable under involution. Let $\varphi:M\to \mathbb{R}$ be a $C^\infty$-function such that $\Gamma_{d\varphi}\subset T^*M$ intersects $\Lambda$ in the graph $\Gamma_\rho$ of a section $\rho:S\to T^*M$ and assume that $\Gamma_\rho$ is contained in the smooth locus of the morphism $\Lambda\to S$. Suppose that, in a neighborhood of $\Gamma_\rho$, $\Lambda$ maps isomorphically to its image in the relative cotangent space $T_S^*M$. Assume that the cotangent vector $v_s\in T^*M_s$ we fixed is the image of $\rho(s)$ in the relative cotangent space.

For any $s\in S$, let $\varphi_s:M_s\to \mathbb{R}$ be the restriction of $\varphi$ to $M_s$, and let $\Lambda_s$ denote the closed cone given by the image of $\Lambda$ in $T^*M_s$. Let $i:S\hookrightarrow M$ denote the closed immersion given by the composition $\pi \circ \rho$, where $\pi:T^*M\to M$ is the projection map. Assume further that $\varphi\circ i=0$, $d\varphi(i(s))\in T^*M_s\setminus M_s$ and that $\Gamma_{d\varphi_s}$ intersects $\Lambda_s$ transversely at $\rho(s)$ for all $s\in S$.

\begin{defin}
    For any point $s\in S$, the \textit{shifted microstalk} of $\mathcal{F}_s$ along the function $\varphi_s$ at the point $i(s)$ is the object given by
\[\Phi_{\varphi_s,i(s)} \mathcal{F}_s:= i_s^*\Gamma_{\{\varphi_s\ge 0\}} \mathcal{F}_s\in \mathrm{Sh}(s).\]
\end{defin}

As in \cite{kashiwara1990sheaves}, $\Gamma_{\{\varphi_s\ge 0\}}\mathcal{F}_s$ denotes the sheaf on $M_s$ given by $k_*k^!\mathcal{F}_s$, where $k$ is the inclusion of the closed subset $\{\varphi_s\ge 0\}$ into $M_s$. For $A\subset M$ we abuse notation by writing $\mathbf{1}_A$ for the $!$-pushforward of $\mathbf{1}_A$ to $M$.

\begin{rem}
The term \emph{shifted microstalk} is justified since the assumption that $\Gamma_{d\varphi_s}$ intersects $\Lambda_s$ transversely implies that the vanishing cycles $\Phi_{\varphi_s,i(s)}\mathcal{F}_s$ really computes a shift of the microstalk of $\mathcal{F}_s$ at $i(s)$ in the codirection $v_s:=d\varphi_s(i(s))$ in the sense of Kashiwara-Schapira\cite{kashiwara1990sheaves}.
\end{rem}

In Proposition \ref{localcontancyofmicro} below, we prove that in the above situation these shifted microstalks vary in a local system along $S$. More precisely, we will prove that the sheaf given by
\[\Phi_{\varphi,i}\mathcal{F}:\simeq i^*\Gamma_{\{\varphi \ge0\}} \mathcal{F}\]
is a local system on $S$ with stalk at a point $s\in S$ equivalent to $\Phi_{\varphi_s,i(s)}\mathcal{F}_s$. We refer to $i^*\Gamma_{\{\varphi \ge0\}} \mathcal{F}$ as the relative shifted mircostalk.

\subsubsection{Local constancy of relative micro-stalks} Before proving the main proposition of this section, let us introduce a bit more notation and prove a few preliminary lemmas.

\begin{defin}
    Let $X$ be a smooth manifold and $Z_1,Z_2\subset X$ be subsets. We say that $Z_1$ and $Z_2$ intersect cleanly if $Z_1\cap Z_2$ is a closed submanifold of $X$, both $Z_1$ and $Z_2$ are closed submanifolds of $X$ in a neighborhood of $Z_1\cap Z_2$ and
    \[T_x(Z_1\cap Z_2)=T_xZ_1\cap T_xZ_2\]
     for all $x\in Z_1\cap Z_2$.
\end{defin}

For two sheaves $\cF_1,\cF_2$, we may consider the microlocal hom-functor $\mu\mathrm{hom}(\cF_1,\cF_2)\in\on{Sh}(T^*M)$ as defined in \cite[Definition 4.4.1]{kashiwara1990sheaves}.

\begin{lem}\label{ssforclean}
    Let $\Lambda_1$ and $\Lambda_2$ be closed conical Lagrangian subsets of $T^*M$ and $\mathcal{F}_i$ be a sheaf on $M$ with $SS(\mathcal{F}_i)\subseteq \Lambda_i$, for $i=1,2$. If $\Lambda_1$ and $\Lambda_2$ intersect cleanly, then
    \[SS(\mathrm{\mu hom}(\mathcal{F}_1,\mathcal{F}_2))\subset T^*_{\Lambda_1\cap \Lambda_2} T^*M.\]
\end{lem}

\begin{proof}
    By Corollary 6.4.3 in \cite{kashiwara1990sheaves}, there is the bound
    \[SS(\mathrm{\mu hom}(\mathcal{F}_1,\mathcal{F}_2))\subset C(\Lambda_2,\Lambda_1),\]
    where $C(\Lambda_2,\Lambda_1)$ is considered as a subset of $T^*T^*M$ via the Hamiltonian isomorphism $TT^*M\cong T^*T^*M$. By the assumption that $\Lambda_1$ and $\Lambda_2$ intersect cleanly and a computation in local coordinates, it follows that
    \[C(\Lambda_2,\Lambda_1)=(\Lambda_1\cap \Lambda_2)\times_{T^*M} T\Lambda_1+(\Lambda_1\cap \Lambda_2)\times_{T^*M} T\Lambda_2\subset TT^*M.\]
    Since $\Lambda_1$ and $\Lambda_2$ are Lagrangian, $T\Lambda_i$ is identified with $T^*_{\Lambda_i}T^*M$ under the Hamiltonian isomorphism over a neighbourhood of $\Lambda_1\cap \Lambda_2$. Hence
    \[C(\Lambda_2,\Lambda_1)\subset T^*_{\Lambda_1\cap \Lambda_2}T^*M.\]
\end{proof}

\begin{defin}
    Let $C_\rho\subset T^*M$ denote the subset
    \[C_\rho=\{\lambda \rho(s)\in T^*M| s\in S , \lambda> 0\}.\]
    In words, $C_\rho$ is the positive half-line spanned by the section $\rho:S\to T^*M$.
\end{defin}

\begin{defin}
    Let $\Lambda_\varphi'\subset T^*M$ denote the locally closed subset given by
    \[\Lambda_\varphi'=\{(x,\lambda d\varphi(x))\in T^*M| \varphi(x)=0, \lambda > 0\}.\]
\end{defin}

\begin{lem}\label{ssforhyper}
    There exist a neighborhood $V$ of the closed submanifold $i(S)$ such that
    \[SS(\mathbf{1}_{\{\varphi\ge 0\}})\cap T^*V\setminus V=\Lambda_\varphi'\cap T^*V\setminus V.\]
\end{lem}

\begin{proof}
    Since $d\varphi_s(i(s))\neq 0$ for all $s\in S$, the hypersurface $\varphi= 0$ is smooth in a neighborhood $V$ of $i(S)$. Thus the statement follows from [KS] Proposition 5.3.3. 
\end{proof}

\begin{lem}\label{intersectionofcones}
    The intersection of $\Lambda$ and $\Lambda_\varphi'$ is $C_\rho$.
\end{lem}

\begin{proof}
   Suppose that $\Lambda$ and $\Lambda'_\varphi$ intersect in a point $(x,\lambda d\varphi(x))$ for some $\lambda> 0$. Since $\Lambda$ and $\Lambda'_\varphi$ are cones, it follows that they intersect at the point $(x,d\varphi(x))$. This implies that $(x,d\varphi(x))$ is in the intersection of $\Lambda$ and $\Gamma_{d\varphi}$. Since the intersection of $\Lambda$ and $\Gamma_{d\varphi}$ is the section $\rho$, $(x,d\varphi(x))=\rho(f(x))$. Thus $(x,\lambda d\varphi(x))=\lambda \rho(f(x))\in C_\rho$.

  Since $\varphi\circ i=0$, it follows that $C_\rho\subset \Lambda_\varphi'$ and since $\Lambda$ is a cone and $\rho$ takes values in $\Lambda$, $C_\rho\subset \Lambda$.
\end{proof}

\begin{lem}\label{l:graphisclean}
    The Lagrangians $\Lambda$ and $\Gamma_{d\varphi}$ intersect cleanly.
\end{lem}

\begin{proof}
    Since $\Lambda$, $\Gamma_{d\varphi}$ and the intersection $C_\rho$ all map isomorphically to the relative cotangent space, it suffices to prove that the intersection in the relative cotangent space is clean.

    As $\Lambda$ is smooth in a neighborhood of the intersection with $\Gamma_{d\varphi}$ and the intersection maps submersively to $S$, it suffices to prove that the intersection is clean in each fiber over $S$. Thus we are reduced to proving that the intersection of $\Lambda_s$ and $\Gamma_{d\varphi_s}$ is clean for any $s\in S$. This intersection is transverse by assumption and in particular clean.
\end{proof}

\begin{lem}\label{conesareclean}
    The cones $\Lambda$ and $\Lambda_\varphi'$ intersect cleanly.
\end{lem}

\begin{proof}
    By the assumption that $d\varphi_s(i(s))\neq 0$, it follows that the hypersurface $\varphi=0$ is smooth in a neighborhood $V$ of $i(S)$. Thus $\Lambda_\varphi '\cap T^*V\setminus V$ is a closed submanifold of $T^*V\setminus V$, in particular $\Lambda_\varphi'$ is a closed submanifold in a neighborhood of $C_\rho$. Since $\Lambda$ is conic and smooth in a neighborhood of $\rho$, it follows that $\Lambda$ is smooth in a neighborhood of $C_\rho$. Since $\Lambda$ and $\Lambda_{\varphi}'$ are conic, it suffices to prove that the intersection is clean at the points of $\rho$. Let $\Gamma_{d\varphi}'=\{ (x,\lambda d\varphi(x))|x\in M,\lambda> 0\}$. Then $\Gamma_{d\varphi}'\cap T^*V\setminus V$ is a closed submanifold of dimension $\text{dim } M+1$ and $\Gamma_{d\varphi}'\cap \Lambda\cap T^*V\setminus V =C_\rho$. Since $\Gamma_{d\varphi}$ intersects $\Lambda$ cleanly by Lemma \ref{l:graphisclean},
    \[T_{\rho(s)} (\Lambda\cap \Gamma_{d\varphi})=T_{\rho(s)}\Lambda \cap T_{\rho(s)} \Gamma_{d\varphi}\]
    for all $s\in S$. Since $\text{dim } C_\rho=\text{dim } \Lambda \cap \Gamma_{d\varphi}+1$ and $\text{dim }\Gamma_{d\varphi}'=\text{dim }\Gamma_{d\varphi}+1$, it follows that
    
    \[T_{\rho(s)} C_\rho = T_{\rho(s)} \Lambda \cap T_{\rho(s)} \Gamma_{d\varphi}'.\]
    Now $\Lambda_\varphi'\subset \Gamma_{d\varphi}'$ and so:
    \[T_{\rho(s)} C_\rho\subset T_{\rho(s)} \Lambda \cap T_{\rho(s)} \Lambda_\varphi'\subset T_{\rho(s)} \Lambda \cap T_{\rho(s)} \Gamma_{d\varphi}'=T_{\rho(s)} C_\rho.\]
Since $\Lambda$ and $\Lambda_\varphi'$ are cones, the intersection is also clean at any $\lambda \rho(s)$, $\lambda>0$.
\end{proof}

\begin{lem}\label{estimates}
    For any sheaf $\mathcal{F}$ with $SS(\mathcal{F})\subset \Lambda$, there exist a neighborhood $V$ of $i(S)$ such that 
    \begin{align}
        &SS(\mu \mathrm{hom}(\mathbf{1}_{\{\varphi\ge 0\}},\mathcal{F})|_{T^*V\setminus V})\subset T^*_{C_\rho}(T^*V\setminus V)\\
        &i^*\pi_{V*} (\mu \mathrm{hom}(\mathbf{1}_{\{\varphi\ge 0\}},\mathcal{F})|_{T^*V\setminus V})\simeq \Phi_{\varphi,i}\mathcal{F}
    \end{align}
    where $\pi_V:T^*V\to V$ is the projection map.
\end{lem}
\begin{proof}
    The first statement follows from combining Lemma \ref{ssforclean}, Lemma \ref{intersectionofcones} and Lemma \ref{conesareclean}.

    To prove the second statement, we use Sato's fiber sequence \cite[4.3.1]{kashiwara1990sheaves} and pull back along $i$. This give the following fiber sequence
    \[i^*(\mathcal{H}om(\mathbf{1}_{\{\varphi\ge 0\}},\mathbf{1})\otimes \mathcal{F})\to \Psi_{\varphi,i}\mathcal{F}\to i^*\pi_{V*} (\mu \mathrm{Hom}(\mathbf{1}_{\{\varphi\ge 0\}},\mathcal{F})|_{T^*\setminus V}).\]
    We may assume the hypersurface $\varphi= 0$ is smooth, thus $\mathcal{H}om(\mathbf{1}_{\{\varphi\ge 0\}},\mathbf{1})\simeq \mathbf{1}_{\{\varphi> 0\}}$. Since $i$ takes values in the set where $\varphi= 0$, it follows that $i^*\mathbf{1}_{\{\varphi> 0\}}\simeq 0$ and we conclude the statement from Sato's fiber sequence.
\end{proof}

Finally, we are ready to prove the main proposition of this section.

\begin{prop}\label{localcontancyofmicro}
    The sheaf $\Phi_{\varphi,i}\mathcal{F}$ is a local system on $S$. Furthermore, for each $s\in S$, there is an isomorphism between the stalk of $\Phi_{\varphi,i}\mathcal{F}$ at $s$ and the vanishing cycles of $\mathcal{F}_s$ on $M_s$ along the function $\varphi_s$ at the point $i(s)$.
\end{prop}

\begin{proof}
    To see that $\Phi_{\varphi,i}\mathcal{F}$ is a local system on $S$, note that the first statement of Lemma \ref{estimates} implies that the sheaf $\mu \mathrm{hom}(\mathbf{1}_{\{\varphi\ge 0\}},\mathcal{F})|_{T^*V\setminus V}$ is pushed forward from a local system on $C_\rho$. Now $C_\rho$ is a $\mathbb{R}_{> 0}$-fiber bundle over $S$, so the further pushforward along $\pi_V:T^*V\to V$ is a local system supported on $i(S)$. Thus the local constancy of $\Phi_{\varphi,i}\mathcal{F}$ follows from the third statement of Lemma \ref{estimates}.
    
    To prove the identification of the stalks, we start by proving that $\Gamma_{\{\varphi\ge 0\}} \mathcal{F}$ is non-characteristic with respect to the map $g_s:M_s\to M$ in a neighborhood of $i(s)$. Lemma \ref{ssforhyper} together with \cite[Remark 6.2.5]{kashiwara1990sheaves} imply that
    \[SS(\Gamma_{\{\varphi\ge 0\}} \mathcal{F})\cap T^*V\setminus V\subset \Lambda\widehat{+} \Lambda_\varphi'=(\Lambda+ \Lambda_\varphi')\cup(\Lambda\underset{\infty}{\widehat{+}} \Lambda_\varphi')\]
    in a neighborhood $V$ of $i(S)$ (see \cite[Definition 6.2.3]{kashiwara1990sheaves} for the relevant notation). Since $d\varphi_s(x)\neq 0$ for all $x$ in a neighborhood of $i(S)$, we may assume that $\Lambda_\varphi'$ maps bijectively to its image in the relative cotangent space in a neighborhood of $i(S)$.

    To see that  $\Lambda+ \Lambda_\varphi'$ is non-characteristic over $S$, suppose $(x,v_1)\in \Lambda$ and $(x,v_2)\in \Lambda_\varphi'$ are such that $(x,v_1+v_2)$ is pulled back from a cotangent vector $(f(x),w)$ of $S$. Since $(x,v_1+v_2)$ is pulled back from $S$, its image in the relative cotangent space of $M$ over $S$ is $0$. Let $(x,\overline{v}_i)$ denote the image of $(x,v_i)$ in the relative cotangent space, and let $\overline{\Lambda}, \overline{\Lambda}_\phi$ denote the respective images of $\Lambda$, $\Lambda_\phi$ in the relative contangent space. Then $\overline{v}_1=-\overline{v}_2$. Since $\Lambda$ is stable under involution, this implies that $\overline{v}_2\in \overline{\Lambda}\cap \overline{\Lambda}_\varphi'$. Since $\Lambda$ and $\Lambda_\varphi'$ map bijectively to their image in the relative cotangent space and $\Lambda\cap \Lambda_\varphi'$ maps bijectively to $\overline{\Lambda}\cap \overline{\Lambda}_\varphi'$, it follows that $v_1=-v_2$. Thus $(x,v_1+v_2)=(x,0)$ and hence $\Lambda+\Lambda_\varphi'$ is non-characteristic over $S$.

    For the non-characteristicness of $(\Lambda\underset{\infty}{\widehat{+}} \Lambda_\varphi')$ over $S$, it suffices to see that
    \[(\Lambda\underset{\infty}{\widehat{+}} \Lambda_\varphi')\subset T^*_{i(S)} M.\]
    This follows from [KS] Proposition 6.2.4(ii) and the same local computation as in Lemma \ref{ssforclean} using the cleanness of the intersection of $\Lambda$ and $\Lambda'_\varphi$.

    By the non-characteristicness of $\Gamma_{\{\varphi\ge 0\}}\mathcal{F}$ and $\mathcal{F}$ relative to $g_s$, it follows that
    \[g_s^*\Gamma_{\{\varphi\ge 0\}} \mathcal{F}\simeq g_s^*\mathcal{H}om(\mathbf{1}_{\{\varphi\ge 0\}},\mathcal{F})\simeq \mathcal{H}om(g_s^*\mathbf{1}_{\{\varphi\ge 0\}},g_s^*\mathcal{F})\simeq \mathcal{H}om(\mathbf{1}_{\{\varphi_s\ge 0\}},\mathcal{F}_s)\]
    Thus we conclude the statement of the proposition by taking the stalk at $i(s)$.
\end{proof}

\subsection{The contraction principle in families}
Our goal is to show that the Whittaker functional is a relative microstalk in the sense defined in the previous section. For a curve over a point, this was proved in \cite{nadler2022whittaker}. We follow the same strategy in the relative setting.

In this subsection we prove the relative version of the contraction principle for vanishing cycles.

\subsubsection{The setup} Let $M\to S$ be a smooth morphism of complex analytic manifolds equipped with a $\mathbb{G}_m$-action relative to $S$. Let $M^0$ denote the fixed points of the $\mathbb{G}_m$-action and $i:S\to M^0$ be a section. Let $a:M^{> 0}\to M$ be the attracting locus of $i(S)$ and let $M^{\le 0}$ denote the global repelling locus. Assume that $M^{>0}$ and $M^{\le 0}$ are smooth over $S$ and that their intersection is $i(S)$.

Fix a $\mathbb{G}_m$-equivariant holomorphic function $f:M^{>0}\to \mathbb{C}$, where $\mathbb{G}_m$ acts on $\mathbb{C}$ by some non-zero weight. Since $i$ takes values in the fixed point locus, it follows that $f\circ i=0$. For a $\mathbb{G}_m$-monodromic sheaf $\mathcal{F}$ on $M$, we wish to compare the vanishing cycles of $a^!\mathcal{F}$ along $f$ at $i(S)$ to a direct vanishing cycles of $\mathcal{F}$ with respect to a function $F:U\to\mathbb{R}$ defined in a neighborhood $U$ of $i(S)$ extending $\operatorname{Re}f$ on $M^{>0}\cap U$.

\subsubsection{Coordinates and the extension of $f$}

The function $F$ will be constructed by choosing coordinates for $M$ diagonalizing the $\mathbb{G}_m$-action. In an open neighborhood $U$ around $i(S)$, we may choose coordinates 
\[(y^{>0}_S,y^0_S,y^{<0}_S):U\hookrightarrow \mathbb{A}_S^{n_{>0}}\times_S \mathbb{A}^{n_0}_S\times_S \mathbb{A}^{n_{<0}}_S\] 
such that the map is an open immersion and the $\mathbb{G}_m$-action is given locally by the standard $\mathbb{G}_m$-action with negative weights on the first factor, zero weight on the second factor and positive weights on the last factor. We may further assume that $i(S)$ is sent to the zero-section. Let $y^{>0}$ denote the composition 
\[U\overset{y^{>0}_S}{\to} \mathbb{A}^{n_{\ge 0}}_S\to \mathbb{A}^{n_{\ge 0}}_{\mathbb{C}},\] 
we use similar notation for the other coordinates.

With respect to these coordinate functions, we may define a function $d_{>0}:U\to \mathbb{R}$ (respectively, $d_0$ and $d_{<0}$) by pulling back the function on $\mathbb{A}^{n_{\ge 0}}_{\mathbb{C}}$ (respectively, $\mathbb{A}^{n_{\ge 0}}_{\mathbb{C}}$ and $\mathbb{A}^{n_{\ge0}}_{\mathbb{C}}$) which sends a point to the square of its distance from the origin in the standard metric. Let $\pi:U\to M^{>0}\cap U$ denote the local projection to the attracting locus in the given coordinates.

Given these choices, we define the function $F:U\to \mathbb{R}$ by
\[F=\operatorname{Re}f\circ \pi-d_0-d_{<0}.\]
This extends $\on{Re}f$ on the attracting locus in $U$, since $i(S)$ is mapped to the zero-section.

\subsubsection{The contraction principle}

Having chosen coordinates and obtained the extension $F$, we obtain a diagram as below, in which the square is cartesian.
\[
\begin{tikzcd}
S \arrow[r, "i"] & M^{>0} \arrow[r, "a"]                      & M                         \\
                 & \{\on{Re} f\ge 0\} \arrow[u, "k'"] \arrow[r, "a'"] & \{F\ge 0\} \arrow[u, "k"]
\end{tikzcd}\]
From this we obtain a natural transformation
\[\Phi_{f,i}a^!:= i^*k'_!k'^!a^!\simeq i^*a^*a_!k'_!k'^!a^!\simeq i^*a^*k_!a'_!a'^!k^!\to i^*a^*k_!k^!\simeq: \Phi_{F,i}.\]

We may now prove the following proposition, which we henceforth refer to as the \emph{contraction principle}.

\begin{prop}\label{contractionprinciplevanishing}
    For any $\mathbb{G}_m$-monodromic sheaf $\mathcal{F}$ on $M$ with singular support in a Lagrangian cone $\Lambda$, the natural map
    $\Phi_{f,i}a^!\mathcal{F}\to \Phi_{F,i}\mathcal{F}$
    is an equivalence of sheaves on $S$.
\end{prop}

\begin{proof}
    By arguing as in \cite{nadler2022whittaker}, the functor $\Phi_{f,i}a^!(-)$ is corepresented by the $!$-pushforward of $\mathbf{1}$ from $U:=B\cap \{\mathrm{Re} f\ge 0\}\cap M^{> 0}$ to $M$, where $B$ is any sufficiently small tubular neighborhood of $i(S)$ in $M$. On the other hand, $\Phi_{F,i}(-)$ is corepresented by the $!$-pushforward of $\mathbf{1}$ from $V:=B\cap \{F\ge 0\}$, where $B$ is any sufficiently small tubular neighborhood of $i(S)$ in $M$.

   From this we conclude that the cofiber of the natural transformation $\Phi_{f,i}a^!(-)\to \Phi_{F,i}(-)$ is corepresented by the $!$-pushforward of $\mathbf{1}$ from $U\setminus V$ to $M$. By arguing as in \cite{nadler2022whittaker} in each fiber over $S$, we conclude that each orbit of the $\mathbb{R}_{>0}$-action intersects $U\setminus V$ either in a half-open interval or has empty intersection with $U\setminus V$. From this we conclude that $\mathrm{Map}(\mathbf{1}_{U\setminus V},\mathcal{F})\simeq 0$ for any $\mathbb{R}_{>0}$-equivariant sheaf $\mathcal{F}$ on $M$, hence the natural map is an equivalence for such sheaves.
\end{proof}

\subsubsection{Well-chosen coordinates} The contraction principle applies to any extension $F$ obtained by a choice of local coordinates. However, in applying the contraction principle, we want to choose the coordinates such that $\Psi_{F,i}\mathcal{F}$ computes a shifted relative microstalk of $\mathcal{F}$.

Assume that $\Lambda\subset T^*M$ intersects $df+T^*_{M^{>0}}M$ cleanly, that $\Lambda \cap (df+T^*_{M^{>0}}M)$ maps submersively to $S$ and that
\begin{align}\label{dimensionassumption}
\mathrm{dim}(\Lambda_s \cap (df_s+T^*_{M^{>0}_s}M_s))\le \text{dim }M_s^{\le 0}
\end{align}
for all $s\in S$.

\begin{lem}
    \label{liftofsection}
    The section $df|_{i(S)}:i(S)\to T^*M^{>0}$ lifts uniquely to a section $\rho:i(S)\to T^*M$ that maps to $0$ in $T^*M^{\le 0}$.
\end{lem}

\begin{proof}
    Since $M^{>0}$ intersects $M^{\le 0}$ cleanly in a section $i:S\to M$, for any $s\in S$, the following square is Cartesian
    \[
\begin{tikzcd}
T_{i(s)}^*M \arrow[r] \arrow[d] & T^*_{i(s)}M^{\le 0} \arrow[d] \\
T^*_{i(s)}M^{>0} \arrow[r]      & T_s^*S.                       
\end{tikzcd}\]
The function $f$ defines a cotangent vector $df(i(s))$ in $T_{i(s)}^*M^{> 0}$, for any $s\in S$. Since $f\circ i=0$, $df(i(s))$ maps to $0$ in $T^*_sS$, thus it admits a unique lift $\rho(s)$ to $T^*_{i(s)}M$ which maps to $0$ under the map $T^*_{i(s)}M\to T^*_{i(s)}M^{\le 0}$. So $\rho(s)$ defines a section $\rho:S\to df+T^*_{M^{>0}}M$.
\end{proof}

Suppose that the image of $\rho$ is contained in the smooth locus of the morphism $\Lambda\to S$. The goal of this subsection is to prove the following result:

\begin{prop}\label{p:Extending test function}
    In the situation described above, the coordinates of $M$ around $i(S)$ can be chosen such that the intersection between $\Gamma_{dF}$ and $\Lambda$ is the image of $\rho$ and, for all $s\in S$, the intersection between $\Gamma_{dF_s}$ and $\Lambda_s$ is transverse.
\end{prop}

Let $\mathbb{G}_m$ act on $T^*M$ by a combination of the $\mathbb{G}_m$-action induced from $M$ and the scaling action in the fibers such that $\rho$ is fixed under the combined action. Note that $df+T^*_{M^{>0}}M$ is stable under the $\mathbb{G}_m$-action. We start by choosing the section that is going to become $\Gamma_{dF}|_{M^{>0}}$.

\begin{lem}\label{l:extending tangent space}
    For every point $i(s)\in i(S)$, there exists a neighbourhood $i(s)\in U_s$ and a $\mathbb{G}_m$-equivariant section $\gamma: U_s\to (df+T^*_{M^{>0}}M)\vert_{U_s}$ such that $\gamma\circ i=\rho$ and for any $s\in S$, the image $I_s$ of the induced map $T_{i(s)}M^{>0}\to T_{\rho(s)}T^*M$ satisfies
    \begin{align}\label{transversality}
        I_s\cap T_{\rho(s)}\Lambda=T_{s} S,
    \end{align}
    where $T_sS$ sits inside $T_{\rho(s)}T^*M$ via the section $\rho$.
\end{lem}

\begin{proof} We start by extending the map $TS\to T(df+T^*_{M^{>0}}M)\times_{df+T^*_{M^{>0}}M} S$ of vector bundles over $S$ induced by $\rho$ to a $\mathbb{G}_m$-equivariant section
\[TM^{> 0}\times_M S\to T(df+T^*_{M^{>0}}M)\times_{df+T^*_{M^{>0}}M}  S,\]
such that $I_s\cap T_{\rho(s)}\Lambda=T_sS$ for all $s\in S$.
To construct such an extension it suffices to produce a section of the relative bundles
\[T_{S}M^{> 0}\times_M S\to T_{S}(df+T^*_{M^{>0}}M)\times_{df+T^*_{M^{>0}}M} S\]
such that the image $I_s^\text{rel}$ satisfies $I_s^\text{rel}\cap T_{S,\rho(s)}\Lambda=0$. By the cleanness of the intersection between $\Lambda$ and $df+T^*_{M^{>0}}M$, to produce such a map locally on $S$, it suffices so see that the rank of the source is less than the rank of the target vector bundle minus the rank of 
\[(T_S\Lambda\cap T_S(df+T^*_{M^{>0}}M))\times_{df+T^*_{M^{>0}}M} S=T_S(\Lambda \cap (df+T^*_{M^{>0}}M))\times_{df+T^*_{M^{>0}}M} S.\]
Since $\Lambda$ and $df+T^*_{M^{>0}}M$ map bijectively to their images in the relative cotangent bundle, the dimension of the fiber of $\Lambda \cap (df+T^*_{M^{>0}}M)\to S$ above $s\in S$ is less than or equal to the dimension of $\Lambda_s\cap (df_s+T^*_{M^{>0}}M)$, hence the rank statement follows from the dimension assumption (\ref{dimensionassumption}).

Since $i(S)$ is a closed submanifold of $M^{>0}$, locally around every point of $i(S)$, there exist a $\mathbb{G}_m$-equivariant section $\gamma$ of $ df+T^*_{M^{>0}}M$ such that $\gamma \circ i=\rho$ and the map on tangent spaces $TM^{> 0}\to T(df+T^*_{M^{>0}}M)$ extends the chosen map over $i(S)$.
\end{proof}

Next we fix the behaviour along $M^{\le 0}$.

\begin{lem}
    Locally around every point of $i(S)$, there exists a section
\[\delta:TM^{\le 0}\times_{M^{\le 0}} S\to TT^*M\times_{T^*M} S\]
of the natural projection such that for all $s\in S$ 
\begin{enumerate}
    \item The kernel of the map $T_{i(s)}M^{\le 0}\to T_{\rho(s)}T^*M\to T_{\rho(s)}T^*M/(I_s+T_{\rho(s)}\Lambda)$ is $T_sS$.
    \item The composition
    \[T_{i(s)}M^{\le 0}\to T_{\rho(s)}T^*M\to T_{(i(s),0)}T^*M^{\le 0}\simeq T_{i(s)}M^{\le 0}\oplus T_{i(s)}^*M^{\le 0}\to T_{i(s)}^*M^{\le 0}\]
    defines a negative definite symmetric bilinear form.
\end{enumerate}
\end{lem}

\begin{proof}
    Consider the affine bundle over $S$ of all sections of $ TT^*M\times_{T^*M} S\to TM^{\le 0}\times_{M^{\le 0}} S$. The condition that the resulting map in $(2)$ is a symmetric bilinear form cuts out an affine subbundle. The condition in $(1)$ and the negative definiteness are open conditions, thus it suffices to see that each fiber over $S$ is non-empty. The fiberwise statement follows since the dimension of $I_s+T_{\rho(s)}\Lambda$ is equal to the dimension of $M$ plus the dimension of $M^{>0}$.
\end{proof}

For a choice of coordinates on $M^{\le0}$ diagonalizing the $\mathbb{G}_m$-action, let $D$ denote the function on $M^{\le 0}$ given by the square of the distance to $i(S)$ with respect to these coordinates. By $(2)$ we can, locally around every point of $i(S)$, choose coordinates on $M^{\le 0}$ such that $\delta$ composed with $TT^*M\times_{T^*M}S\to TT^*M^{\le 0}\times_{T^*M} S$ is the map on tangent spaces induced by $-dD:M^{\le 0}\to T^*M^{\le 0}$ above $i(S)$.

\begin{proof}[Proof of Proposition \ref{p:Extending test function}]
   Choose $\gamma$, $\delta$ and the coordinates on $M^{\le 0}$ as described above. Let $\delta'$ be an equivariant section of $T^*M\times_M M^{\le 0}$ lifting $-dD$ and with value $\rho$ above $i(S)$ and inducing $\delta$ on tangent spaces above $i(S)$. To extend these coordinates to all of $M$, we need to choose the $\mathbb{G}_m$-equivariant local projections $\pi_{M^{\le 0}}:M\to M^{\le 0}$ and $\pi_{M^{> 0}}:M\to M^{>0}$. Since $\gamma$ is an equivariant section of $df+T^*_{M^{>0}}M$ and $\gamma$ vanishes when pulled back to $M^{\le 0}$, we can choose $\pi_{M^{> 0}}$ such that $\gamma=\pi_{M^{>0}}^*df|_{M^{>0}}$. Consider the section
   \[\delta'-\pi_{M^{>0}}^*df|_{M^{\le 0}}\]
   of $T^*M\times_M M^{\le 0}$. This form is equivariant, vanishes above $i(S)$ and lifts $-dD$ in the cotangent space of $M^{\le 0}$, so we can choose $\pi_{M^{\le 0}}:M\to M^{\le 0}$ such that \[-\pi_{M^{\le 0}}^*dD =\delta'-\pi_{M^{>0}}^*df|_{M^{\le 0}}.\]
   We claim that these coordinates produce a function $F:M\to \mathbb{R}$ with the desired properties.

   By construction $F=\on{Re} f\circ \pi_{M^{>0}}-D\circ \pi_{M^{\le 0}}$. Thus 
   \[dF|_{M^{\le 0}}=\pi_{M^{>0}}^*df|_{M^{\le 0}}-\pi_{M^{>0}}^*df|_{M^{\le 0}}+\delta'=\delta'\]
   and since $dD$ vanishes along $i(S)$
   \[dF|_{M^{>0}}=\gamma-\pi_{M^{\le 0}}^*dD|_{M^{>0}}=\gamma.\]
   Since $dF$ takes the value $\rho$ over $i(S)$, we conclude that $\Gamma_\rho$ is contained in the intersection of $\Lambda$ and $\Gamma_{dF}$. At the image of $\rho$ the tangent space of $\Gamma_{dF}$ is given by the image of $TM\times_M S\to TT^*M\times_\rho S$ determined by $\delta$ and the map on tangent spaces induced by $\gamma$. Thus it follows that $\Lambda$ intersects $\Gamma_{dF}$ cleanly in $\Gamma_\rho$. 
\end{proof}

\section{Local Constancy of the Whittaker Functional}

In this section we apply the results of the previous sections to conclude the local constancy of the Whittaker functional. First we have to rigidify the situation in such a way that we can apply the contraction principle for vanishing cycles to obtain a description of the Whittaker functional directly as a vanishing cycles over $\Bun_G(X/S)$. This part of the argument is directly analogous to Nadler and Taylors argument and all we do is upgrading it to the relative setting. 

\subsection{Rigidifying the moduli problems} Recall our convention that, when $S$ is real analytic, moduli problems are interpreted by pulling back the moduli problem in the universal case along the map $S\to (\mathcal{M}_g)_\mathbb{R}$, see Appendix A. Thus in describing the moduli problems below we assume $S$ is algebraic.

Locally on $S$, we may choose a section $x:S\to X$. This section defines a closed subscheme of $X$ given by an ideal sheaf $\mathcal{I}$. For each $n\in \mathbb{Z}_{>0}$, let $D_n(x)$ denote the closed subscheme $V(\mathcal{I}^n)$ of $X$. 

Let $\Bun_{G,N}^{\Omega,nx}(X/S)$ be the stack over $S$ that sends a test object $T/S$ to the space consisting of a $G$-bundle $P_G$ on $X\times_S T$ and a $B$-reduction of $P_G|_{T\times_SD_n(x)}$ together with an identification of the induced $T$-bundle with  $\check{\rho}(\Omega)|_{T\times_SD_n(x)}$.

Similarly, let $\Bun_{G}^{\Omega,nx}(X/S)$ be the moduli space of $G$-bundles on $X/S$ with an identification of its restriction to $D_n(x)$ with the $T$-bundle $\check{\rho}(\Omega)$. These moduli spaces form a diagram
\[
\begin{tikzcd}
\Bun_N^{\Omega,nx}(X/S) \arrow[r, "a_n"] \arrow[d] & \Bun_G^{\Omega,nx}(X/S) \arrow[d] \arrow[rd, "\pi_n"] &        \\
\Bun_N^\Omega(X/S) \arrow[r]                      & {\Bun_{G,N}^{\Omega,nx}}(X/S) \arrow[r]               & \Bun_G(X/S)
\end{tikzcd}\]
where $\Bun_N^{\Omega,nx}(X/S)$ is defined so that the above square is Cartesian. The $\mathbb{G}_m$-action on $G$ through conjugation by $\check{\rho}$ acts trivially on $T$ and fixes the subgroup $B$. This provides $\mathbb{G}_m$-actions and equivariance of all stacks and morphisms in the diagram above.

Let $B^-$ be the Borel subgroup opposite to $B$. For $B^{\pm}\in\lbrace B,B^-\rbrace$, let $\Bun_{B^{\pm}}^{\Omega,nx}(X/S)$ denote the stack of $B^{\pm}$-bundles on $X/S$ with a reduction to the $T$-bundle $\check{\rho}(\Omega)$ over $D_n(x)$.

\begin{prop}
    For $n$ sufficiently large, the maps $\Bun_N^{\Omega,nx}(X/S) \to \Bun_G^{\Omega,nx}(X/S)$ and $\Bun_{B^-}^{\Omega,nx}(X/S)\to \Bun_G^{\Omega,nx}(X/S)$ are locally closed immersions.
\end{prop}

\begin{proof} Since the statement is closed under base change, we may assume $S$ is algebraic. 
    Since $\Bun_N^{\Omega}(X/S)$ is quasi-compact, we can choose a quasi-compact open substack $U\subseteq \Bun_G(X/S)$ containing the image of $\Bun_N^{\Omega}(X/S)$. For $n$ sufficiently large, the stack $\Bun_G^{\Omega,nx}(X/S)|_U$ is representable by a scheme and so is $\Bun_N^{\Omega,nx}(X/S)$.

    Since all the schemes are smooth, it suffices to prove that the maps are surjective on cotangent spaces. Since the schemes are smooth over $S$, it suffices to prove that they are surjective on the cotangent spaces relative to $S$. Thus the family version reduces to the same statement as in [NT] (3.12).
\end{proof}

\begin{prop}
    For $n$ sufficiently large, there exist a $\mathbb{G}_m$-stable neighborhood $V$ of the section $\check{\rho}(\Omega):S\to \Bun_G^{\Omega,nx}(X/S)$ such that the attracting locus of the section $\check{\rho}(\Omega)$ under the $\mathbb{G}_m$-action on $V$ is the closed subspace $\Bun_N^{\Omega,nx}(X/S)\cap V$ and the global repelling locus in $V$ is $\Bun_{B^-}^{\Omega,nx}(X/S)\cap V$.
\end{prop}

\begin{proof} We may assume $S$ algebraic. We will prove that there exists a neighborhood $V$ of $\Bun_T^{\Omega,nx}(X/S)\to \Bun_G^{\Omega,nx}(X/S)$ such that the global attracting locus in $V$ is $V\cap \Bun_B^{\Omega,nx}(X/S)$. The statement of the proposition is immediately deduced from this.

Since the $\mathbb{G}_m$-action on $B$ by conjugation via $\check{\rho}$ contracts $B$ to $T$, $\Bun_{B}^{\Omega,nx}(X/S)$ maps to the attracting locus of the $\mathbb{G}_m$-action on $\Bun_G^{\Omega,nx}(X/S)$. This map is a closed immersion for $V$ sufficiently small by the same argument as in the previous proposition. Hence to see that this closed substack is the full attracting locus in a neighborhood of $\Bun_T^{\Omega,nx}(X/S)$, it suffices to prove that the inclusion into the attracting locus is an isomorphism on tangent spaces at each point of $\Bun_T^{\Omega,nx}(X/S)$.

Since $\Bun_B^{\Omega,nx}(X/S)$ is smooth over $S$, it suffices to prove that the inclusion into the attracting locus is an isomorphism on the relative tangent spaces over $S$, for each point $L$ of $\Bun_T^{\Omega,nx}$. Hence the statement follows as in \cite[Proposition 3.2.2]{nadler2022whittaker} since
\[H^1(X_s,\mathfrak{b}_L(-nx_s))\to H^1(X_s,\mathfrak{g}_L(-nx_s))\]
identifies with the non-negative weight space of the $\mathbb{G}_m$-action on the target.
\end{proof}

\subsection{The nilpotent cones and the Kostant section}

\begin{prop}
    The universal nilpotent cone $\mathrm{Nilp}_u\subset T^* \mathrm{Bun}_G(X/S)$ intersects the shifted conormal bundle $d\varphi +T^*_{\Bun_N^\Omega} \Bun_G(X/S)$ transversally and the intersection maps isomorphically to $S$.
\end{prop}

\begin{proof}
First we prove that the relative nilpotent cone $\mathrm{Nilp}$ and the shifted conormal bundle in the relative cotangent space $T_S^*\Bun_G(X/S)$ intersect transversally. This follows by the same argument as in \cite{nadler2022whittaker}, namely, consider the relative Hitchin fibration
\[H_S:T^*_S\Bun_G(X/S)\to \mathrm{Hitch}(X/S).\]
The shifted conormal bundle in the relative cotangent space is image of the Kostant section of $H_S$. The Kostant section takes values in the locus of regular Higgs fields and the Hitchin map is smooth on this locus. The relative nilpotent cone is the preimage of the zero map $\mathbf{0}:S\to \mathrm{Hitch}(X/S)$. Thus the intersection is transverse, since a section of a smooth map intersects the preimage of a submanifold of the base transversely. The intersection is given by $\mathrm{Kos}\circ \mathbf{0}:S\to T^*_S\Bun_G(X/S)$ and hence maps isomorphically to $S$.

    We have to prove that the intersection in the absolute cotangent space is also transverse. Since $\mathrm{Nilp}_u$ and $d\varphi +T^*_{\Bun_N^\omega} \Bun_G$ map bijectively to the relative cotangent space and the intersection there is transverse, it suffices to lift $\mathrm{Kos}\circ \mathbf{0}:S\to T^*_S\Bun_G(X/S)$ to $T^*\Bun_G(X/S)$ such that it factors through the intersection in the absolute cotangent space.

    The bundle $\check{\rho}(\Omega)$ defines a section $S\to \Bun_G(X/S)$. This gives a short exact sequence
\[0\to T^*S\to T^*\Bun_G\times_{\check{\rho}(\Omega)} S\to T_S^*\Bun_G\times_{\check{\rho}(\Omega)} S\to 0\]
    as coherent sheaves on $S$. The differential of the section defined by $\check{\rho}(\Omega)$ gives a splitting of this sequence. In particular we obtain a morphism of stacks
    \[T_S^*\Bun_G\times_{\check{\rho}(\Omega)} S\to T^*\Bun_G(X/S).\]
    Since the compostion $\mathrm{pr}_{\Bun_G}\circ \mathrm{Kos}\circ \mathbf{0}$ is the map $S\to \Bun_G(X/S)$ defined by $\check{\rho}(\Omega)$, we obtain a map
    \[s:S\to T^*\Bun_G(X/S)\]
    lifting $\mathrm{Kos}\circ \mathbf{0}$. It suffices to prove that this map factors through both $\mathrm{Nilp}_u$ and the shifted conormal bundle.

    That $s$ factors through $d\varphi +T^*_{\Bun_N^\Omega} \Bun_G$ follows since $\check{\rho}(\Omega)$ defines an $S$-point of $\Bun_N^\Omega$ and $\varphi\circ \check{\rho}(\Omega)=0$. To see that $S$ factors through the universal nilpotent cone, note that there is a canonical $B$-reduction on $\check{\rho}(\Omega)$ such that $\mathrm{Kos}\circ \mathbf{0}$ composed with the map
    \[T^*_S\Bun_G(X/S)\times_{\check{\rho}(\Omega)} S\to T^*_S\Bun_B(X/S)\times_{\check{\rho}(\Omega)} S\]
    is the zero-section. This $B$-reduction defines a map of short exact sequences
    \[\begin{tikzcd}
T^*S \arrow[r] \arrow[d] & T^*\Bun_G\times_{\check{\rho}(\Omega)}S \arrow[r] \arrow[d] &   T_S^*\Bun_G\times_{\check{\rho}(\Omega)}S  \arrow[d]   \\
T^*S \arrow[r]                      &  T^*\Bun_{B}\times_{\check{\rho}(\Omega)}S \arrow[r]               & T_S^*\Bun_{B}\times_{\check{\rho}(\Omega)}S
\end{tikzcd}\]
This map is compatible with the splittings of the two sequences defined by the differential of the maps determined by $\check{\rho}(\Omega)$. Thus $s$ composed with the middle vertical map is the zero-section and hence $s$ factors through $\mathrm{Nilp}_u$.

\end{proof}

\begin{prop}
    The universal nilpotent cone $\mathrm{Nilp}_u^n$ in $T^*\mathrm{Bun}_G^{\Omega,nx}$ intersects the shifted conormal bundle $d\mathrm{Re}\varphi_n+T^*_{\Bun_N^{\Omega,nx}} \Bun_G^{\Omega,nx}$ cleanly and the intersection is smooth over $S$.
\end{prop}

\begin{proof}
    We start by proving that the following map is an isomorphism
    \[\mathrm{Nilp}_u^n\cap (d\mathrm{Re}\varphi_n+T^*_{\Bun_N^{\Omega,nx}} \Bun_G^{\Omega,nx})\to \mathrm{Nilp}_\mathrm{rel}^n\cap (d\mathrm{Re}\varphi_n+T^*_{\Bun_N^{\Omega,nx}} \Bun_G^{\Omega,nx})_\mathrm{rel},\]
    where the first intersection is inside the absolute cotangent space $T^*\Bun_G^{\Omega,nx}$ and the second intersection is in the relative cotangent space $T^*_S\Bun_G^{\Omega,nx}$. By definition $\mathrm{Nilp}_u^n$ is contained in
    \[T^*\Bun_G\times_{\Bun_G}\Bun_G^{\Omega,nx}\subset T^*\Bun_G^{\Omega,nx},\]
    where it is given by the preimage of $\mathrm{Nilp}_u\subset T^*\Bun_G(X/S)$ under the first coordinate projection. Thus
    \[\mathrm{Nilp}_u^n\cap (d\mathrm{Re}\varphi_n+T^*_{\Bun_N^{\Omega,nx}} \Bun_G^{\Omega,nx})\subseteq T^*\Bun_G\times_{\Bun_G}\Bun_G^{\Omega,nx}\]
    and is given by the preimage under the first coordinate projection of the intersection
    \[\mathrm{Nilp}_u\cap (d\mathrm{Re}\varphi +T^*_{\Bun_N^\Omega} \Bun_G)\subseteq T^*\Bun_G.\]
    The analogous statement in the relative cotangent bundle is also true, thus it suffices to prove that following map is an equivalence
    \[\mathrm{Nilp}_u\cap (d\mathrm{Re}\varphi+T^*_{\Bun_N^\Omega} \Bun_G)\to \mathrm{Nilp}_\mathrm{rel}^n\cap (d\mathrm{Re}\varphi+T^*_{\Bun_N^\Omega} \Bun_G)_\mathrm{rel},\]
    where the intersections are in $T^*\Bun_G$ and $T^*_{S}\Bun_G$, respectively. This statement follows from the previous proposition and we conclude that the intersection in the absolute cotangent space maps isomorphically to the relative cotangent space.

    Given this it suffices to prove that the intersection in the relative cotangent space is clean. Since the intersection is smooth over $S$ and the nilpotent cone is smooth in a neighborhood of the intersection, it suffices to prove that for each point $s$ of $S$ the intersection in the fiber above $s$ is clean. Thus we have reduced the problem for a family of curve to the case of a single curve, which is \cite[Proposition 3.3.3]{nadler2022whittaker}.
\end{proof}

We need one more input from \cite{nadler2022whittaker} in the case when $S$ is a point.

\begin{prop}[Nadler-Taylor]\label{dimensionofintersection}
    For any point $s\in S$, the dimension of the intersection of the shifted conormal bundle
    \[d\mathrm{Re}\varphi_{n,s}+T^*_{\mathrm{Bun}_N^{\Omega,nx_s}(X_s)} \Bun_G^{\Omega,nx_s}(X_s)\]
    and nilpotent cone inside $T^*\Bun_G^{\Omega,nx_s}(X_s)$ is $n\dim N$.
\end{prop}

\begin{proof}
    This is the dimension claim in [NT] Proposition 3.3.3.
\end{proof}

With this final piece at hand, we can now prove the local constancy of the Whittaker functional.

\begin{proof}[Proof of Proposition \ref{Localconstancyofwhittaker}]
    By Proposition \ref{whittakerisvanishing} the Whittaker functional is given by vanishing cycles along $\chi$ at the section $\check{\rho}(\Omega):$
    \[\mathrm{Whit}(\mathcal{F})\simeq \Phi_{\varphi,\check{\rho}(\Omega)}a^!\mathcal{F}.\]
    By smooth base change for vanishing cycles, this is identified with
    \[\Phi_{\varphi_n,\check{\rho}(\Omega)}a_n^!\pi_n^!\mathcal{F},\]
    for any $n\ge 1$. By choosing $n$ sufficiently large, we may choose a neighborhood $U$ of the section $\check{\rho}(\Omega):S\to \Bun_G^{\Omega,nx}(X/S)$ such that the following are satisfied: 
    \begin{enumerate}
        \item $U$ is representable by an analytic space.
        \item  Inside $U$, the contracting locus of $\check{\rho}(\Omega)$ under the $\mathbb{G}_m$-action is $\Bun_N^{\Omega,nx}(X/S)$.
        \item The global repelling locus is $\Bun_{B^-}^{\Omega,nx}(X/S)$.
    \end{enumerate}
    Note that  $\Bun_N^{\Omega,nx}(X/S)$ and $\Bun_{B^-}^{\Omega,nx}(X/S)$ are smooth over $S$ and intersect in the section $\check{\rho}(\Omega)$. Fiberwise over $S$, the intersection is transverse. By Proposition \ref{dimensionofintersection}, the dimension assumption \ref{dimensionassumption} is satisfied. Thus, after potentially shrinking $U$, we choose coordinates such that $F:U\to \mathbb{R}$ satisfying the conclusion of Proposition \ref{p:Extending test function}.

    Since the map $\Bun_G^{\Omega,nx}(X/S)\to \Bun_G(X/S)$ is $\mathbb{G}_m$-equivariant for the trivial action on the target, the sheaf $\pi_n^!a^!\mathcal{F}$ has a natural $\mathbb{G}_m$-equivariance structure. Thus we are in the situation of Proposition \ref{contractionprinciplevanishing}, which implies that
    \[\mathrm{Whit}(\mathcal{F})\simeq \Phi_{F,\check{\rho}(\Omega)}\mathcal{F}.\]
    Since $\Gamma_{dF}$ intersects $\mathrm{Nilp}_u^n$ in the locus of everywhere regular Higgs fields, the intersection is contained in the smooth locus of $\mathrm{Nilp}_u^n\to S$. Thus from the local constancy of microstalks in families, see Proposition \ref{localcontancyofmicro}, we conclude that $\Psi_{F,\check{\rho}(\Omega)}\mathcal{F}$ is a local system on $S$. The identification of the stalks follows from the same proposition.
\end{proof}

\part{}

\section{Beilinson's Spectral Projector in Families}
In this section we set up Beilinson's spectral projector for families of automorphic sheaves. The spectral projector provides a formula for the left adjoint of the inclusion $\mathrm{Sh}_\mathrm{Nilp_{rel}}(\Bun_G)\to \mathrm{Sh}(\Bun_G)$ in terms of the Hecke action. We observe that this formula is given by non-characteristic kernels and use this to prove that Beilinson's spectral projector sends ind-non-characteristic ind-constructible sheaves to sheaves with universal nilpotent singular support.

\subsubsection{} Due to the relative nature of the objects we consider, to formulate the results without assuming $S$ to be contractible, we will consider many objects as (pre)sheaves on $S_\mathrm{et}$ with values in the $(\infty,1)$-category $\mathrm{Pr}_e^L$ of presentable $e$-linear stable categories. By this we will mean a (pre)sheaf on the site $S_{et}$, spanned by etale morphisms $T\to S$ from a contractible analytic variety and covers generated by the jointly surjective families of morphisms. 

To us, the most important example of such a sheaf of categories is the sheaf which on objects is given by
\[(T\to S)\mapsto \mathrm{Sh}(\Bun_G(X\times_S T/T))\]
with $*$-pullback functoriality. 

\begin{rem}
To avoid clunky notation, we will sometimes use the same notation for a presheaf of categories and its category of global sections.  Thus we will simply denote the above sheaf by $\mathrm{Sh}(\Bun_G(X/S))$.
\end{rem}

We denote the category of all presheaves on $S_\mathrm{et}$ with values in a category $\mathcal{C}$ by $\mathrm{PSh}(S_\mathrm{et},\mathcal{C})$.

\begin{defin}
    For $\mathcal{C}$ an $(\infty,1)$-category, a presheaf $\mathcal{F}:S_{et}\to \mathcal{C}$ is said to be locally constant if, for all $T'\to T$ over $S$ with $T'$ and $T$ contractible analytic varieties, the restriction map
    \[\mathcal{F}(T)\to \mathcal{F}(T')\]
    is an equivalence. Let $\mathrm{PSh}_{lc}(S_\mathrm{et},\mathcal{C})$ denote the full subcategory spanned by locally constant presheaves.
\end{defin}

For more details see Appendix B.

\subsection{Hecke action in families} Let us start by recalling the Hecke action in the relative setting. For a finite set $I$, consider the correspondence of stacks 
\[
\begin{tikzcd}
	& {\mathrm{Hecke}_I(X/S)} \arrow[ld, "\overset{\leftarrow}{h}"'] \arrow[rd, "\overset{\rightarrow}{h}\times \mathfrak{s}"]\\
	{\mathrm{Bun}_G(X/S)} && {\mathrm{Bun}_G(X/S)\times_S X^{I/S}}
\end{tikzcd}\]
where $\mathrm{Hecke}_I(X/S)$ denotes the stack with $T$-valued points given by an an $I$-indexed set of sections of $X\times_S T/T$, a pair of $G$-bundles over $X\times_S T$ and an identification away from the union of the graphs of the sections.

The stack $\mathrm{Hecke}_I(X/S)$ admits a map $\mathfrak{r}:\mathrm{Hecke}_I(X/S)\to \mathrm{Hecke}_I^\mathrm{loc}(X/S)$, where $\mathrm{Hecke}_I^\mathrm{loc}(X/S)$ denotes the similarly defined local Hecke stack. Let $\mathrm{Coord}(X/S)$ denote étale sheafification of the prestack over $X$ which to a map $x:T\to X$ associates the space of isomorphisms between $T[[z]]$ and the formal completion of $X\times_S T$ along the section over $T$ defined by $x$, which preserves the origin. The map $\mathrm{Coord}(X/S)\to X$ is a torsor for the group scheme $\mathrm{Aut}^0(D)$ of automorphisms of $D=\Spec \mathbb{C}[[z]]$ preserving the point given by the maximal ideal $(z)$.

The local Hecke stack sits in a correspondence
\[\mathrm{Hecke}_I^\mathrm{loc}(X/S)\leftarrow \mathrm{Hecke}_I^\mathrm{loc}(X/S)\times_{X^{I/S}} \mathrm{Coord}(X/S)^{I/S} \to ( L^+G \backslash \mathrm{Gr}_{G})^{I/S}.\]
The geometric Satake equivalence associates to each object of $\mathrm{Rep}(\check{G})$ an $\mathrm{Aut}^0(D)$-equivariant sheaf on $L^+G \backslash \mathrm{Gr}_{G}$. Pulling back along the above correspondence and using the $\mathrm{Aut}^0(D)$-equivariance to descent along the left morphism, we obtain a functor
\[\mathrm{Sat}_I: (\mathrm{Rep}(\check{G})^{\otimes I})\to \mathrm{Sh}(\mathrm{Hecke}_I^\mathrm{loc}(X/S)).\]

\begin{defin}
    For any $V\in \mathrm{Rep}(\check{G})^{\otimes I}$, the Hecke operator with kernel $V$ is the functor
    \[H(V,-):\mathrm{Sh}(\Bun_G(X/S))\to \mathrm{Sh}(\Bun_G(X/S)\times_S X^{I/S})\]
    given by sending $\mathcal{F}$ to the sheaf
    \[(\overset{\rightarrow}{h}\times \mathfrak{s})_!((\overset{\leftarrow}{h})^*\mathcal{F}\otimes \mathfrak{r}^*\mathrm{Sat}_I(V))\]
\end{defin}

\subsubsection{} As usual we integrate all the Hecke operators to an action of a single monoidal category. Consider the functor $R:\mathrm{TwArr}(\mathrm{Fin})\to \mathrm{Pr}_e^L$, which on objects is given by
\[(I\to J)\mapsto \mathrm{Rep}(\check{G})^{\otimes I}\otimes \mathrm{Sh}(X^{J/S})\]
and sends a morphism $f:(I_1\to J_1)\to (I_2\to J_2)$ given by a diagram
\[
\begin{tikzcd}
I_1 \arrow[r] \arrow[d, "f_I"] & J_1                   \\
I_2 \arrow[r]                  & J_2 \arrow[u, "f_J"']
\end{tikzcd}\]
to the functor 
\[\mathrm{mult}^{f_I}\otimes \Delta_{f_J!}:\mathrm{Rep}(\check{G})^{\otimes I_1}\otimes \mathrm{Sh}(X^{J_1/S})\to \mathrm{Rep}(\check{G})^{\otimes I_2}\otimes \mathrm{Sh}(X^{J_2/S}),\]
where $\mathrm{mult}^{f_I}$ is the multiplication map along the fibers of $f_I$ given by the monoidal structure on $\mathrm{Rep}(\check{G})$ and $\Delta_{f_J}$ is the the diagonal map determined by $f_J$.

\begin{defin}
    Let $\mathrm{Rep}(\check{G})_{\mathrm{Ran}(X/S)}\in \mathrm{Pr}_e^L$ denote the colimit of $R:\mathrm{TwArr}(\mathrm{Fin})\to \mathrm{Pr}_e^L$.
\end{defin}

The functor can be promoted to a lax symmetric monoidal functor. This gives a symmetric monoidal structure on $\mathrm{Rep}(\check{G})_{\mathrm{Ran}(X/S)}$. The combined action of the Hecke operators integrates to an action of $\mathrm{Rep}(\check{G})_{\mathrm{Ran}(X/S)}$ on $\mathrm{Sh}(\Bun_G(X/S))$ linear over $\on{Sh}(S)$. This action is defined by letting the image of $V\otimes \mathcal{S}$ in $\mathrm{Rep}(\check{G})_{\mathrm{Ran}(X/S)}$, for $V\in \mathrm{Rep}(\check{G})^I$, $\mathcal{S}\in \mathrm{Sh}(X^{J/S})$ and $\phi:I\to J$, act by the following composition:
\begin{align*}
    \mathrm{Sh}(\Bun_G)\overset{H(V,-)}{\to} &\mathrm{Sh}(\Bun_G\times_S X^{I/S})\overset{(\mathrm{
id}\times \Delta_\phi)^*}{\to }\mathrm{Sh}(\Bun_G\times_S X^{J/S})\\
&\overset{\otimes\mathrm{pr}_2^*\mathcal{S}}{\to }\mathrm{Sh}(\Bun_G\times_S X^{J/S})\overset{\mathrm{pr}_{\Bun_G!}}{\to} \mathrm{Sh}(\Bun_G),
\end{align*}
where $H(V,-)$ is the Hecke operator with kernel $V$.

We can consider $\mathrm{Rep}(\check{G})_{\mathrm{Ran}(X/S)}$ as a presheaf of categories on $S_{et}$ with values in $\mathrm{Alg}(\mathrm{Pr}^L_{st})$, which on objects is defined by 
\[(T\to S)\mapsto \mathrm{Rep}(\check{G})_{\mathrm{Ran}(X\times_S T/T)}\in \mathrm{Alg}(\mathrm{Pr}^L_e).\]
The Hecke operators define an action of the presheaf of algebras $\mathrm{Rep}(\check{G})_{\mathrm{Ran}(X/S)}$ on the sheaf $\mathrm{Sh}(\Bun_G(X/S))$. 
 
\subsubsection{Non-characteristic Hecke kernels} Since the family of curves $X/S$ is assumed to be proper, the functors
\[ \Delta_{f_J!}: \mathrm{Sh}(X^{J_1/S})\to \mathrm{Sh}(X^{J_2/S})\]
preserves the full subcategories of sheaves that are non-characteristic over $S$. Let $\mathrm{Sh}(X^{J/S})^\mathrm{nc}$ denote the full subcategory of non-characteristic sheaves. Thus we obtain a functor
\[R^\mathrm{nc}:\mathrm{TwArr}(\mathrm{Fin})\to \mathrm{Pr}^L_e\]
given by 
    \[(I\overset{\phi}{\to} J)\mapsto \mathrm{Rep}(\check{G})^{\otimes I}\otimes \mathrm{Sh}(X^{J/S})^\mathrm{nc}\]

\begin{defin}
    Let $\mathrm{Rep}(\check{G})_{\mathrm{Ran}(X/S)}^\mathrm{nc}$ denote the category given by the colimit of $R^\mathrm{nc}:\mathrm{TwArr}(\mathrm{Fin})\to \mathrm{Pr}^L_e$.
\end{defin}
We refer to $\mathrm{Rep}(\check{G})_{\mathrm{Ran}(X/S)}^\mathrm{nc}$ as the category of non-characteristic kernels relative to $S$. It admits a natural map
\[\mathrm{Rep}(\check{G})_{\mathrm{Ran}(X/S)}^\mathrm{nc}\to \mathrm{Rep}(\check{G})_{\mathrm{Ran}(X/S)},\]
which upgrades to a map of presheaves of categories on $S_{et}$.

\subsection{Kernel of the spectral projector}\label{ss:10.2} Let $\mathrm{Sh}_{\mathrm{Nilp}_\mathrm{fib}}(\Bun_G(X/S))$ denote the subcategory of $\mathrm{Sh}(\Bun_G(X/S))$ of sheaves $\mathcal{F}$ such that for all $s\in S$ the sheaf $i_s^*\mathcal{F}\in \mathrm{Sh}_\mathrm{Nilp}(\Bun_G(X_s))$.

Recall the moduli space of local systems for a family of curves, see Appendix \ref{LocSys in families}. The goal in this subsection is to construct a kernel $R_S\in \mathrm{QCoh}(\mathrm{LS}_{\check{G}}(X/S))\otimes  \mathrm{Rep}(\check{G})_{\mathrm{Ran}(X/S)}^\mathrm{nc}$ such that the composition
\[\mathrm{Sh}(\Bun_G(X/S))\overset{R_S\star -}{\to} \mathrm{QCoh}(\mathrm{LS}_{\check{G}}(X/S))\otimes \mathrm{Sh}(\Bun_G(X/S))\overset{\Gamma_*\otimes \mathrm{id}}{\to}\mathrm{Sh}(\Bun_G(X/S))\]
factors through $\mathrm{Sh}_{\mathrm{Nilp}_\mathrm{fib}}(\Bun_G(X/S))$ and provides a left adjoint to the inclusion 
\[
\mathrm{Sh}_{\mathrm{Nilp}_\mathrm{fib}}(\Bun_G(X/S))\to \mathrm{Sh}(\Bun_G(X/S)).
\]

\begin{rem}
    Here $\Gamma_*:\mathrm{QCoh}(\mathrm{LS}_{\check{G}}(X/S))\to D(e)$ is considered as a map of presheaves of categories on $S_{et}$, where the target is the constant sheaf with value $D(e)$.
\end{rem}

Recall the following result from \cite{arinkin2020stack}:

\begin{thm}\label{kernel over a point}
    For each geometric point $s$ of $S$, there exist an object $\mathrm{R}_s\in \mathrm{QCoh}(\mathrm{LS}_{\check{G}}(X_s)) \otimes\mathrm{Rep}(\check{G})_{\mathrm{Ran}(X_s)}$ such that the functor
    \[\mathrm{Sh}(\Bun_G(X_s))\overset{\mathrm{R}_s\star-}{\to} \mathrm{QCoh}(\mathrm{LS}_{\check{G}}(X_s))\otimes \mathrm{Sh}(\Bun_G(X_s))\overset{\Gamma\otimes \mathrm{id}}{\to} \mathrm{Sh}(\Bun_G(X_s))\]
    factors through the subcategory $\mathrm{Sh}_\mathrm{Nilp}(\Bun_G(X_s))$ and provides a left adjoint to the inclusion into $\mathrm{Sh}(\Bun_G(X_s))$. 
\end{thm}

For every $s\in S$ we refer to the object $\mathrm{R}_s\in \mathrm{Sh}(\Bun_G(X_s))\otimes \mathrm{Rep}(\check{G})_{\mathrm{Ran}(X_s)}$ as the kernel of Beilinson's spectral projector for $X_s$.

Let us recall the explicit formula for $\mathrm{R}_s$ as described in \cite[\S 13.1]{arinkin2020stack}. Let $F:\mathrm{Rep}(\check{G})\to \mathrm{Sh}_{\mathrm{lc}}(X_s)\otimes \mathrm{QCoh}(\mathrm{LS}_{\check{G}}(X_s))$ be the map given by the universal $\check{G}$-local system on $X_s$ and let
\[F^J:\mathrm{Rep}(\check{G})^{\otimes J}\to \mathrm{Sh}_{\mathrm{lc}}(X_s^J)\otimes \mathrm{QCoh}(\mathrm{LS}_{\check{G}}(X_s))\]
be the map given by the $J$-fold tensor product of $F$ and the monoidal structure on $\mathrm{QCoh}(\mathrm{LS}_{\check{G}}(X_s))$. There is an equivalence
\[R_{X_s}\simeq \underset{(I\overset{\phi}{\to} J)\in \mathrm{TwArr}(\mathrm{Fin})}{\colim} \mathrm{ins}_{I\to J} (\mathrm{id}\otimes ( F^J\circ \mathrm{mult}^\phi)(R^{\boxtimes I}_{\mathrm{Rep}(\check{G})}))\in \mathrm{QCoh}(\mathrm{LS}_{\check{G}}(X_s)) \otimes\mathrm{Rep}(\check{G})_{\mathrm{Ran}(X_s)},\]
where $R^{\boxtimes I}\in (\mathrm{Rep}(\check{G})\otimes \mathrm{Rep}(\check{G}))^{\otimes I}$ is the $I$-fold tensor product of the unit of self-duality in the rigid category $\mathrm{Rep}(\check{G})$.

We now bundle these kernels together to a single kernel over $S$, which is non-characteristic relative to $S$.

\begin{prop}\label{bundleing the kernels}
    There exist an object $\mathrm{R}_S\in \mathrm{QCoh}(\mathrm{LS}_{\check{G}}(X/S))\otimes \mathrm{Rep}(\check{G})_{\mathrm{Ran}(X/S)}^\mathrm{nc}$ equipped with an equivalence between the image of $\mathrm{R}_S$ under the map
    \[i_s^*:\mathrm{QCoh}(\mathrm{LS}_{\check{G}}(X/S))\otimes \mathrm{Rep}(\check{G})_{\mathrm{Ran}(X/S)}^\mathrm{nc}\to \mathrm{QCoh}(\mathrm{LS}_{\check{G}}(X_s))\otimes  \mathrm{Rep}(\check{G})_{\mathrm{Ran}(X_s)}\]
    and $\mathrm{R}_s$, for all points $s$ of $S$. 
\end{prop}

\begin{proof}
    For each contractible $S'\to S$, the universal $\check{G}$-local system on $\mathrm{LS}_{\check{G}}(X\times_S S')$ defines a natural transformation
\[\mathrm{Rep}(\check{G})\to \mathrm{Sh}_{\mathrm{lc}}(X\times_S S')\otimes \mathrm{QCoh}(\mathrm{LS}_{\check{G}}(X\times_S S')).\]
     These maps are compatible with the restriction map in $S'$ and thus define a map of presheaves of categories
     \[F:\mathrm{Rep}(\check{G})\to \mathrm{Sh}_{\mathrm{lc}}(X/S)\otimes \mathrm{QCoh}(\mathrm{LS}_{\check{G}}(X/S)),\]
      where $\mathrm{Rep}(\check{G})$ is considered as a constant presheaf. Taking the $J$-fold tensor product of $F$ and using the monoidal structure on $\mathrm{QCoh}(\mathrm{LS}_{\check{G}}(X/S))$, we obtain a map of presheaves
    \[\mathrm{Rep}(\check{G})^{\otimes J}\to (\mathrm{Sh}_{\mathrm{lc}}(X/S))^{\otimes J}\otimes \mathrm{QCoh}(\mathrm{LS}_{\check{G}}(X/S)).\]
    The exterior product defines a map of sheaves $ (\mathrm{Sh}_{\mathrm{lc}}(X/S))^{\otimes J}\to  \mathrm{Sh}_{\mathrm{lc}}(X^{J/S}/S)$. Thus we obtain a map of sheaves
    \[\mathrm{Rep}(G)^{\otimes J}\to \mathrm{Sh}_{\mathrm{lc}}(X^{J/S}/S)\otimes \mathrm{QCoh}(\mathrm{LS}_{\check{G}}(X/S))\]
    Composing with the inclusion $\mathrm{Sh}_{\mathrm{lc}}(X^{J/S}/S)\to \mathrm{Sh}^\mathrm{nc}(X^{J/S}/S)$, we obtain a map
    \[F^J: \mathrm{Rep}(\check{G})^{\otimes J}\to \mathrm{Sh}^\mathrm{nc}(X^{J/S}/S)\otimes \mathrm{QCoh}(\mathrm{LS}_{\check{G}}(X/S)).\]
    We then define the object $R_S$ by the following formula
    \[R_S:= \underset{(I\overset{\phi}{\to} J)\in \mathrm{TwArr}(\mathrm{Fin})}{\colim}\mathrm{ins}_{I\to J} (\mathrm{id}\otimes (F^J\circ \mathrm{mult}^\phi)(R^{\boxtimes I}_{\mathrm{Rep}(\check{G})})).\]
    By comparing with the formula for $R_s$, we obtain an obvious equivalence $i_s^*R_S\simeq R_s$ for all points $s$ of $S$. 
\end{proof}

\begin{prop}
The composition
    \[\mathrm{Sh}(\Bun_G(X/S))\overset{R_S\star -}{\to} \mathrm{QCoh}(\mathrm{LS}_{\check{G}}(X/S))\otimes \mathrm{Sh}(\Bun_G(X/S))\overset{\Gamma_*\otimes \mathrm{id}}{\to}\mathrm{Sh}(\Bun_G(X/S))\]
factors through $\mathrm{Sh}_{\mathrm{Nilp}_\mathrm{fib}}(\Bun_G(X/S))$ and provides a left adjoint to the inclusion $\mathrm{Sh}_{\mathrm{Nilp}_\mathrm{fib}}(\Bun_G(X/S))\to \mathrm{Sh}(\Bun_G(X/S))$.
\end{prop}

\begin{proof}
    Let $L:\mathrm{Sh}(\Bun_G(X/S))\to \mathrm{Sh}(\Bun_G(X/S))$ denote the composition. We start by proving that $L$ factors through the subsheaf $\mathrm{Sh}_{\mathrm{Nilp}_\mathrm{fib}}(\Bun_G(X/S))$. Since the sheaves are defined on the category of contractible test object, we may assume $S$ is contractible and consider the category of global sections. Since $S$ is contractible, on global sections the functor $\Gamma_*:\mathrm{QCoh}(\mathrm{LS}_{\check{G}}(X/S))\to D(e)$ factors as
    \[\mathrm{QCoh}(\mathrm{LS}_{\check{G}}(X/S))\overset{\simeq}{\to}\mathrm{QCoh}(\mathrm{LS}_{\check{G}}(X_s))\overset{\Gamma_{s*}}{\to} D(e).\]
    Since pullback $i^*_s$ intertwines the Hecke action, it follows that the map of categories
    \[\mathrm{Sh}(\Bun_G(X/S))\overset{R_S\star -}{\to} \mathrm{QCoh}(\mathrm{LS}_{\check{G}}(X/S))\otimes \mathrm{Sh}(\Bun_G(X/S))\overset{\Gamma_*\otimes \mathrm{id}}{\to}\mathrm{Sh}(\Bun_G(X/S))\overset{i^*_s}{\to} \mathrm{Sh}(\Bun_G(X_s))\]
    is equivalent to
    \[\mathrm{Sh}(\Bun_G(X/S))\to \mathrm{Sh}(\Bun_G(X_s))\overset{R_s\star -}{\to} \mathrm{QCoh}(\mathrm{LS}_{\check{G}}(X_s))\otimes \mathrm{Sh}(\Bun_G(X_s))\overset{\Gamma_{s*}\otimes \mathrm{id}}{\to}\mathrm{Sh}(\Bun_G(X_s)).\]
    By Theorem \ref{kernel over a point}, this composition factors through $\mathrm{Sh}_\mathrm{Nilp}(\Bun_G)(X_s)$.

    To see that $L$ provides a left adjoint, we start by producing the unit of the adjunction. By the colimit formula for $R_S$ we obtain a map from the term indexed by $(\emptyset\to \emptyset)\in \mathrm{TwArr}(\mathrm{Fin})$, this gives a map
    \[\mathcal{O}_{\mathrm{LS}_{\check{G}}(X/S)} \otimes \mathbf{1}_{\mathrm{Rep}(\check{G})_{\mathrm{Ran}(X/S)}^\mathrm{nc}}\to R_S\]
    in the category $\mathrm{QCoh}(\mathrm{LS}_{\check{G}}(X/S))\otimes \mathrm{Rep}(\check{G})_{\mathrm{Ran}(X/S)}^\mathrm{nc}$. Thus we obtain a natural transformation to $L$ from the following composition
    \[\mathrm{Sh}(\Bun_G(X/S))\overset{\mathcal{O}_{\mathrm{LS}_{\check{G}}(X/S)} \otimes \mathrm{id}}{\to} \mathrm{QCoh}(\mathrm{LS}_{\check{G}}(X/S))\otimes \mathrm{Sh}(\Bun_G(X/S))\overset{\Gamma_*\otimes \mathrm{id}}{\to}\mathrm{Sh}(\Bun_G(X/S)).\]
    The unit map $e\to \Gamma(\mathrm{LS}_{\check{G}}(X/S),\mathcal{O}_{\mathrm{LS}_{\check{G}}(X/S)})$ provides a natural transformation from $\mathrm{id}_{\mathrm{Sh}(\Bun_G(X/S))}$ to this composition above. This provides the candidate for the unit 
    \[\mathrm{id}_{\mathrm{Sh}(\Bun_G(X/S))}\to L.\]
    To prove that it is the unit in an adjunction, we have to prove that it is an equivalence when evaluated on objects of $\mathrm{Sh}_{\mathrm{Nilp}_\mathrm{fib}}(\Bun_G(X/S))$. It suffices to prove this after pullback to the fiber above point $s\in S$, thus the statement follows from the pointwise result of \cite{arinkin2020stack} Corollary 18.2.9.
\end{proof}

\subsubsection{The kernel in the ind-constructible setting} The same constructions can be carried out in the setting of ind-constructible sheaves. In this case we replace the sheaf $\mathrm{Sh}_{\mathrm{lc}}(X/S)$ by $\mathrm{QLisse}(X/S)$. The role of $\mathrm{LS}_{\check{G}}$ is replaced by $\mathrm{LS}_{\check{G}}^\mathrm{restr}$, the moduli space of $\check{G}$-local systems with restricted variation as defined in Appendix \ref{LocSys in families}.

In the ind-constructible sheaf context we have two notions of non-characteristicness (see Appendix \ref{A:non-char} for the two notions). The sheaf of categories  $\mathrm{Rep}(\check{G})_{\mathrm{Ran}(X/S)}^\mathrm{con,nc}$ will refer to the ind-non-characteristic variant. That is $\mathrm{Rep}(\check{G})_{\mathrm{Ran}(X/S)}^\mathrm{con,nc}$ denotes the colimit of the functor $\mathrm{TwArr}(\mathrm{Fin})\to \mathrm{Pr}^L_e$ given by
\[(I\overset{\phi}{\to} J)\mapsto \mathrm{Rep}(\check{G})^{\otimes I}\otimes \mathrm{Sh}^\mathrm{con,nc}(X^{J/S})\]
considered as a presheaf of categories on $S_{et}$, where $\mathrm{Sh}^\mathrm{con}(X^{J/S})^\mathrm{nc}$ is the category of ind-non-characteristic sheaves.

Thus by the analogous formula as described above, we obtain a kernel
\[\mathrm{R}_S^\mathrm{con}\in \mathrm{QCoh}(\mathrm{LS}^\mathrm{restr}_{\check{G}}(X/S))\otimes \mathrm{Rep}(\check{G})_{\mathrm{Ran}(X/S)}^\mathrm{con,nc}.\]

Here we use that assumption that $X$ has genus at least $1$ and the result that $\mathrm{QLisse}(X/S)\simeq \mathrm{IndLisse}(X/S)$, see \cite[Prop. E.2.8]{arinkin2020stack}.

\subsection{The ind-constructible spectral projector and the universal nilpotent cone} We will prove that the ind-constructible spectral projector sends ind-non-characteristic sheaves to sheaves with universal nilpotent singular support.

\subsubsection{}
In this section we take $S$ to be a complex analytic manifold. We will only apply the results of this section in the case $S=\mathcal{M}_g$, however. See Appendix B for the definition of the ind-constructible sheaf context.

\begin{prop}\label{non-char ind-con projector}
The action map
    \[\mathrm{R}_{S}^\mathrm{con}\star-:\mathrm{Sh}^\mathrm{con}(\Bun_G(X/S)) \to \mathrm{QCoh}(\mathrm{LS}_{\check{G}}^\mathrm{restr}(X/S))\otimes \mathrm{Sh}^\mathrm{con}(\Bun_G(X/S))\]
restricts to a map
\[\mathrm{R}_{S}^\mathrm{con}\star-:\mathrm{Sh}^{\mathrm{con},\mathrm{nc}}(\Bun_G(X/S)) \to \mathrm{QCoh}(\mathrm{LS}_{\check{G}}^\mathrm{restr}(X/S))\otimes \mathrm{Sh}_{\mathrm{Nilp}_u}^\mathrm{con}(\Bun_G(X/S)).\]
\end{prop}

Fix a locally constant cosheaf $\mathcal{Z}$ of derived stacks on $S_{et}$, such that, for contractible test objects $T\to S$, the category $\mathrm{QCoh}(\mathcal{Z}(T/S))$ is dualizable and $\mathcal{O}_{\mathcal{Z}(T/S)}$ is a compact object. Fix a map of cosheaves of stacks $f:\mathcal{Z}\to \mathrm{LS}^\mathrm{restr}_{\check{G}}(X/S)$, and let $R_{S,\mathcal{Z}}^\mathrm{con}$ denote the image of $R_S^\mathrm{con}$ under the map
\[f^*\otimes \mathrm{id}:\mathrm{QCoh}(\mathrm{LS}^\mathrm{restr}_{\check{G}}(X/S))\otimes \mathrm{Rep}(\check{G})_{\mathrm{Ran}(X/S)}^\mathrm{con,nc}\to \mathrm{QCoh}(\mathcal{Z})\otimes \mathrm{Rep}(\check{G})_{\mathrm{Ran}(X/S)}^\mathrm{con,nc}.\]
We prove the following result.

\begin{prop}
The action map
    \[\mathrm{R}_{S,\mathcal{Z}}^\mathrm{con}\star-:\mathrm{Sh}^\mathrm{con}(\Bun_G(X/S)) \to \mathrm{QCoh}(\mathcal{Z})\otimes \mathrm{Sh}^\mathrm{con}(\Bun_G(X/S))\]
restricts to a map
\[\mathrm{R}_{S,\mathcal{Z}}^\mathrm{con}\star-:\mathrm{Sh}^{\mathrm{con},\mathrm{nc}}(\Bun_G(X/S)) \to \mathrm{QCoh}(\mathcal{Z})\otimes \mathrm{Sh}_{\mathrm{Nilp}_u}^\mathrm{con}(\Bun_G(X/S)).\]
\end{prop}

This is a statement about maps of presheaves of categories on $S_\mathrm{et}$ and it suffices to check that the values on each contractible test object $S'\to S$ factors through the desired category. Thus we now assume that $S$ is contractible and consider the categories of global sections over $S$ of the corresponding presheaves of categories. By Proposition \ref{kernel over a point} and Proposition \ref{bundleing the kernels}, for each point $s\in S$, the composition
\[\mathrm{Sh}^\mathrm{con}(\Bun_G(X/S)) \overset{\mathrm{R}_{S,\mathcal{Z}}^\mathrm{con}\star-}{\to} \mathrm{QCoh}(\mathcal{Z})\otimes \mathrm{Sh}^\mathrm{con}(\Bun_G(X/S))\overset{\mathrm{id}\otimes i_s^*}{\to} \mathrm{QCoh}(\mathcal{Z})\otimes \mathrm{Sh}^\mathrm{con}(\Bun_G(X_s))\]
factors through the subcategory $\mathrm{QCoh}(\mathcal{Z})\otimes \mathrm{Sh}^\mathrm{con}_\mathrm{Nilp}(\Bun_G(X_s))\to \mathrm{QCoh}(\mathcal{Z})\otimes \mathrm{Sh}^\mathrm{con}(\Bun_G(X_s))$. Thus, since $\mathrm{QCoh}(\mathcal{Z})$ is dualizable, by Lemma \ref{criterion for universal nilpotence ind-version} it suffices to prove that $\mathrm{R}_{S,\mathcal{Z}}^\mathrm{con}\star-$ factors through the subcategory $\mathrm{QCoh}(\mathcal{Z})\otimes \mathrm{Sh}^{\mathrm{con},\mathrm{nc}}(\Bun_G(X/S))$.

Since $R^\mathrm{con}_{S,\mathcal{Z}}\in \mathrm{QCoh}(\mathcal{Z})\otimes  \mathrm{Rep}(\check{G})_{\mathrm{Ran}(X/S)}^{\mathrm{con},\mathrm{nc}}$ and the subcategory of ind-non-characteristic sheaves is closed under colimit, it suffices to prove that the Hecke action by any object of the category $\mathrm{Rep}(\check{G})_{\mathrm{Ran}(X/S)}^{\mathrm{con},\mathrm{nc}}$ on an object of $\mathrm{Sh}^{\mathrm{con},\mathrm{nc}}(\Bun_G(X/S))$ is again ind-non-characteristic. 

\begin{lem}
    For any point $s\in S$, $\mathcal{F}_s\in \mathrm{Sh}(\Bun_G(X_s))$ and $V\in \mathrm{Rep}(\check{G})$, the sheaf $H(V,\mathcal{F}_s)\in \mathrm{Sh}(\Bun_G(X_s)\times X_s)$ is non-characteristic with respect to $X_s$.
\end{lem}

\begin{proof}

    The argument follows the logic of \cite[Appendix B.]{gaitsgory2022toy}, and we will refer \emph{loc.cit} for the notation.

    Suppose $v=(0,\xi_X)$ is a cotangent vector of $\Bun_G(X_s)\times X_s$ pulled back from $X_s$, which is in the singular support of $H(V,\mathcal{F}_s)$. Since $v$ is in the singular support of $H(V,\mathcal{F}_s)$, it follows that there exist  a cotangent vector $\xi'$ of $\Bun_G(X_s)$ and a cotangent vector $\xi_H$ of $\mathrm{Hecke}_{X_s}^\mathrm{loc}$ in the singular support of $\mathrm{Sat}(V)\in \mathrm{Sh}(\mathrm{Hecke}_{X_s}^\mathrm{loc})$ such that
\[d\overset{\leftarrow}{h}(\xi')+d\mathfrak{r}(\xi_H)=d\mathfrak{s}(\xi_X),\]
    where we use the notation as in \cite{gaitsgory2022toy} Appendix B. In the notation of \cite{gaitsgory2022toy} we have $\xi''=0$. Let $x$ be a point of $X_s$ and $D(x)^\circ$ be the punctured formal disk around $x$. By the same argument as in loc. cit. it follows that $\xi'|_{D(x)^\circ}=\xi''|_{D(x)^\circ}=0$, hence $\xi'=0$. Thus we are in the situation described below B.19 in loc. cit. and hence conclude that $\xi_X=0$. So $v=0$ and $H(V,\mathcal{F}_s)$ is non-characteristic with respect to $X_s$.
\end{proof}

\begin{prop}\label{noncharHecke}
    Let $K\in \mathrm{Rep}(\check{G})_{\mathrm{Ran}(X/S)}^{\mathrm{con},\mathrm{nc}}$ be an ind-non-characteristic Hecke kernel. If $\mathcal{F}\in \mathrm{Sh}^{\mathrm{con}}(\Bun_G(X/S))^\mathrm{nc}$, then $K\star \mathcal{F}\in \mathrm{Sh}^{\mathrm{con,nc}}(\Bun_G(X/S))$.
\end{prop}
\begin{proof}
The category $\mathrm{Rep}(\check{G})_{\mathrm{Ran}(X/S)}^\mathrm{nc}$ is generated under colimits by objects of the form $\text{ins}_{I\to J}V\otimes \mathcal{S}$, where $V\in \mathrm{Rep}(\check{G})^{\otimes I}$ is a simple tensor, say $V\simeq \otimes_{i\in I} V_i$ with $V_i\in \mathrm{Rep}(\check{G})$, and $\mathcal{S}\in \mathrm{Sh}(X^{J/S})^\mathrm{nc}$. Thus we may assume $K$ is of this form. We may also assume that $\mathcal{F}$ is constructible. The action of $K$ is given by the following composition
\begin{align*}
    \mathrm{Sh}(\Bun_G)\overset{H(V,-)}{\to} &\mathrm{Sh}(\Bun_G\times_S X^{I/S})\overset{(\mathrm{
id}\times \Delta_\phi)^*}{\to }\mathrm{Sh}(\Bun_G\times_S X^{J/S})\\
&\overset{\otimes\mathrm{pr}_2^*\mathcal{S}}{\to }\mathrm{Sh}(\Bun_G\times_S X^{J/S})\overset{\mathrm{pr}_{\Bun_G!}}{\to} \mathrm{Sh}(\Bun_G).
\end{align*}
To prove that this sends $\mathcal{F}$ to a sheaf non-characteristic over $S$, it suffices to prove that $\mathcal{F}$ is sent to a sheaf that is non-characteristic relative to $X^{J/S}$ under the composition
\[ \mathrm{Sh}(\Bun_G)\overset{H(V,-)}{\to}\mathrm{Sh}(\Bun_G\times_S X^{I/S})\overset{(\mathrm{id}\times \Delta_\phi)^*}{\to }\mathrm{Sh}(\Bun_G\times_S X^{J/S}).\]
This composition is equivalent to $H(W,-)$, where $W\simeq \mathrm{mult}_\phi V$. So it suffices to prove that $H(W,\mathcal{F})$ is $X^{J/S}$-non-characteristic. Since $W\simeq \otimes_{j\in J} W_j$ is a simple tensor, we may write $H(W,-)$ as a composition of $H(W_j,-)$ and thus reduce to proving the statement when $|J|=1$.

In this case we see that $H(W,\mathcal{F})$ is non-characteristic relative to $S$, since it is given by a smooth pullback followed by a proper pushforward applied to an $S$-non-characteristic sheaf. Thus to see that $H(W,\mathcal{F})$ is $X$-non-characteristic it suffices to check the statement fiberwise on $S$, but this is exactly the statement of the previous lemma.
\end{proof}

From this proposition, we conclude that Beilinson's spectral projector sends ind-non-characteristic sheaves to ind-non-characteristic sheaves. In particular we have the following corollary:

\begin{cor}\label{creating nilpotent ind-constructible sheaves}
    The action of $R^\mathrm{con}_{S,\mathcal{Z}}$ provides a left adjoint to the inclusion
    \[i:\mathrm{QCoh}(\mathcal{Z})\otimes \mathrm{Sh}_{\mathrm{Nilp}_u}^\mathrm{con}(\Bun_G(X/S))\to \mathrm{QCoh}(\mathcal{Z})\otimes \mathrm{Sh}^{\mathrm{con},\mathrm{nc}}(\Bun_G(X/S)).\]
\end{cor}
\begin{proof}
    First we prove that the spectral projector factors through the category of sheaves with universal nilpotent singular support. Since the spectral projector sends ind-non-characteristic sheaves to ind-non-characteristic sheaves, by Lemma \ref{criterion for universal nilpotence ind-version} it suffices to prove the statement fiberwise on $S$. Over each fiber, this is \cite[Corollary 15.5.4]{arinkin2020stack}.

    The colimit formula for $R_{S,\mathcal{Z}}$ provides a map from the unit object in $\mathrm{QCoh}(\mathcal{Z})\otimes  \mathrm{Rep}(\check{G})_{\mathrm{Ran}(X/S)}^{\mathrm{con},\mathrm{nc}}$. Thus we obtain a natural transformation $\mathrm{id}\to R_{S,\mathcal{Z}}$. To prove that functor is a left adjoint, it suffices to prove that the natural transformation is an equivalence for objects in $\mathrm{QCoh}(\mathcal{Z})\otimes \mathrm{Sh}_{\mathrm{Nilp}_u}^\mathrm{con}(\Bun_G(X/S))$. Since equivalences can be checked fiberwise on $S$, we are reduced to proving the statement when $S$ is a point. In this case, the statement follows from \cite[Corollary 15.5.3]{arinkin2020stack}.
\end{proof}

\begin{proof}[Proof of Proposition \ref{non-char ind-con projector}]
    We can assume $S$ contractible. In this case it follows that $\mathrm{LS}_{\check{G}}^\mathrm{restr}(X/S)$ is a colimit of stacks $i_{\cZ}:\mathcal{Z}\to \mathrm{LS}_{\check{G}}^\mathrm{restr}(X/S)$ such that $\mathcal{O}_\mathcal{Z}$ is compact. It follows that $R_S^{\on{con}}$ is a colimit over $\cZ$ of functors
    \[
    \mathrm{R}_{S,\cZ}^\mathrm{con}\star-:\mathrm{Sh}^{\mathrm{con},\mathrm{nc}}(\Bun_G(X/S)) \to \mathrm{QCoh}(\cZ)\otimes \mathrm{Sh}_{\mathrm{Nilp}_u}^\mathrm{con}(\Bun_G(X/S))\xrightarrow{i_{\cZ,*}\otimes\on{id}}
    \]
    \[
    \to \mathrm{QCoh}(\mathrm{LS}_{\check{G}}^\mathrm{restr}(X/S))\otimes \mathrm{Sh}_{\mathrm{Nilp}_u}^\mathrm{con}(\Bun_G(X/S)).
    \]
\end{proof}

\subsection{Back to the Betti spectral projector} We continue to assume that $S$ is a complex manifold. We will now return to studying Beilinson's spectral projector $R_S$ introduced in \S \ref{ss:10.2}. The goal of this subsection is to prove that $R_S$ sends ind-non-characteristic sheaves on $\on{Bun}_G(X/S)$ to sheaves with universal nilpotent singular support. That is:

\begin{prop} \label{nilpontece of the spectral projector}
    The composition
    \[\mathrm{Sh}^{\mathrm{con},\mathrm{nc}}(\Bun_G(X/S))\to \mathrm{Sh}^{\mathrm{nc}}(\Bun_G(X/S))\overset{R_S\star-}{\to} \mathrm{QCoh}(\mathrm{LS}_{\check{G}}(X/S))\otimes  \mathrm{Sh}(\Bun_G(X/S))\]
    factors through $\mathrm{QCoh}(\mathrm{LS}_{\check{G}}(X/S))\otimes  \mathrm{Sh}_{\mathrm{Nilp}_u}(\Bun_G(X/S))$. In other words, the Betti spectral projector sends ind-constructible, ind-non-characteristic sheaves to sheaves with universal nilpotent singular support.
\end{prop}

\begin{rem}
    We expect that the Betti spectral projector applied to any non-characteristic Betti sheaf (not necessarily ind-constructible) has universal nilpotent singular support, and hence defines a left adjoint to the inclusion $\on{Sh}_{\on{Nilp}_u}(\on{Bun}_G(X/S))\into \on{Sh}^{\on{nc}}(\on{Bun}_G(X/S))$. However, we are only able to prove that it sends ind-non-characteristic constructible sheaves to Betti sheaves with universal nilpotent singular support.
\end{rem}

The idea is that the ind-constructible spectral projector applied to an ind-constructible sheaf and realized as a Betti sheaf is essentially an extension of coefficients from the Betti spectral projector applied to the same sheaf. Then the point is that extending coefficients along a surjective map does not change the singular support. 

More precisely, note that by unwinding the formulas, the following two compositions
\[
    \mathrm{Sh}^\mathrm{con}(\Bun_G(X/S)) \overset{R_S^\mathrm{con}\star-}{\to }  \mathrm{QCoh}(\mathrm{LS}^\mathrm{restr}_{\check{G}}(X/S))\otimes  \mathrm{Sh}^\mathrm{con}(\Bun_G(X/S))
    \]
    \[
    \to \mathrm{QCoh}(\mathrm{LS}^\mathrm{restr}_{\check{G}}(X/S))\otimes  \mathrm{Sh}(\Bun_G(X/S))
\]
and 
\begin{align*}
    \mathrm{Sh}^\mathrm{con}(\Bun_G(X/S)) \to  & \mathrm{Sh}(\Bun_G(X/S)) \overset{R_S\star-}{\to}  \mathrm{QCoh}(\mathrm{LS}_{\check{G}}(X/S))\otimes  \mathrm{Sh}(\Bun_G(X/S)) \\
    \overset{i^*\otimes \mathrm{id}}{\to} & \mathrm{QCoh}(\mathrm{LS}^\mathrm{restr}_{\check{G}}(X/S))\otimes  \mathrm{Sh}(\Bun_G(X/S)) 
\end{align*}
are equivalent, where $i: \mathrm{LS}^\mathrm{restr}_{\check{G}}(X/S)\to \mathrm{LS}_{\check{G}}(X/S)$ is the natural map.

Thus by Proposition \ref{non-char ind-con projector}, the composition 
\begin{equation}\label{eq:comp}
\mathrm{Sh}^\mathrm{con}(\Bun_G(X/S)) \to \mathrm{Sh}(\Bun_G(X/S))\overset{R_S\star -}{\to}\mathrm{QCoh}(\mathrm{LS}_{\check{G}}(X/S))\otimes  \mathrm{Sh}(\Bun_G(X/S)) \overset{i^*\otimes \mathrm{id}}{\to}
\end{equation}
\[
\to \mathrm{QCoh}(\mathrm{LS}_{\check{G}}^\mathrm{restr}(X/S))\otimes  \mathrm{Sh}(\Bun_G(X/S))
\]
factors through $\mathrm{QCoh}(\mathrm{LS}_{\check{G}}^\mathrm{restr}(X/S))\otimes  \mathrm{Sh}_{\mathrm{Nilp}_u}(\Bun_G(X/S))$.

Let $e'/e$ be an extension of algebraically closed fields and $\mathrm{LS}_{\check{G}}^\mathrm{restr,e'}(X/S))$ denote the base change of $\mathrm{LS}_{\check{G}}^\mathrm{restr}(X/S))$ to $e'$. The composition (\ref{eq:comp}) with $e$ replaced by $e'$ similarly factors through $\mathrm{QCoh}(\mathrm{LS}_{\check{G}}^\mathrm{restr,e'}(X/S))\otimes  \mathrm{Sh}_{\mathrm{Nilp}_u}(\Bun_G(X/S))$.

\begin{lem} \label{faithful flatness of compact substacks}
    Suppose $S$ is contractible. The family of functors 
    \[i_{e'}^*:\mathrm{QCoh}(\mathrm{LS}_{\check{G}}(X/S))\to \mathrm{QCoh}(\mathrm{LS}_{\check{G}}^\mathrm{restr,e'}(X/S)),\] 
    for $e'$ ranging over the algebraically closed field extensions of $e$, is jointly conservative.
\end{lem}

\begin{proof}
    Any field valued point of $\mathrm{LS}_{\check{G}}(X/S)$ factors through $\mathrm{LS}_{\check{G}}^{\mathrm{restr},e'}(X/S)$ for some algebraically closed field extension $e'/e$. Since the collection of all field valued points on a derived eventually coconnective stack whose underlying classical stack is Noetherian is jointly conservative (cf. \cite[Lemma 2.6.1.3]{Lurie}), this implies the lemma.
\end{proof}

Using this, we can now prove Proposition \ref{nilpontece of the spectral projector}:

\begin{proof}[Proof of Proposition \ref{nilpontece of the spectral projector}]
We have have already seen that for any $i_{e'}:\mathrm{LS}_{\check{G}}^{\mathrm{restr},e'}(X/S)\to \mathrm{LS}_{\check{G}}(X/S)$, the composition of $R_S\star-$ with $\mathrm{QCoh}(\mathrm{LS}_{\check{G}}(X/S))\otimes  \mathrm{Sh}(\Bun_G(X/S))\overset{i_{e'}^*\otimes\on{id}}{\to } \mathrm{QCoh}(\mathrm{LS}_{\check{G}}^{\mathrm{restr},e'}(X/S))\otimes  \mathrm{Sh}(\Bun_G(X/S))$ factors through $\mathrm{QCoh}(\mathrm{LS}_{\check{G}}^{\mathrm{restr},e'}(X/S))\otimes  \mathrm{Sh}_{\mathrm{Nilp}_u}(\Bun_G(X/S))$. Thus to prove the proposition, we will use Lemma \ref{faithful flatness of compact substacks} to prove that if an object has universally nilpotent singular support after pullback to all $\mathrm{LS}_{\check{G}}^{\mathrm{restr},e'}(X/S)$, then the object itself has universally nilpotent singular support. Consider the following diagram
\[
\begin{tikzcd}
\mathrm{Sh}_{\mathrm{Nilp}_u}(\Bun_G(X/S))\otimes \mathrm{QCoh}(\mathrm{LS}_{\check{G}}(X/S)) \arrow[d, hook] \arrow[r, "\mathrm{id}\otimes i^*_{e'}"]                                           & \mathrm{Sh}_{\mathrm{Nilp}_u}(\Bun_G(X/S))\otimes \mathrm{QCoh}(\mathrm{LS}_{\check{G}}^{\mathrm{restr},e'}(X/S)) \arrow[d, hook]                                                           \\
\mathrm{Sh}(\Bun_G(X/S))\otimes \mathrm{QCoh}(\mathrm{LS}_{\check{G}}(X/S)) \arrow[u, "L\otimes \mathrm{id}" description, bend right=49] \arrow[r, "\mathrm{id}\otimes i^*_{e'}"] & \mathrm{Sh}(\Bun_G(X/S))\otimes \mathrm{QCoh}(\mathrm{LS}_{\check{G}}^{\mathrm{restr},e'}(X/S)) \arrow[u, "{L\otimes \mathrm{id}}" description, bend right=49, shift right]
\end{tikzcd}\]
    where $L:\mathrm{Sh}(\Bun_G(X/S))\to \mathrm{Sh}_{\mathrm{Nilp}_u}(\Bun_G(X/S))$ is the formally defined left adjoint to the inclusion. Note that the vertical downwards arrows are fully faithful, since these categories of quasi-coherent sheaves are dualizable. 

    To see the statement suppose $\mathcal{F}$ is an object of $\mathrm{Sh}(\Bun_G(X/S))\otimes \mathrm{QCoh}(\mathrm{LS}_{\check{G}}(X/S))$ such that $(\mathrm{id}\otimes i_{e'}^*)\mathcal{F}$ lies in the subcategory of universally nilpotent sheaves for all $\mathrm{LS}_{\check{G}}^{\mathrm{restr},e'}(X/S)$. Let $K$ denote the cofiber of the map $\mathcal{F}\to (L\otimes \mathrm{id}) \mathcal{F}$. It follows that $(\mathrm{id}\otimes i_{e'}^*)K\simeq 0$. Since $\mathrm{Sh}(\Bun_G(X/S))$ is dualizable and the maps $i^*_{e'}$ are jointly conservative, it follows that $K\simeq 0$, hence proving that $\mathcal{F}$ has universally nilpotent singular support. 
\end{proof}

\section{Local Constancy of the Automorphic Category}\label{s:prfmainthm}

\subsection{Proof of main theorem}

In this section we prove the local constancy of the automorphic category. The main result is the following theorem, which was first conjectured in this form in \cite{nadler2021automorphic} and before that a weaker form was stated as Conjecture 4.14 in \cite{ben2018betti}.

\begin{thm}\label{localconstancyofautomorphic}
    Let $S$ be a contractible real analytic manifold, $s$ be a point of $S$ and $X/S$ be a smooth proper family of curves. The restriction map
    \[i_s^*:\mathrm{Sh}_{\mathrm{Nilp}_u}(\Bun_G(X/S))\to \mathrm{Sh}_{\mathrm{Nilp}}(\Bun_G(X_s)) \]
    is an equivalence of categories.
\end{thm}

To prove Theorem \ref{localconstancyofautomorphic} we first proving the following result. The proof of this result will be an induction on the semi-simple rank of $G$. Let $\Psi:\Bun_G(X/S)\to S$ denote the structure map.

\begin{thm}\label{localconstancyofhom}
    Let $S$ be a real analytic manifold and $X/S$ a smooth proper family of curves. For any pair of automorphic sheaves $\mathcal{F},\mathcal{G}\in \mathrm{Sh}_{\mathrm{Nilp}_u}(\Bun_G(X/S))$, the mapping sheaf $\Psi_*\mathcal{H}om(\mathcal{F},\mathcal{G})$ is a locally constant sheaf on $S$. Furthermore, for any point $s\in S$, there is an identification of the stalk
    \[\Psi_*\mathcal{H}om(\mathcal{F},\mathcal{G})_s\simeq \mathrm{Map}(i_s^*\mathcal{F},i_s^*\mathcal{G}).\]
\end{thm}

\begin{proof}

Since the statement is local on $S$, we may assume that $X/S$ admits a section $x:S\to X$ and that $S$.

\step For the base case, suppose $G$ has semi-simple rank $0$. In this case $G=T$ is a torus and hence $\mathrm{Nilp}_u$ is simply the zero-section of $T^*\Bun_T(X/S)$. Thus $\mathcal{H}om(\mathcal{F},\mathcal{G})\in \mathrm{Sh}_{\mathbf{0}}(\Bun_T(X/S))$. To see that the pushforward of $\mathcal{H}om(\mathcal{F},\mathcal{G})$ to $S$ is a locally constant sheaf, let $\Bun_T(X/S)^\mathrm{rig}$ denote moduli stack of $T$-bundles on $X/S$ together with a trivialization along the section $x:S\to X$. The stack $\Bun_T(X/S)^\mathrm{rig}$ is representable by a smooth proper family of varieties over $S$. The $m$-fold intersection of $\Bun_T(X/S)$ over $\Bun_T(X/S)^\mathrm{rig}$ is given by $\Bun_T(X/S)^\mathrm{rig}\times T^m$. Since the Jacobian of $X/S$ is smooth and proper over $S$, $\Bun_T(X/S)^\mathrm{rig}\times T^m$ is topologically a locally trivial fibration over $S$ for each $m$. Since $\Bun_T(X/S)^\mathrm{rig}\to \Bun_T(X/S)$ is smooth, it follows by descent that the pushforward of a locally constant sheaf along $\Psi:\Bun_T(X/S)\to S$ is a locally constant sheaf on $S$. The identification of stalks follows since the inclusion of the fiber above a point $s\in S$ is an equivalence on homotopy types.

\step Let $G$ be a reductive group of semi-simple rank $r>0$ and suppose the statement is known for all reductive groups of semi-simple rank strictly less than $r$. We start by using the induction hypothesis to find a collection of objects $\mathcal{F}_i$ of $\mathrm{Sh}_{\mathrm{Nilp}_u}(\on{Bun}_G(X/S))$, for $i\in I$, such that:

\begin{enumerate}
    \item For each $s\in S$, the sheaves $i_s^*\mathcal{F}_i\in \mathrm{Sh}_{\mathrm{Nilp}}(\on{Bun}_G(X_s))$ generate the category under colimits.

    \item For any $\mathcal{G}\in \mathrm{Sh}_{\mathrm{Nilp}_u}(\Bun_G(X/S))$, the sheaf $\Psi_*\mathcal{H}om(\mathcal{F}_i,\mathcal{G})$ is a local system on $S$ with an identification $\Psi_*\mathcal{H}om(\mathcal{F}_i,\mathcal{G})_s\simeq \mathrm{Hom}_{\mathrm{Sh}_{\mathrm{Nilp}}(\on{Bun}_G(X_s))}(i_s^*\mathcal{F}_i,i_s^*\mathcal{G})$.
\end{enumerate}

The collection of sheaves $\mathcal{F}_i$ will consist of two type of sheaves. First the Eisenstein type; that is, the sheaves of the form $\mathrm{Eis}_{P,!}(\mathcal{F})$, where $P$ is a proper parabolic subgroup of $G$ with Levi quotient $M$, and $\mathcal{F}\in \mathrm{Sh}_{\mathrm{Nilp}_u}(\Bun_M(X/S))$. By Proposition \ref{eisnilpotence}, we have $\mathrm{Eis}_{P,!}\mathcal{F}\in \mathrm{Sh}_{\mathrm{Nilp}_u}(\Bun_G(X/S))$.

The second type is the sheaves of the form $V\star \mathrm{Poinc}_{!,\mathrm{Nilp}_u}$, where $\mathrm{Poinc}_{!,\mathrm{Nilp}_u}=R_S\star \mathrm{Poinc}_!$ denotes the Beilinson spectral projector applied to the Poincaré sheaf, and $V\in \operatorname{Rep} \check{G}$ is a dualizable object which acts as a Hecke operator along $x:S\to X$.

We claim that $\mathrm{Poinc}_{!,\mathrm{Nilp}_u}$ and hence $V\star \mathrm{Poinc}_{!,\mathrm{Nilp}_u}$ has universal nilpotent singular support. It suffices to prove this in the universal case $S=\fM_g$. In this case, since the Poincaré sheaf itself is ind-constructible and non-characteristic, see Proposition \ref{poincare non-char}, it follows by Propositions \ref{nilpontece of the spectral projector} and \ref{noncharHecke} that $\mathrm{Poinc}_{!,\mathrm{Nilp}_u}$ lies in $\mathrm{Sh}_{\mathrm{Nilp}_u}(\Bun_G(X/\fM_g))$.

By \cite[Thm A]{faergeman2022non}, the images of these sheaves in $\mathrm{Sh}_{\mathrm{Nilp}}(\on{Bun}_G(X_s))$ generate the target under colimits.

Thus we have to show that for any $\mathcal{G}\in \mathrm{Sh}_{\mathrm{Nilp}_u}(\Bun_G(X/S))$, the sheaves
\[\Psi_*\mathcal{H}om(\mathrm{Eis}_{P,!}\mathcal{F},\mathcal{G})\]
and
\[\Psi_*\mathcal{H}om(V\star \mathrm{Poinc}_{!,\mathrm{Nilp}_u},\mathcal{G})\]
are local systems $S$ and identify their stalks. 

To see locally constancy in the first case, we use the induction hypothesis and note that 
\[ \Psi_*\mathcal{H}om(\mathrm{Eis}_{P,!}\mathcal{F},\mathcal{G})\simeq \Psi_*\mathcal{H}om(\mathcal{F},\mathrm{CT}_{P,*}(\mathcal{G}))\]
Since $\on{CT}_{P,*}$ preserves universal nilpotent singular support, see Remark \ref{r:CT*} and $M$ has strictly smaller semi-simple rank than $G$, it follows by the induction hypothesis that $\Psi_*\mathcal{H}om(\mathcal{F},\mathrm{CT}_*^P( \mathcal{G}))$ is locally constant on $S$. The identification of the stalk at $s$ also follows from the induction hypothesis.

In the second case we note that
\[\Psi_*\mathcal{H}om(V\star \mathrm{Poinc}_{!,\mathrm{Nilp}_u},\mathcal{G})\simeq \Psi_*\mathcal{H}om(\mathrm{Poinc}_!,V^\vee \star \mathcal{G})\simeq \mathrm{Whit}(V^\vee \star \mathcal{G}).\]
In this case the local constancy follows from the local constancy of the Whittaker functional, Proposition \ref{Localconstancyofwhittaker}. The identification of the stalk at $s$ follows from the same proposition.

\step 
Note that the sheaves $\mathcal{F}_i$, for $i\in I$, generate $\mathrm{Sh}_{\mathrm{Nilp}_u}(\Bun_G(X/S))$ under colimits. To see this it suffices to show that for any $\mathcal{G}\in \mathrm{Sh}_{\mathrm{Nilp}_u}(\Bun_G(X/S))$ there exist an $i\in I$ and a non-zero map $\mathcal{F}_i[n_i]\to \mathcal{G}$ for some $n_i\in \mathbb{Z}$. This follows since the contractibility of $S$ and property (2) above imply that $\mathrm{Hom}(\mathcal{F}_i,\mathcal{G})\simeq \mathrm{Hom}(\mathcal{F}_{i,s},\mathcal{G}_s)$ and $\mathcal{F}_{i,s}$, for $i\in I$, generate $\mathrm{Sh}_{\mathrm{Nilp}}(\on{Bun}_G(X_s))$.

Consider any $\mathcal{F},\mathcal{G}\in \mathrm{Sh}_{\mathrm{Nilp}_u}(\Bun_G(X/S))$. Since $\mathcal{F}_i$, $i\in I$, generate $\mathrm{Sh}_{\mathrm{Nilp}_u}(\Bun_G(X/S))$, we can write $\mathcal{F}\simeq \colim \mathcal{F}_i$ for some index category. Thus 
\[\Psi_*\mathcal{H}om(\mathcal{F},\mathcal{G})\simeq \lim \Psi_*\mathcal{H}om(\mathcal{F}_i,\mathcal{G}).\]
Since $\Psi_*\mathcal{H}om(\mathcal{F}_i,\mathcal{G})$ is a locally constant sheaf on $S$ for each $i$, it follows that $\Psi_*\mathcal{H}om(\mathcal{F},\mathcal{G})$ is a locally constant sheaf on $S$. Since $S$ is contractible, for any locally constant sheaf, the stalk at $s$ is equivalent to the global sections over $S$ and hence the stalk of the limit is the limit of the stalks. Thus, the identification of the stalk of $\Psi_*\mathcal{H}om(\mathcal{F},\mathcal{G})$ follows from the identification for each $\Psi_*\mathcal{H}om(\mathcal{F}_i,\mathcal{G})$.

\end{proof}

Given Proposition \ref{localconstancyofhom}, we conclude Theorem \ref{localconstancyofautomorphic} as follows:

\begin{proof}[Proof of Theorem \ref{localconstancyofautomorphic}]
    Since $S$ is contractible, it follows from Theorem \ref{localconstancyofhom} that the global sections of $\Psi_*\mathcal{H}om(\mathcal{F},\mathcal{G})$ is equivalent to the stalk at $s$. Hence from the identification of the stalks of $\Psi_*\mathcal{H}om(\mathcal{F},\mathcal{G})$ we conclude that $i_s^*$ is fully faithful. 

    That $i^*_s$ is essentially surjective then follows, since the images of the sheaves $\mathcal{F}_i$ constructed in the proof of Theorem \ref{localconstancyofhom} under $i_s^*$ generate the target under colimits.
\end{proof}

\section{Additional results}

\subsection{Constancy of the Kostant coefficient}

In this section we prove two additional independent results

First, we prove that the coefficient of the irreducible components of $\on{Nilp}_u$ containing an everywhere regular nilpotent Higgs field on each fiber is locally constant along $S$ in the characteristic cycle of an automorphic sheaf.

\begin{defin}
   Let $\mathrm{Nilp}^\mathrm{reg}_{u}$ denote the open substack of $\mathrm{Nilp}_{u}$ consisting of elements such that on each fiber $s\in S$, they define an everywhere regular nilpotent Higgs field. Note that this stack is smooth over $S$.
\end{defin}

The result is the following.

\begin{prop}\label{kostantconstant}
    Let $K$ be an irreducible component of $\mathrm{Nilp}_u$ and $s$ be a geometric point on $S$. If $K_s$ is an irreducible component of the fiber $i_s^{-1}K$ containing an everywhere regular nilpotent Higgs bundle, then for any automorphic sheaf $\mathcal{F}\in \mathrm{Sh}_\mathrm{Nilp_u}(\mathrm{Bun}_G(X/S))$
    \[\mathrm{mult}(K_s,CC(\mathcal{F}_s))=\mathrm{mult}(K,CC(\mathcal{F})).\]
\end{prop}

To prove the proposition we start by noting the following lemma, which allows us to control the characteristic cycle of the pullback along $i_s$.

\begin{lem}
    For any geometric point $s$ on $S$, the inclusion map $i_s:\mathrm{Bun}_G(X_s)\to \mathrm{Bun}_G(X/S)$ is non-characteristic relative to the universal nilpotent cone. 
\end{lem}

\begin{proof}
    This follows from the non-characteristicness of the universal nilpotent cone relaitve to $S$. If a cotangent vector is in the kernel of the differential of $i_s$, then it is pulled back from $S$. Then the non-characteristicness of the univeral nilpotent cone implies that this vector must be $0$.
\end{proof}

It follows from the above lemma and \cite[Proposition 9.4.3]{kashiwara1990sheaves} that the characteristic cycle of the pullback of a sheaf along $i_s$ is the pullback of the characteristic cycle.

\begin{proof}[Proof of Proposition \ref{kostantconstant}.]
    Write
    \[CC(\mathcal{F})=n_K [K]+\sum_{i\in I} n_i[Z_i]\]
    where $Z_i$ varies over the set of irreducible components of $\mathrm{Nilp}_u$ different from $K$. By the above:
    \[CC(i_s^*\mathcal{F})=n_K [i_s^{-1}K]+\sum_{i\in I} n_i[i_s^{-1} Z_i].\]
    To prove the proposition it suffices to prove that $i_s^{-1} Z_i$ does not contain $K_s$ for any $i\in I$. 
    
    Let $Z$ be an irreducible component of $\mathrm{Nilp}_u$ such that $i_s^{-1}Z$ contains $K_s$. By assumption $K_s\cap \mathrm{Nilp}^\mathrm{reg}_{X_s}$ is non-empty. Hence $Z\cap \mathrm{Nilp}^\mathrm{reg}_{u}$ is an irreducible component of $\mathrm{Nilp}^\mathrm{reg}_{u}$. Similarly $K\cap \mathrm{Nilp}^\mathrm{reg}_{u}$ is an irreducible component of $\mathrm{Nilp}^\mathrm{reg}_{u}$. Since $\mathrm{Nilp}^\mathrm{reg}_{X/S}$ is smooth and these two irreducible components have a non-trivial intersection, the components must be equal. Hence
    \[Z\cap\mathrm{Nilp}^\mathrm{reg}_{u}=K\cap \mathrm{Nilp}^\mathrm{reg}_{u}.\]
    Since $\mathrm{Nilp}^\mathrm{reg}_{u}$ is open in $\mathrm{Nilp}_{u}$, this set is a non-empty open subset of both $Z$ and $K$. Since $Z$ and $K$ are irreducible, it follows that $Z=K$. 
\end{proof}

\subsection{Wobbly bundles are pure of codimension 1}\label{s:codim1} In this subsection, we prove a claim of Drinfeld on the codimension of the complement of very stable bundles in the moduli stack of $G$-bundles.

\subsubsection{} Fix a smooth projective curve $X$ of genus $g\geq 2$. Denote by 
\[
\on{Bun}_G^{\on{vs}}(X)\subset \on{Bun}_G(X)
\]

\noindent the open subset consisting of very stable $G$-bundles; i.e., bundles $\sP_G$ satisfying that any nilpotent Higgs field $\phi\in \Gamma(X,\fg^*_{\sP_G}\otimes\Omega^1_X)$ is zero. Alternatively, the very stable locus is the complement of the image of the projection
\[
\on{Nilp}_X\backslash 0_{\on{Bun}_G(X)}\to \on{Bun}_G(X).
\]

\noindent Since $\on{Nilp}_X\subset T^*(\on{Bun}_G(X))$ is a conic closed Lagrangian, it follows that $\on{Bun}_G^{\on{vs}}(X)$ is a dense open subset. Moreover, it is known (for $g\geq 2$) that very stable bundles are stable.

\subsubsection{} We let $\on{Wob}(X):=\on{Bun}_G(X)\backslash \on{Bun}_G(X)^{\on{vs}}$ denote the complement of the locus of very stable bundles.

In \cite[Remarque 3.6 (ii)]{laumon1988analogue} (which concerns $G=GL_n$), Laumon states that Drinfeld announced in a letter to Deligne that $\on{Wob}(X)$ is of pure codimension one. Since then, a proof has been slow to appear. For $G=GL_2$, a proof appeared in \cite{pal2021wobbly}, and for $G=GL_n$, a proof appeared in \cite{pal2022conjecture}. Both proofs rely on explicit algebro-geometric methods.

\subsubsection{} In this subsection, we give a short proof of this assertion combining non-trivial results of Beilinson-Drinfeld \cite{beilinson1991quantization} and \cite{faergeman2022non}.

\begin{thm}\label{t:pureofcodim1}

$\on{Wob}(X)$ is of pure codimension one.
    
\end{thm}

\noindent In fact, the proof is an immediate consequence of the existence (due to Beilinson-Drinfeld) of a D-module on $\on{Bun}_G(X)$ whose singular support is the entire global nilpotent cone.\\

\emph{Proof of Theorem \ref{t:pureofcodim1}}

By \cite{beilinson1991quantization}, there exists a D-module $\cF\in\on{D-mod}(\on{Bun}_G(X))$ with $\on{SS}(\cF)=\on{Nilp}_X$. Indeed, this is the Hecke eigensheaf associated to a $\check{G}$-local system with an oper structure, see \cite[Prop. 5.1.2]{beilinson1991quantization} (see also \cite[Thm. 6.2.3]{GLC1}\footnote{A priori, we do not know that $\cF$ has regular singularities, so we cannot immediately apply the Riemann-Hilbert correspondence. This fact was proven in \cite{arinkin2020stack}, but there is no need to apply it here.} In fact, $\cF$ is in the heart of the t-structure of $\on{D-mod}(\on{Bun}_G(X))$ and irreducible on each connected component of $\on{Bun}_G(X)$. Indeed, this follows from \cite[Thm. 11.1.41 + Remark 11.1.2.3]{faergeman2022non}.

Assume for contradiction that there exists an irreducible component $\cW\subset \on{Wob}(X)$ of codimension $>1$ in $\on{Bun}_G(X)$. Let $\cW'\subset \on{Bun}_G(X)$ be the union of the irreducible components not equal to $\cW$. Let $(\sP_G,\phi)\in \on{Nilp}_X$ be a nilpotent Higgs bundle with $\phi\neq 0$ and with $\sP_G\in \cW^{\on{sm}}(\bC)$ a point in the smooth locus of $\cW$ (such a Higgs bundle exists since $\cW\subset\on{Wob}(X)$).

To get a contradiction, it suffices to show that $(\sP_G,\phi)\notin \on{SS}(\cF)$. Since $j: \on{Bun}_G(X)^{\on{vs}}\into \on{Bun}_G(X)$ is open dense, it follows that
\[
\cF\simeq j_{!*}(\sL)
\]

\noindent for some irreducible vector bundle $\sL$ on $\on{Bun}_G(X)^{\on{vs}}$ with a integrable connection. Then $\sL$ extends across the embedding
\[
\on{Bun}_G(X)^{\on{vs}}\subset \on{Bun}_G(X)\backslash \cW',
\]

\noindent since the complement of the left-hand side has codimension $>1$. It follows that $(\sP_G,\phi)\notin \on{SS}(\cF)$.

\qed

\begin{rem}\label{r:generalization}
It follows from our main theorem that the natural generalization of Theorem \ref{t:pureofcodim1} holds for $\on{Bun}_G(X/S)$ for a family of smooth projective curves $X\to S$. Indeed, it would suffice to prove this for $S$ contractible. Then, if one knew the existence of a sheaf on $\on{Bun}_G(X/S)$ with singular support equal to $\on{Nilp}_u$, one could run the same argument as above.

By the above argument (and using that $\cF$ has regular singularities cf. \cite{arinkin2020stack}), one finds a sheaf on $\on{Bun}_G(X_s)$ with full nilpotent singular support. By Theorem \ref{localconstancyofautomorphic}, we may extend this to a sheaf on $\on{Bun}_G(X/S)$ with full universal nilpotent singular support.
\end{rem}

\appendix

\section{Real analytic bases}\label{S:analyticbases}

Nadler and Yun's conjecture is formulated in the setting of a holomorphic family of curves over a real analytic base, in this section we review the basic definitions over real analytic bases. The reader only interested in the case where $S$ is holomorphic can safely ignore this section and assume $S$ holomorphic throughout the paper.

\subsection{Moduli problems over real bases} Let $\mathrm{An}_\mathbb{R}$ denote the category of real analytic spaces. A real analytic stack will refer to a sheaf of anima on $\mathrm{An}_\mathbb{R}$ with the etale topology, we denote this category by $\mathrm{Stk}_\mathbb{R}$. Similarly we define $\mathrm{An}_\mathbb{C}$ to be the category of complex analytic spaces and $\mathrm{Stk}_\mathbb{C}$ the category of complex analytic stacks. Sending a complex analytic space to its underlying real analytic spaces defines a functor $(-)_\mathbb{R}:\mathrm{Stk}_\mathbb{C}\to \mathrm{Stk}_\mathbb{R}$.

\begin{defin}
    Let $S$ be a real analytic stack. A family of proper holomorphic curves over $S$ is a map of real analytic stacks $S\to (\mathcal{M}_g)_\mathbb{R}$. We denote by $X\to S$ the pullback of the underlying real analytic stack of the universal genus g curve $\mathcal{C}_g$.
\end{defin}

\subsubsection{}\label{ss:B.101} Given a family of smooth proper holomorphic curves over $S$, we define the real analytic stack $\mathrm{Bun}_G(X/S)$ as the pullback
\[\Bun_G(X/S)\simeq \Bun_G(\mathcal{C}_g/\mathcal{M}_g)_\mathbb{R}\times_{(\mathcal{M}_g)_\mathbb{R}} S.\]
When $S$ is a real analytic manifold, this is a smooth Artin stack.
Note that the functor that sends a complex analytic stack to its category of Betti sheaves factors through the functor $(-)_\mathbb{R}$, thus when $S$ is complex analytic it does not matter if we instead consider the underlying real analytic stack.

Similarly, we define all the usual stacks living over $\mathrm{Bun}_G(X/S)$ by pulling back from $\mathrm{Bun}_G(\mathcal{C}_g/\mathcal{M}_g)$. For example we define the Hecke stack over $\Bun_G(X/S)\times_S X^{I/S}$ to be
\[\mathrm{Hecke}_I(X/S)\simeq \mathrm{Hecke}_I(\mathcal{C}_g/\mathcal{M}_g)_\mathbb{R}\times_{(\Bun_G \times_{\mathcal{M}_g} \mathcal{C}_g^{I/S})_\mathbb{R}} \Bun_G(X/S)\times_S X^{I/S}.\]

In general we will follow the convention that all moduli problems are considered in the holomorphic setting,  the real analytic variant is then obtained by pulling back the underlying real stack of the universal case along the given map $S\to \mathcal{M}_{g,\mathbb{R}}$. 

\section{Sheaf Contexts}

Two different sheaf contexts will appear in this paper. Our main theorem will be formulated in the context of all Betti sheaves, which we will simply denote by $\mathrm{Sh}(-)$. However, at a technical point of the proof we will need to make use of a family version of the ind-constructible sheaf context.

In the family setting for stacks over a fixed base $S$, the ind-constructible sheaf context is most naturally considered as a presheaf of categories on the category of contractible open subsets of $S$.

\subsection{Presheaves of presentable categories}

For a real analytic smooth Deligne-Mumford stack $S$, let $S_{et}$ denote the category of contactible real analytic manifolds $T$ with an etale morphism $T\to S$. By a presheaf of presentable categories on $S$, we mean a functor from $S_{et}^\mathrm{op}$ with values in the $(\infty,1)$-category $\mathrm{Pr}^L_e$. Denote this category by $\mathrm{PSh}(S_{et},\mathrm{Pr}^L_e)$.

For an $(\infty,1)$-category $\mathcal{C}$, we denote the category of all presheaves on $S_\mathrm{et}$ with values in $\mathcal{C}$ by $\mathrm{PSh}(S_\mathrm{et},\mathcal{C})$.

\begin{defin}
    For $\mathcal{C}$ an $(\infty,1)$-category, a presheaf $\mathcal{F}:S_{et}\to \mathcal{C}$ is said to be locally constant if, for all $T'\to T$ over $S$ with $T'$ and $T$ contractible analytic varieties, the restriction map
    \[\mathcal{F}(T)\to \mathcal{F}(T')\]
    is an equivalence. Let $\mathrm{PSh}_{lc}(S_\mathrm{et},\mathcal{C})$ denote the full subcategory spanned by locally constant presheaves.
\end{defin}

Suppose $\mathcal{C}$ admits limits. To comfort the reader, we point out that such a presheaf is in fact automatically a sheaf and that we did not loose information by only considering contractible test objects. The reader who does not need comfort can safely skip this as the mathematical content is irrelevant for the paper.

\begin{defin}
    Let $\mathrm{Sh}(S,\mathcal{C})$ denote the category of all sheaves on the etale site of $S$ with values on $\mathcal{C}$. Let $\mathrm{Sh}_{lc}(S,\mathcal{C})$ denote the subcategory of locally constant sheaves, ie sheaves which locally on $S$ are equivalent to a constant sheaf.
\end{defin}

 \begin{lem}
     The restriction functor from $\mathrm{Sh}_{lc}(S,\mathcal{C})$ to $\mathrm{PSh}(S_\mathrm{et},\mathcal{C})$ factors through $\mathrm{Sh}_{lc}(S,\mathcal{C})$ and defines an equivalence of categories
     \[\mathrm{Sh}_{lc}(S,\mathcal{C})\simeq \mathrm{PSh}_{lc}(S_\mathrm{et},\mathcal{C}).\]
 \end{lem}

 \begin{proof}
We may assume that $S$ is connected. Let $\mathcal{F}$ be an object of $\mathrm{Sh}_{lc}(S,\mathcal{C})$ and suppose it is locally equivalent to the constant sheaf associated to an object $c\in \mathcal{C}$. Let $T\to S$ be an etale morphism of analytic varieties. By covering $T$ with by a cover on which $\mathcal{F}$ is constant, it follows from the sheaf condition for this cover that
\[\mathcal{F}(T)\simeq \lim_{\Pi_\infty T} F\]
where $\Pi_\infty T$ denotes the underlying homotopy type of $T$ and $F:\Pi_\infty T\to \mathcal{C}$ is some functor. In particular, if $T'\to T$ is a map between contractible varieties over $S$, then the restriction map
\[\mathcal{F}(T)\to \mathcal{F}(T')\]
is an equivalence in $\mathcal{C}$. Thus the restriction functor factors through $\mathrm{PSh}_{lc}(S_\mathrm{et}, \mathcal{C})$.

To see that the restriction functor induces an equivalence of categories, note that covers by contractible varieties is a basis for the analytic etale topology. Thus the restriction functor is an equivalence on sheaf categories with values in $\mathcal{C}$. It remains to see that the sheaf condition is automatically satisfied for a locally constant presheaf on the site of contractible test objects.

Let $\mathcal{F}$ be a locally constant presheaf on $S_\mathrm{et}$. Let $T_i\to T$ be a cover of $T$ by contractible analytic varieties and let $\mathcal{S}$ denote the sieve generated by this cover, then we want to prove that the map
\[\mathcal{F}(T)\to \lim_{(T'\to T)\in \mathcal{S}^\mathrm{op}} \mathcal{F}(T')\]
is an equivalence. By assumption all restriction maps in the diagram on the right are equivalences, thus the limit is determined by the geometric realization of the $\mathcal{S}^\mathrm{op}$. This realization is equivalent to the homotopy type of $T$ and hence contractible, thus the map above is an equivalence.
 \end{proof}

\subsubsection{All Betti sheaves} Given a morphism of stacks $\mathcal{X}\to S$, consider the presheaf of categories on $S$ defined by the assignment
\[(T\to S)\mapsto \mathrm{Sh}(\mathcal{X}\times_ST)\]
with $*$-pullback functoriality. This defines an object of $\mathrm{PSh}(S_{et},\mathrm{Pr}^L_e)$, which we will denote by $\mathrm{Sh}(\mathcal{X}/S)$. For a morphism $f:\mathcal{Y}\to \mathcal{X}$ over $S$, we obtain a morphism of sheaves of categories $f^*:\mathrm{Sh}(\mathcal{X}/S)\to \mathrm{Sh} (\mathcal{Y}/S)$. By base change between $!$-pushforward and $*$-pullback, we also obtain a map $f_!:\mathrm{Sh}(\mathcal{Y}/S)\to \mathrm{Sh}(\mathcal{X}/S)$.

The natural transformation $\mathrm{Sh}^\mathrm{con}(-)\to \mathrm{Sh}(-)$ induces a map $\mathrm{Sh}^{\mathrm{con}}(\mathcal{X}/S)\to \mathrm{Sh}(\mathcal{X}/S)$, which intertwines with the $*$-pullback and $!$-pushforward functors for all Betti sheaves. 

\subsubsection{Lisse sheaves} When $X\to S$ is a representable morphism, we define the presheaf of categories $\mathrm{QLisse}(X/S)$ of lisse sheaves on $X$ relative to $S$ as follows. Consider the presheaf $S_{et}^\mathrm{op}\to \mathrm{Pr}^L$ given by sending $T\to S$ to the category $\mathrm{QLisse}(X\times_S T)$, where $\mathrm{QLisse}(-)$ is defined as in \cite{arinkin2020stack} in the ind-constructible Betti sheaf context.

\subsubsection{Ind-constructible Betti sheaves} Assume that $S$ is complex analytic. In this case, let $\mathrm{Sh}^\mathrm{con}(-)$ denote the sheaf context of ind-constructible Betti sheaves, where constructible means along a locally finite stratification by complex analytic manifolds. Define a functor $S_{et}^\mathrm{op}\to \mathrm{Pr}^L_e$ by the assignment
\[(T\to S)\mapsto \mathrm{Sh}^\mathrm{con}(\mathcal{X}\times_ST).\]
with $*$-pullback functoriality. This is again a presheaf of presentable categories. For a morphism of stacks $f:\mathcal{Y}\to \mathcal{X}$ over $S$, we obtain morphisms of presheaves $f^*:\mathrm{Sh}^\mathrm{con}(\mathcal{X}/S)\to \mathrm{Sh}^\mathrm{con} (\mathcal{Y}/S)$ and $f_!:\mathrm{Sh}^\mathrm{con}(\mathcal{Y}/S)\to \mathrm{Sh}^\mathrm{con}(\mathcal{X}/S)$.

In the context of all Betti sheaves, we similarly define the presheaf of categories $\mathrm{Sh}_\mathrm{lc}(X/S)$ given by sending $T\to S$ to the category of locally constant sheaves on $X\times_S T$. There are two maps of sheaves of categories $\mathrm{QLisse}(X/S)\to \mathrm{Sh}^\mathrm{con}(X/S)$ and $\mathrm{QLisse}(X/S)\to \mathrm{Sh}_{\mathrm{lc}}(X/S)$. These two functors intertwine the relevant functoriality in $X$ and $S$.

\subsection{Moduli spaces of local systems in families}\label{LocSys in families} In this subsection, we will define the family version of the moduli space of $\check{G}$-local systems $\mathrm{LS}_{\check{G}}(X/S)$, which is a cosheaf on $S_{et}$ with values in derived stacks. This cosheaf will play the role of the spectral side in the family version of the Betti geometric Langlands correspondence.

\subsubsection{The Betti context}
    Let $\mathrm{LS}_{\check{G}}(Y)$ be the derived stack over $e$ with $T$-points given by the space of right t-exact symmetric monoidal functors
\[\mathrm{Rep}(\check{G})\to \mathrm{Sh}_{\mathrm{lc}}(Y) \otimes \mathrm{QCoh}(T).\]

This definition also applies in the context of ind-constructible sheaves, in which case we replace the category $\mathrm{Sh}_\mathrm{lc}(Y)$ with $\mathrm{QLisse}(Y)$. The resulting stack in the ind-constructible setting is denoted $\mathrm{LS}^\mathrm{restr}_{\check{G}}(Y)$.

\begin{construction}
    Define a cosheaf of stacks $\mathrm{LS}_{\check{G}}(X/S)$ on $S_{et}$ by defining it on the basis of contractible analytic space over $S$. To each contractible space $S'\to S$ the cosheaf $\mathrm{LS}_{\check{G}}$ associates the stack $\mathrm{LS}_{\check{G}}(X\times_S S')$ of $\check{G}$-local systems on the total space $X\times_S S'$. 
    Since $X\to S$ is smooth and proper, for any map between contractible space $i:S''\to S'$ over $S$, the pullback map $\mathrm{Sh}_{\mathrm{lc}}(X\times _SS')\to \mathrm{Sh}_{\mathrm{lc}}(X\times_S S'')$ is an equivalence of categories. Thus it induces an equivalence of stacks $i^*:\mathrm{LS}_{\check{G}}(X\times_S S')\to \mathrm{LS}_{\check{G}}(X\times_SS'')$. The cosheaf $\mathrm{LS}_{\check{G}}(X/S)$ associates to $i:S''\to S'$ the restriction map given by the equivalence $(i^*)^{-1}:\mathrm{LS}_{\check{G}}(X\times_SS'')\to \mathrm{LS}_{\check{G}}(X\times_SS')$.
\end{construction}
    
The analogous construction makes sense in the context of ind-constructible sheaves, we denote the resulting cosheaf of stacks by $\mathrm{LS}_{\check{G}}^\mathrm{restr}(X/S)$.

For each contractible analytic space $S'$ over $S$, the universal $\check{G}$-local system on $\mathrm{LS}_{\check{G}}(X\times_SS')$ defines a functor
\[\mathrm{Rep}(\check{G})\to \mathrm{Sh}_{\mathrm{lc}}(X\times_SS') \otimes \mathrm{QCoh}(\mathrm{LS}_{\check{G}}(X\times_SS')).\]
For any map of contractible spaces $S''\to S'$ over $S$, we obtain a commutative square
\[
\begin{tikzcd}
\mathrm{Rep}(\check{G}) \arrow[d, "\mathrm{id}"] \arrow[r] & \mathrm{Sh}_{\mathrm{lc}}(X\times_SS') \otimes \mathrm{QCoh}(\mathrm{LS}_{\check{G}}(X\times_SS')) \arrow[d] \\
\mathrm{Rep}(\check{G}) \arrow[r]                          & \mathrm{Sh}_{\mathrm{lc}}(X\times_SS'') \otimes \mathrm{QCoh}(\mathrm{LS}_{\check{G}}(X\times_SS''))            
\end{tikzcd}\]
Thus the universal $\check{G}$-local systems on contractible test objects define a map of sheaves of categories on $S_\mathrm{et}$
\[\mathrm{Rep}(\check{G})\to \mathrm{Sh}_{\mathrm{lc}}(X/S) \otimes \mathrm{QCoh}(\mathrm{LS}_{\check{G}}(X/S)),\]
where $\mathrm{Rep}(\check{G})$ is considered as a constant sheaf and the tensor product is as sheaves of categories.

\subsection{Non-characteristicness}\label{A:non-char} We will use a few different notions of non-characteristicness throughout the paper. In this section we give the definitions and relevant relations between them.

We denote the full subcategory of non-characteristic sheaves by $\mathrm{Sh}^\mathrm{nc}(\mathcal{X})$.

\subsubsection{Betti non-characteristicness} For a morphism $f:\mathcal{X}\to S$ of smooth Artin stacks, a sheaf $\mathcal{F}\in \mathrm{Sh}(\mathcal{X})$ is said to be $f$-non-characteristic if the singular support $\mathrm{SS}(\mathcal{F})\subset T^*\mathcal{X}$ is non-characteristic relative to $f$.

\subsubsection{Ind-constructible non-characteristicness} Suppose $S$ is a complex ananlytic manifold. In the ind-constructible sheaf context we have two different notions of non-characteristicness.

\begin{defin}
    Let $f:\mathcal{X}\to S$ be a morphism of smooth Artin stacks. An ind-constructible sheaf $\mathcal{F}\in \mathrm{Sh}^\mathrm{con}(\mathcal{X})$ is said to be non-characteristic if every perverse cohomology sheaf $\mathcal{H}^p\mathcal{F}$ is a filtered colimit of non-characteristic constructible sheaves.
\end{defin}

In this definition, by a non-characteristic constructible sheaf we mean a constructible sheaf such that $SS(\mathcal{F})$ is non-characteristic. The stronger notion of non-characteristicness in the ind-constructible setting is the notion of ind-non-characteristicness. We denote the category of ind-non-characteristic sheaves by $\mathrm{Sh}^{\mathrm{con,nc}}(\mathcal{X})$.

\begin{defin}
   Let $f:\mathcal{X}\to S$ be a morphism of smooth Artin stacks. An ind-constructible sheaf $\mathcal{F}\in \mathrm{Sh}^\mathrm{con}(\mathcal{X})$ is said to be ind-non-characteristic, if $\mathcal{F}$ is a filtered colimit of non-characteristic constructible sheaves.
\end{defin}

It is clear that ind-non-characteristicness implies non-characteristicness. The difference between the two happens at cohomological degree negative infinity: any non-characteristic ind-constructible sheaf that is bounded from below is ind-non-characteristic.

The notion of non-characteristic sheaves is closed under all colimits.

\begin{lem}\label{colimits of ind-nonchar}
    Let $\mathcal{X}\to S$ be a smooth Artin stack over $S$. The full subcategory of $S$-non-characteristic ind-constructible sheaves on $\mathcal{X}$ is closed under colimits. 
\end{lem}
\begin{proof}
    The statement is smooth-local on $\mathcal{X}$ and thus we may assume $\mathcal{X}$ is a smooth variety over $S$. Let $\mathcal{F}$ be an ind-constructible sheaf on $\mathcal{X}$ and suppose $\mathcal{F}\simeq \colim\mathcal{F}_i$ for $S$-non-characteristic ind-constructible sheaves $\mathcal{F}_i$. Let $\mathcal{H}_n^p\mathcal{F}$ denote the $n^{th}$ perverse homology sheaf of $\mathcal{F}$. It suffices to prove that each perverse subsheaf $\mathcal{G}$ of $\mathcal{H}_n^p\mathcal{F}$ is $S$-non-characteristic. Since
    \[\mathcal{H}_n\mathcal{F}\simeq \tau_{\ge n} \colim \tau_{\le n}\mathcal{F}_i [-n]\]
    and $\mathcal{G}$ is compact, it follows that $\mathcal{G}\to \mathcal{F}$ factors through $\mathcal{G}\to \tau_{\le n} \mathcal{F}_i[-n]$, for some $i\in I$. Since $\mathcal{G}$ is $0$-connected, the map $\mathcal{G}\to \tau_{\le n} \mathcal{F}_i[-n]$ factors through a map $\mathcal{G}\to \tau_{\ge n}\tau_{\le n} \mathcal{F}_i[-n]\simeq \mathcal{H}_n\mathcal{F}_i$. Since the composition $\mathcal{G}\to \mathcal{H}_n\mathcal{F}_i\to \mathcal{H}_n\mathcal{F}$ is the inclusion of a perverse subsheaf, it follows that $\mathcal{G}\to \mathcal{H}_n\mathcal{F}_i$ is also the inclusion of a perverse subsheaf. It follows that $SS(\mathcal{G})\subseteq SS(\mathcal{F}_i)$ and hence $\mathcal{G}$ is $S$-non-characteristic since $\mathcal{F}_i$ is. 
\end{proof}

\newpage

\bibliographystyle{alpha}
\bibliography{Bip}

\end{document}